\documentclass[a4paper, 11pt]{amsart}

\usepackage{enumitem, amsmath, amssymb, amsthm, mathtools, caption, booktabs, array, multirow, graphicx, hyperref, color}
\usepackage[toc]{appendix}

\newcommand{\lb}{\left(}
\newcommand{\rb}{\right)}
\newcommand{\eps}{\varepsilon}
\newcommand{\R}{\mathbb{R}}
\newcommand{\C}{\mathbb{C}}
\newcommand{\Z}{\mathbb{Z}}
\newcommand{\F}{\mathbb{F}}
\renewcommand{\P}{\mathbb{P}}
\newcommand{\pd}{\partial}
\newcommand{\diff}{\mathrm{d}}
\newcommand{\Diff}{\mathrm{Diff}}
\renewcommand{\phi}{\varphi}
\newcommand{\ip}[2]{\langle #1, #2 \rangle}
\newcommand{\lspan}[1]{\langle #1 \rangle}
\newcommand{\mopnsi}[2]{\widetilde{\mathcal{M}}{}_{0, #1} (J, #2)}
\newcommand{\mupnsi}[2]{\mathcal{M}{}_{0, #1} (J, #2)}
\newcommand{\norm}[1]{\lVert #1 \rVert}
\renewcommand{\tilde}{\widetilde}
\newcommand{\conj}[1]{\overline{#1}}
\newcommand{\ev}{\operatorname{ev}}
\newcommand{\SL}{\mathrm{SL}}
\newcommand{\SU}{\mathrm{SU}}
\newcommand{\tri}{\triangle}
\newcommand{\double}{\tilde}
\newcommand{\D}{\mathcal{D}}
\newcommand{\CO}{\operatorname{\mathcal{CO}^0}}
\newcommand{\Res}{\operatorname{Res}}
\newcommand{\codim}{\operatorname{codim}}

\newcommand{\im}{\operatorname{Im}}
\newcommand{\re}{\operatorname{Re}}
\renewcommand{\ker}{\operatorname{ker}}
\newcommand{\coker}{\operatorname{coker}}
\newcommand{\End}{\operatorname{End}}
\newcommand{\adj}{\operatorname{adj}}
\newcommand{\Sym}{\operatorname{Sym}}
\newcommand{\qcup}{\mathbin{*}}
\newcommand{\Char}{\operatorname{char}}
\newcommand{\PD}{\operatorname{PD}}
\newcommand{\cosec}{\operatorname{cosec}}
\newcommand{\mc}{\colon}
\newcommand{\id}{\operatorname{id}}
\newcommand{\rightiso}{\xrightarrow{\sim}}
\newcommand{\mtw}{\mathcal{M}_2}
\newcommand{\begen}{\begin{enumerate}[label=(\roman*), align=right]}
\newcommand{\crit}{\operatorname{Crit}}
\newcommand{\trans}[2]{\mathbin{{}_{#1} \mathbin{\pitchfork}_{#2}}}

\theoremstyle{plain}
\newtheorem{thm}{Theorem}[section]
\newtheorem{lem}[thm]{Lemma}
\newtheorem{prop}[thm]{Proposition}
\newtheorem{cor}[thm]{Corollary}

\theoremstyle{remark}

\newtheorem{ex}[thm]{Example}

\theoremstyle{definition}
\newtheorem{defn}[thm]{Definition}

\theoremstyle{plain}
\newtheorem{mthm}{Theorem}

\theoremstyle{plain}
\newtheorem{athm}{Theorem}[section]
\newtheorem{alem}[athm]{Lemma}

\theoremstyle{remark}

\theoremstyle{definition}

\address{Centre for Mathematical Sciences\\University of Cambridge\\CB3 0WB\\UK}
\email{j.smith@dpmms.cam.ac.uk}

\begin{document}

\title{Floer cohomology of Platonic Lagrangians}
\author{Jack Smith}

\begin{abstract}  We analyse holomorphic discs on Lagrangian $\SU(2)$-orbits in a family of quasihomogeneous threefolds of $\SL(2, \C)$, previously studied by Evans--Lekili, introducing several techniques that should be applicable to wider classes of homogeneous Lagrangians.  By studying the closed--open map we place strong restrictions on the self-Floer cohomology of these Lagrangians, which we then compute using the Biran--Cornea pearl complex.
\end{abstract}

\maketitle
\tableofcontents

\section{Introduction}
\label{secIntro}

\subsection{Background}
\label{sscBackg}

In \cite{EL1}, Evans and Lekili initiated the study of \emph{homogeneous Lagrangian submanifolds} of K\"ahler manifolds, that is Lagrangian submanifolds which are the orbit of a Lie group action on the ambient manifold by holomorphic symplectomorphisms.  Amongst other things, they showed that the standard integrable complex structure can be used to construct moduli spaces of holomorphic discs, introduced a particularly simple type of disc, which they termed \emph{axial}, and showed that all index $2$ discs are of this form.  Using the machinery they developed they computed the Floer cohomology of the Chiang Lagrangian in $\C\P^3$ with itself.

Rather surprisingly, working over a field $k$ this cohomology is non-zero if and only if the characteristic of $k$ is $5$.  Evans--Lekili partially explained this using the Auroux--Kontsevich--Seidel criterion (see \protect \MakeUppercase {P}roposition\nobreakspace \ref {labAKS}) for eigenvalues of quantum multiplication by the first Chern class, although this argument also leaves open the possibility of the cohomology being non-zero in characteristic $7$.  It is natural to ask whether there is a simple way in which one can rule this out.

The Chiang Lagrangian is the first in a family of four `Platonic' Lagrangian $\SU(2)$-orbits inside quasihomogeneous Fano threefolds of $\SL(2, \C)$, and one can also ask what the self-Floer cohomology of the other three Lagrangians is.  The aim of the present paper is to address these two questions, with a view towards developing a more general understanding of the Floer theory of homogeneous Lagrangians.

\subsection{Outline of the paper}
\label{sscPaperOutline}

We begin in Section\nobreakspace \ref {secHomLag} by studying holomorphic discs in a complex manifold $X$ whose boundaries lie on a totally real submanifold $L$ (by which we mean a submanifold $L$ such that for all $p \in L$ we have $T_pX = T_pL \oplus J \cdot T_pL$, where $J$ is the complex structure on $X$) which is homogeneous with respect to some group action, with the aim of applying these results when $X$ is K\"ahler and $L$ Lagrangian.  This largely follows \cite{EL1}, reviewing various definitions and slightly simplifying and generalising Evans--Lekili's result that index $2$ discs are axial.

We then specialise to the case of the Platonic Lagrangians: a family of Lagrangian $\SU (2)$-orbits $L_C$ in a sequence of four Fano threefolds $X_C$, parametrised by configurations $C$ of points on the sphere ($C$ can be a triangle $\tri$, tetrahedron $T$, octahedron $O$, or icosahedron $I$, and the respective threefolds are $\C\P^3$, the quadric, the threefold known as $V_5$, and the Mukai--Umemura threefold $V_{22}$); Section\nobreakspace \ref {secChiangFamily} reviews the construction of these objects and sets out their basic properties.  Each $X_C$ carries a holomorphic action of $\SL(2, \C)$, complexifying the $\SU(2)$-action, with dense Zariski open orbit $W_C$ and compactification divisor $Y_C = X_C \setminus W_C$.

The main content of the paper is contained in Section\nobreakspace \ref {secDiscAn}, where we introduce several new ideas for analysing holomorphic discs bounded by these Lagrangians.  First we define an antiholomorphic involution $\tau$ on the dense open orbit $W_C$, built from exponentiating complex conjugation on the Lie algebra $\mathfrak{sl}(2, \C) \cong \mathfrak{su}(2) \otimes \C$, which extends across the compactification divisor $Y_C$ when $C$ is $O$ or $I$.  When $C$ is $\tri$ or $T$, although $\tau$ itself cannot be defined globally, we can still use it to reflect holomorphic discs.  By gluing discs to their reflections were are able to reduce problems involving open holomorphic curves (discs) to closed curves (spheres), and hence employ tools from algebraic geometry.

We then, in analogy with the study of meromorphic functions on Riemann surfaces, define the notion of a \emph{pole} of a holomorphic curve in $X_C$---essentially this is a point where the curve hits $Y_C$---and prove various properties.  In particular, we recover the result that all index $2$ discs are axial for this family by independent methods.  The guiding principle is that just as a meromorphic function on a compact Riemann surface is defined up to the addition of a constant by the positions and principal parts of its poles, a disc should---roughly speaking---be determined up to the action of $\SU(2)$ by the positions of its poles and some local data at these points (although in reality there are global complications arising from monodromy around poles).  The poles of a disc determine the degree of the rational curve obtained by gluing it to its reflection, and controlling this degree is crucial later in enumerating the index $4$ discs.

Next we show, by considering discs hitting a (complex) $1$-dimensional orbit $N_C \subset Y_C$, that a large part of the closed--open map can be computed using just axial discs.  From this we build an eigenvalue constraint analogous to that of Auroux--Kontsevich--Seidel, and prove:
\begin{mthm}[\protect \MakeUppercase {C}orollary\nobreakspace \ref {labCO}\ref{COitm2}, \protect \MakeUppercase {C}orollary\nobreakspace \ref {labCOLO}, \protect \MakeUppercase {P}roposition\nobreakspace \ref {RuleOutChars}] \label{labPrimes} If $HF^*(L_C, L_C; k)$ is non-zero over a field $k$ of characteristic $p$, then $p$ must be $5$, $2$, $2$ or $2$ for $C$ equal to $\tri$, $T$, $O$ or $I$ respectively.
\end{mthm}

The result for the octahedron actually relies on an orientation computation using the explicit calculation of $HF^*(L_\tri, L_\tri; \Z)$ later in the paper.  For the icosahedron a certain bad bubbled configuration can occur and spoil the count of discs meeting $N_I$, so our constraint reduces just to the Auroux--Kontsevich--Seidel criterion itself, but it can be strengthened using a trick based on the antiholomorphic involution and a change of relative spin structure---this is the content of \protect \MakeUppercase {P}roposition\nobreakspace \ref {RuleOutChars}.

The significance of characteristic $2$ for the octahedron and icosahedron is natural given the existence of the global antiholomorphic involution fixing the Lagrangian in these cases: the quantum corrections in the pearl complex cancel with their reflections modulo $2$.  The characteristics are less clear for the triangle and tetrahedron.  In particular, the fact that $2$ occurs again for the tetrahedron, making it appear to fall into the same pattern as the octahedron and icosahedron, with the triangle as the lone exceptional case, seems to be a numerical coincidence arising from the fact that the numbers involved in the eigenvalue constraints are fairly small.  It also seems to be a coincidence that there is exactly one possible prime in each case.

Although we only develop these techniques (the involution, pole analysis, and constraints on the closed--open map) in the context of the Platonic Lagrangians in this paper, many of the ideas can be applied more widely, to other families of homogeneous Lagrangians.  This is the subject of work in progress by the present author.  See for instance \cite{SmDCS}, where the closed--open map computation is generalised and combined with the study of certain discrete symmetries in order to calculate the self-Floer cohomology of a family of $\mathrm{PSU}(N-1)$-homogeneous Lagrangians in $(\C\P^{N-2})^N$ and some related examples.

In Section\nobreakspace \ref {secMorsePearl} we return to the Lagrangians themselves and construct (as much as is necessary) Heegaard splittings and Morse functions.  This allows us to calculate everything in the pearl complex associated to the Lagrangians (see Section 2 of \cite{BCLQH} for the definitions and Section 3.6 for the identification of the resulting (co)homology with Floer (co)homology) except the index $4$ contributions.  Then in Section\nobreakspace \ref {secCompHF} we combine this knowledge with our understanding of the closed--open map to place strong constraints on the self-Floer cohomology, compute the required index $4$ counts, and deduce:
\begin{mthm}[\protect \MakeUppercase {P}roposition\nobreakspace \ref {labHFLOIk}, \protect \MakeUppercase {P}roposition\nobreakspace \ref {labCOConstr2}, \protect \MakeUppercase {C}orollary\nobreakspace \ref {labTriHF}, \protect \MakeUppercase {C}or\-ollary\nobreakspace \ref {labTHF}, \protect \MakeUppercase {C}orollary\nobreakspace \ref {labIHF}] \label{labHF}  Fix an orientation and spin structure on each Lagrangian $L_C$.  Working over a field $k$ of characteristic $5$, $2$, $2$ and $2$ in the four cases respectively, the Floer cohomology groups are given as $\Z/2$-graded $k$-vector spaces by
\begin{align*}
HF^0(L_\tri, L_\tri; k) &\cong HF^1(L_\tri, L_\tri; k) \cong k
\\ HF^0(L_T, L_T; k) &\cong  HF^1(L_T, L_T; k) \cong k
\\ HF^0(L_O, L_O; k) &\cong HF^1(L_O, L_O; k) \cong k^2
\\ HF^0(L_I, L_I; k) &\cong HF^1(L_I, L_I; k) \cong k.
\end{align*}

Working over $\Z$, the $\Z/2$-graded Floer cohomology rings are concentrated in degree $0$ with
\begin{align*}
HF^0(L_\tri, L_\tri; \Z) &\cong \Z/(5)
\\ HF^0(L_T, L_T; \Z) &\cong \Z/(4)
\\ HF^0(L_O, L_O; \Z) &\cong \Z[x]/(2, x^2+x+1)
\\ HF^0(L_I, L_I; \Z) &\cong \Z/(8).
\end{align*}
\end{mthm}

The results for $L_\tri$ were proved by Evans--Lekili, but the others are new.  By far the hardest part is computing $HF^*(L_I, L_I)$ over $\Z$---if one is only interested in working over fields then the rather involved calculations of Appendix\nobreakspace \ref {secInd4Icos} can be avoided.  In each case the Lagrangian is wide over fields of the special characteristic, meaning that its self-Floer cohomology has the same rank as its classical cohomology, whilst the Floer cohomology over $\Z$ is as big as is allowed by the restrictions we derive from the closed--open map.  Note that $HF^0(L_O, L_O; \Z)$ is the field $\F_4$ of four elements.

Evans and Lekili remark \cite[Corollary B]{EL1} that their results imply that $L_\tri$ is not Hamiltonian-displaceable from itself or from the standard Clifford torus in $\C\P^3$ (recent work by Konstantinov \cite[Corollary 1.2]{MomchilLocSys} using higher rank local systems shows that it is also non-displaceable from the standard $\R\P^3$).  Similarly the fact that $L_T$, $L_O$ and $L_I$ are Floer cohomologically non-trivial, with appropriate coefficients, immediately shows that they are also non-displaceable from themselves.  In fact, in their subsequent paper \cite[Section 7.1]{EL2} Evans--Lekili showed that the real locus of the quadric $X_T$, which is a monotone Lagrangian sphere (homogeneous for a different $\SU(2)$-action on $X_T$), split-generates the Fukaya category over any field $k$ of characteristic $2$, so $L_T$ is not displaceable from this sphere either (note that the ring $QH^*(X_T; k)$ is isomorphic to $k[E]/(E^4)$, so already every element is invertible or nilpotent; in particular, the whole Fukaya category forms one of their summands $D^\pi \mathcal{F}(X_T; k)_0$).  By \cite[Corollary 6.2.6]{EL2} we actually deduce that $L_T$ also split-generates the Fukaya category of the quadric over $k$.

The paper concludes with three appendices, which contain technical discussions which would otherwise distract from the main thread of the computation.  The first establishes transversality for the pearl complex in our setting, the second describes the analysis of index $4$ discs on $L_I$, whilst the third collects together explicit coordinate expressions for various configurations of points on the sphere.

\subsection{Acknowledgements}
\label{sscAck}

First and foremost I am extremely grateful to my supervisor, Ivan Smith, for constant encouragement, guidance and suggestions, for proposing this project to begin with, and for feedback on earlier versions of this paper.  I am also indebted to Dmitry Tonkonog for many useful discussions (in particular for pointing me towards the paper of Haug), to Benjamin Barrett, Jonny Evans, Luis Haug, Momchil Konstantinov and Yank{\i}  Lekili for helpful conversations, and to the anonymous referee who proposed a large number of corrections and improvements.  Wolfram Mathematica was invaluable for algebraic manipulation and experimentation, especially in Appendix\nobreakspace \ref {secInd4Icos}.  A Mathematica notebook containing code for verifying various calculations in the paper is available at \href{http://arxiv.org/abs/1510.08031}{{\ttfamily arXiv:1510.08031}}.  This work was funded by EPSRC.

\section{Homogeneous totally real submanifolds}
\label{secHomLag}

\subsection{Preliminaries}
\label{sscPrelim}

We begin with the following definition, which differs slightly from that given in \cite{EL1}:

\begin{defn} \label{labHomLag} If $X$ is a complex manifold carrying an action of a compact Lie group $K$ by holomorphic automorphisms, and $L$ is a totally real submanifold which is an orbit of the $K$-action, then we say $( X, L )$ is \emph{$K$-homogeneous}.
\end{defn}

Given a complex manifold $X$ with complex structure $J$, and a totally real submanifold $L$, the Maslov index homomorphism $\mu \mc \pi_2(X, L) \rightarrow \Z$ is constructed as follows.  For a continuous map $u \mc (D, \pd D) \rightarrow (X, L)$, where $D$ denotes the closed unit disc $\{ z \in \C : | z | \leq 1 \}$, and $\pd D$ its boundary, consider the complex vector bundle $u^* TX$ over $D$.  This bundle can be trivialised, and in this trivialisation the subbundle $u|_{\pd D}^* TL$ is represented by a map $B \mc \pd D \rightarrow \mathrm{GL}(n, \C)/\mathrm{GL}(n, \R)$, where $n=\dim L=\dim_\C X$.  Then $(\det B)^2 / |\det B|^2$ defines a continuous map $\pd D \rightarrow \pd D$, and we set $\mu(u)$ to be its winding number.

This number is independent of all choices made, and is invariant under homotopies of $u$ relative to its boundary.  If $L$ is orientable then $B$ lifts to a well-defined map $B_+$ to $\mathrm{GL}(n, \C)/ \mathrm{GL}_+(n, \R)$, where $\mathrm{GL}_+$ denotes those matrices with positive determinant.  Then $(\det B_+) /|\det B_+|$ is well-defined, and the Maslov index is twice its winding number, so is even.  In fact $\mu$ is really given by pairing with the \emph{Maslov class} in $H^2(X, L; \Z)$, which we also denote by $\mu$ and which restricts to $2c_1(X)$ in $H^2(X; \Z)$.

For a non-zero class $A \in H_2 ( X, L)$, define the moduli space of $k$-times-marked, parametrised \mbox{($J$-)}holomorphic discs in class $A$ to be
\begin{multline*}
\mopnsi{k}{A} = \{ (u, z_1, \dots, z_k) : u \mc (D, \pd D) \rightarrow (X, L) \text{ holomorphic,}
\\ [u]=A \text{, and } z_1, \dots, z_k \in \pd D \text{ distinct} \}.
\end{multline*}
The virtual dimension of this moduli space is $\dim L + \mu ( A ) + k$.  We will usually drop the subscript $0$ (representing the genus of the curve) and the $J$ from the notation.  Let the corresponding moduli space of unparametrised discs be
\[
\mathcal{M}_k(A) = \mupnsi{k}{A} \coloneqq \mopnsi{k}{A}/\mathrm{PSL}(2, \R),
\]
where $\phi \in \mathrm{PSL}(2, \R)$ acts via $\phi \cdot (u, z_1, \dots, z_n) = (u \circ \phi^{-1}, \phi(z_1), \dots, \phi(z_2))$, of virtual dimension $\dim L+\mu(A) + k - 3$ \cite[Theorem 5.3]{Oh}.

Evans and Lekili \cite[Lemma 3.2]{EL1} made the following crucial observation:

\begin{lem} \label{labDiscReg} If $(X, L)$ is $K$-homogeneous then every holomorphic disc
\[
u \mc (D, \pd D) \rightarrow (X, L)
\]
is regular, and hence all of the above moduli spaces are smooth manifolds of the expected dimension.
\end{lem}
Their proof actually shows that all partial indices of such discs are non-negative (see Section\nobreakspace \ref {sscParInd} for the definition of partial indices, where we also review this argument), which will be used to establish various transversality results later.

\subsection{Axial discs}
\label{sscAxDiscs}

Again following Evans--Lekili, we next define the notion of an axial disc:

\begin{defn} If $(X, L)$ is $K$-homogeneous, $u \mc (D, \pd D) \rightarrow (X, L)$ is a holomorphic disc, and there exists a smooth group homomorphism $R \mc \R \rightarrow K$ such that (possibly after reparametrising $u$) we have $u(e^{i\theta}z)=R(\theta)u(z)$ for all $z \in D$ and all $\theta \in \R$, then we say $u$ is \emph{axial}.
\end{defn}

We will frequently make use of Lie groups, Lie algebras and their actions so let us briefly fix notation.  The Lie algebra of the compact group $K$ will be denoted by $\mathfrak{k}$ (Fraktur k).  More generally, Lie groups will be denoted in uppercase (for example $\mathrm{GL}(n, \C)$ or $G$) whilst the corresponding Lie algebras will be denoted by the same names but in lowercase Fraktur (e.g.~$\mathfrak{gl}(n, \C)$ or $\mathfrak{g}$ respectively).  The exponential map from a Lie algebra to the corresponding Lie group will be denoted by $e^\cdot$, whilst if a Lie group $G$ acts on a manifold $M$ the infinitesimal action of a Lie algebra element $\xi \in \mathfrak{g}$ on a point $p \in M$, meaning
\[
\left.\frac{\diff}{\diff t}\right|_{t=0} e^{t\xi}p,
\]
will usually be denoted by $\xi \cdot p$.  We will sometimes also use $g \cdot p$ to denote the action of a group element $g \in G$ on $p$, although we will often just write $gp$.

Identifying the upper half-plane (with infinity adjoined) with $D$ via the M\"obius map $z \mapsto (iz+1)/(z+i)$ sending $0$, $1$ and $\infty$ to $-i$, $1$ and  $i$ respectively, we get an identification of the group of holomorphic automorphisms of $D$ with the subgroup $\mathrm{PSL}(2, \R)$ of the group of all M\"obius maps.  We view the Lie algebra $\mathfrak{psl}(2, \R) \cong \mathfrak{sl}(2, \R)$ as sitting inside the algebra $\mathrm{Mat}_{2 \times 2} (\C)$ of $2 \times 2$ complex matrices.  Under these identifications, the rotation $z \mapsto e^{i\theta}z$ of $D$ is generated by the matrix
\[
\rho \coloneqq \begin{pmatrix} 0 & \frac{1}{2} \\ -\frac{1}{2} & 0 \end{pmatrix}.
\]

Before proving the main results of this subsection (\protect \MakeUppercase {L}emma\nobreakspace \ref {labInd2Ax} and \protect \MakeUppercase {L}emma\nobreakspace \ref {labInd4Ax}), we first need a straightforward result about $\mathfrak{psl}(2, \R)$:

\begin{lem} \label{labLieAlg} For $\eta \in \mathfrak{psl}(2, \R) \leq \mathrm{Mat}_{2 \times 2}(\C)$ the following are equivalent:
\begen
\item \label{sl2itm1} $\eta$ acts on $\pd D$ without fixed points, i.e.~$\eta \cdot z \neq 0$ for all $z \in \pd D$.
\item \label{sl2itm2} $\det \eta > 0$.
\item \label{sl2itm3} Some real multiple of $\eta$ is conjugate by an element of $\SL(2, \R)$ to $\rho$.
\end{enumerate}
\end{lem}
\begin{proof} \ref{sl2itm1}$\implies$\ref{sl2itm2}: Assuming $\eta$ acts without fixed points, by compactness of $\pd D$ we can pick $\eps > 0$ such that $\norm{\eta \cdot z} \geq \eps$ for all $z \in \pd D$ (using the standard metric on $\pd D$).  This means that as $t$ increases from $0$ the point $e^{t\eta} \cdot 1$ moves around the unit circle at speed at least $\eps$, so at some time $T \in (0, 2\pi/\eps]$ it returns to its starting point.  In other words, there exists $T \in (0, 2 \pi / \eps ]$ such that $e^{T \eta} \cdot 1 = 1$.  An explicit calculation of $e^{T \eta} \cdot 1$ shows that $\det \eta > 0$.

\ref{sl2itm2}$\implies$\ref{sl2itm3}: Given $\eta \in \mathfrak{psl}(2, \R)$ with $\det \eta > 0$, scale $\eta$ to make its determinant $\frac{1}{4}$.  Then $\eta$ and $\rho$ both have eigenvalues $\pm i/2$ (they are both trace-free), so are conjugate over $\C$.  It is well-known that real matrices conjugate over $\C$ are conjugate over $\R$, so we have $\eta = g \rho g^{-1}$ for some $g \in \mathrm{GL}(n, \R)$.  Replacing $g$ by $g \cdot \lb \begin{smallmatrix} \!\!-1 & 0 \\ 0 & 1 \end{smallmatrix} \rb$ and reversing the sign of $\eta$, if necessary, we may assume that $\det g > 0$.  Dividing $g$ by the square root of its determinant then ensures $g \in \SL(2, \R)$ as required.

\ref{sl2itm3}$\implies$\ref{sl2itm1}: This is immediate from the fact that $\rho$ acts on $\pd D$ without fixed points.
\end{proof}

We can now prove a slightly stronger version of \cite[Corollary 3.10]{EL1}:

\begin{lem} \label{labInd2Ax} If $(X, L)$ is $K$-homogeneous and $u \mc (D, \pd D) \rightarrow (X, L)$ is a holomorphic disc of Maslov index $2$ then $u$ is axial.
\end{lem}
\begin{proof} Let $A = [u]$ and $n=\dim L$.  By \protect \MakeUppercase {L}emma\nobreakspace \ref {labDiscReg} we have that the space $M \coloneqq \widetilde{\mathcal{M}}_0(A)$ of unmarked parametrised holomorphic discs in class $A$ is a smooth manifold of dimension $n+2$.  The tangent space $T_u M$ consists of smooth sections of $u^* TX$, holomorphic over the interior of $D$, which lie in $u|_{\pd D}^* TL$ when restricted to $\pd D$.  For $z \in \pd D$, let $E_z$ denote the `evaluate at $z$' map $T_uM \rightarrow T_{u(z)}L$.

The group $K$ acts smoothly on $M$ on the left by post-composition with the action on $X$ (i.e.~for an element $k \in K$ and a disc $v \in M$ we define the disc $k \cdot v$ by $(k \cdot v)(z) = k \cdot  v (z)$ for all $z \in D$), whilst $\mathrm{PSL}(2, \R)$ acts smoothly on the right by reparametrisation.  For brevity let $\mathfrak{h}$ denote $\mathfrak{psl}(2, \R)$, and let $\psi \mc \mathfrak{h} \rightarrow T_u M$ denote the infinitesimal reparametrisation action at $u$.

Since $\mathfrak{k} \cdot p = T_p L$ for all $p \in L$ (by homogeneity), we see that for each $z \in \pd D$ the map $E_z$ is surjective when restricted to $\mathfrak{k} \cdot u \leq T_uM$, so $\dim ((\mathfrak{k}\cdot u) \cap \ker E_z) = \dim \mathfrak{k}\cdot u - n$.  And $\dim u \cdot \mathfrak{h} = 3$, otherwise $u$ must be constant and hence of index zero.  Counting dimensions inside $T_uM$ we see that
\[
\label{labDims}
\dim \big((\mathfrak{k}\cdot u) \cap (u \cdot \mathfrak{h})\big) \geq \dim \mathfrak{k}\cdot u + \dim u \cdot \mathfrak{h} - \dim T_uM \geq \dim \mathfrak{k}\cdot u + 1 - n,
\]
so we deduce that for all $z \in \pd D$
\[
(\mathfrak{k}\cdot u)\cap (u \cdot \mathfrak{h}) \nsubseteq (\mathfrak{k} \cdot u) \cap \ker E_z,
\]
i.e.~$E_z$ cannot vanish on $(\mathfrak{k}\cdot u) \cap (u \cdot \mathfrak{h})$.  Letting $\mathfrak{g} = \psi^{-1} (\mathfrak{k}\cdot u) \leq \mathfrak{h}$ be the space of infinitesimal reparametrisations which act like an element of $\mathfrak{k}$, we conclude that $E_z \circ \psi |_\mathfrak{g} \neq 0$ for all $z \in \pd D$ (and, in particular, $\mathfrak{g} \neq 0$).  In other words, the subspace $\mathfrak{g}$ has no global fixed points when it acts on $\pd D$.

Suppose we can show that $\mathfrak{g}$ contains an individual element $\eta$ which acts without fixed points on $\pd D$, and therefore satisfies the three equivalent conditions in \protect \MakeUppercase {L}emma\nobreakspace \ref {labLieAlg}.  After scaling such an $\eta$ we may assume that we have $\eta = g \rho g^{-1}$ for some $g \in \SL(2, \C)$, and by definition of $\mathfrak{g}$ there exists $\xi \in \mathfrak{k}$ with $\xi \cdot u = u \cdot \eta$.  Reparametrising $u$ by $g$ we then get $\xi \cdot u = u \cdot \rho$ and hence $e^{\theta \xi}u(z) = u(e^{i\theta}z)$ for all $z \in D$ and all $\theta \in \R$ (recalling that the $e^\cdot$ on the left-hand side denotes the exponential map $\mathfrak{k} \rightarrow K$), so $u$ is axial as required.  It therefore remains to show the existence of such an $\eta$.

First note that $\mathfrak{g}$ is a Lie subalgebra of $\mathfrak{h}$.  Indeed, it is the projection to $\mathfrak{h}$ of the subalgebra of $\mathfrak{k} \oplus \mathfrak{h}$ which acts trivially on $u$.  Our problem is thus to show that a subalgebra $\mathfrak{g}$ of $\mathfrak{h} = \mathfrak{psl}(2, \R)$ with no global fixed point on $\pd D$ contains a fixed-point-free element.  This is clear if $\dim \mathfrak{g}$ is $1$ or $3$, so we are left to deal with the case where $\mathfrak{g}$ is two-dimensional, to which we now restrict our attention.

Let $\eta_H$, $\eta_X$, and $\eta_Y$ be the standard basis vectors
\[
\begin{pmatrix} 1 & 0 \\ 0 & -1 \end{pmatrix}\text{, } \begin{pmatrix} 0 & 1 \\ 0 & 0 \end{pmatrix}\text{, and } \begin{pmatrix} 0 & 0 \\ 1 & 0 \end{pmatrix}
\]
of $\mathfrak{h}$, and let $\lspan{\cdot}$ denote linear span.  We cannot have $\mathfrak{g}=\lspan{ \eta_X, \eta_Y }$, since this space is not closed under the Lie bracket, so we can pick a basis for $\mathfrak{g}$ of the form $\eta_H + a \eta_X + b \eta_Y, c \eta_X + d \eta_Y$, with $a, b, c, d \in \R$ and $c$ and $d$ not both zero.  Without loss of generality $c \neq 0$, so we can change basis to be of the form $\eta_H + \kappa \eta_Y, \eta_X + \lambda \eta_Y$, with $\kappa, \lambda \in \R$.  Closure under the Lie bracket forces $\kappa^2 = 4 \lambda$, but then every $\eta \in \mathfrak{g}$ fixes the point $w \coloneqq (\kappa i/2 + 1)/(\kappa/2 + i) \in \pd D$, contradicting our hypothesis.  Therefore $\dim \mathfrak{g} = 2$ is impossible, and the proof is complete.
\end{proof}

Using this we obtain a similarly-modified version of \cite[Corollary 3.11]{EL1}:

\begin{lem} \label{labInd4Ax} Suppose $(X, L)$ is $K$-homogeneous and $Z$ is a complex submanifold of $X \setminus L$ of complex codimension $2$, which is $K$-invariant setwise.  If $u \mc (D, \pd D) \rightarrow (X, L)$ is a holomorphic disc of Maslov index $4$ which intersects $Z$ cleanly in a single point then $u$ is axial.
\end{lem}
\begin{proof}  $K$ acts on the normal bundle of $Z$ in $X$, and taking the projectivisation we obtain an action on the exceptional divisor of the blowup $\tilde{X}$ of $X$ along $Z$.  This action extends to the whole of the complex manifold $\tilde{X}$, and the projection $\pi \mc \tilde{X} \rightarrow X$ is $K$-equivariant.  The lift $\tilde{L}$ of $L$ to $\tilde{X}$ is a $K$-homogeneous totally real submanifold, and the proper transform $\tilde{u}$ of $u$ is a holomorphic disc $(D, \pd D) \rightarrow (\tilde{X}, \tilde{L})$ of Maslov index $2$ (under the blowup the index is decreased by $\codim Z - 1$ multiplied by twice the number of intersection points of $u$ and $Z$), so by \protect \MakeUppercase {L}emma\nobreakspace \ref {labInd2Ax} we deduce that $\tilde{u}$ is axial.  This implies that $u$ itself is axial.
\end{proof}

\subsection{The form of axial discs}
\label{sscAxForm}

If our compact Lie group $K$ has a complexification $G$, and the action of $K$ on $X$ extends to an action of $G$ (holomorphic in both the $G$ and $X$ factors), then axial discs have a particularly simple form.  Indeed, if $u \mc (D, \pd D) \rightarrow (X, L)$ is holomorphic and $R \mc \R \rightarrow K$ is a Lie group morphism, such that $u(e^{i \theta} z)=R(\theta)u(z)$ for all $z \in D$ and all $\theta \in \R$, then we claim that
\begin{equation}
\label{eqAxDisc}
u(z)=e^{-i R'(0) \log z} u(1)
\end{equation}
for all non-zero $z \in D$.

Note first that we have $R(\theta)=e^{\theta R'(0)}$ for all $\theta \in \R$, so \eqref{eqAxDisc} holds on $\pd D$.  Moreover, we see that $e^{2 \pi R'(0)}$ fixes $u(1)$, and hence the right-hand side of \eqref{eqAxDisc} is well-defined for all $z \in \C^*$.  Fix a point $p \in \pd D$ and pick vectors $\xi_1, \dots, \xi_n$ in the Lie algebra $\mathfrak{k}$ of $K$ whose infinitesimal actions at $u(p)$ form a basis for $T_{u(p)} L$.  Then the map
\[
(z_1, \dots, z_n) \mapsto e^{\sum z_i \xi_i}u(p)
\]
defines a holomorphic parametrisation of a neighbourhood of $u(p)$, under which $L$ corresponds to $\R^n \subset \C^n$ in coordinate space.  So in our chart both sides of \eqref{eqAxDisc} are given on a neighbourhood of $p$ in $D\setminus \{0\}$ by $n$ continuous functions, holomorphic off $\pd D$, and equal and real on $\pd D$.  The standard Schwarz reflection argument then proves that they agree on the whole neighbourhood of $p$ (or at least the component containing $p$), and hence, by the identity theorem, on all of $D \setminus \{0\}$.

We will frequently want to describe various axial discs later, so it will be convenient to have a shorthand for expressions of the form appearing on the right-hand side of \eqref{eqAxDisc}.  For $\xi \in \mathfrak{k}$ and $p \in X$ satisfying $e^{2\pi \xi}p=p$ we therefore define $u_{\xi, p}$ to be the map
\[
z \mapsto e^{-i \xi \log z}p.
\]
We are being deliberately vague about the domain of definition here.  The obvious choice is $\C^*$, but in our applications the map will in fact extend over $0$ and $\infty$ to give a whole sphere.  Sometimes we will just want the disc (i.e.~the restriction of the sphere to $D$).  Hopefully it will be clear from the context.

Note that if a holomorphic disc $u \mc D \rightarrow X$ is invariant under a finite group of rotations, so that for some positive integer $n$ we have $u(e^{2\pi i/n}z)=u(z)$ for all $z \in D$, then $u$ factors through $z \mapsto z^n$ via some holomorphic map $v \mc D \rightarrow X$.

\section{The Platonic Lagrangians}
\label{secChiangFamily}

\subsection{The Chiang Lagrangian}
\label{sscChiangLag}

Given a finite-dimensional complex inner product space $W$, the symplectic form $\omega$ induced by the metric and complex structure has primitive $1$-form
\[
\lambda = \im(z^\dag \diff z)/2
\]
(meaning $\omega = \diff \lambda$), where $z$ is a vector of coordinates with respect to an orthonormal basis and $\dag$ denotes conjugate transpose.  The unitary group $\mathrm{U}(W)$ clearly preserves this $1$-form, and hence its action on $W$ is Hamiltonian, with moment map $\tilde{\mu} \mc W \rightarrow \mathfrak{u}(W)^*$ given by
\[
\ip{\tilde{\mu}(z)}{\xi} = -X_\xi \lrcorner \lambda_z = -\frac{1}{2} \im z^\dag \xi z = \frac{i}{2} z^\dag \xi z
\]
for all $z \in W$ and all $\xi \in \mathfrak{u}(W)$.  Here $\ip{\cdot}{\cdot}$ denotes the pairing between $\mathfrak{u}(W)^*$ and $\mathfrak{u}(W)$, $X_\xi$ is the vector field generated by the infinitesimal action of $\xi$, and $\lrcorner$ denotes contraction.  Our sign convention is that the moment map satisfies $\omega(X_\xi, \cdot) = \ip{\diff \tilde{\mu}}{\xi}$ for all $z$ and $\xi$.

Now consider the fundamental representation $V$ of $\SU(2)$.  This is (tautologically) unitary with respect to the standard inner product $g$, so all of its tensor powers $V^{\otimes d}$ are also unitary with respect to the corresponding tensor powers $g^{\otimes d}$.  Inside $V^{\otimes d}$ we have the subrepresentation comprising totally symmetric tensors, which is isomorphic to the $d$th symmetric power $S^dV$ of $V$, and we deduce that $S^dV$ is unitary with respect to the restriction of $g^{\otimes d}$.  Fix a basis of $S^dV$ which is orthonormal with respect to this inner product, let $\phi \mc \mathfrak{su}(2) \rightarrow \mathrm{Mat}_{(d+1)\times(d+1)}(\C)$ describe the infinitesimal action in this basis, and let $z$ denote a corresponding coordinate vector.

Taking $W = S^dV$ above, we see that the $\SU(2)$-action on $S^dV$ is Hamiltonian, with moment map $\tilde{\mu} \mc S^dV \rightarrow \mathfrak{su}(2)^*$ defined by
\[
\ip{\tilde{\mu}(z)}{\xi} = \frac{i}{2} z^\dag \phi(\xi) z
\]
for all $z \in S^d V$ and all $\xi \in \mathfrak{su}(2)$.  This action commutes with the diagonal $\mathrm{U}(1)$-action on $S^dV$, and the moment map is $\mathrm{U}(1)$-invariant, so it descends to a Hamiltonian action on the projective space $\P S^dV$ with moment map $\mu$ given by
\begin{equation}
\label{eqMomMap}
\ip{\mu([z])}{\xi} = \frac{i}{2} \frac{z^\dag \phi(\xi) z}{z^\dag z}
\end{equation}
for all $z \in S^dV$, representing $[z] \in \P S^dV$, and all $\xi \in \mathfrak{su}(2)$.  Our convention is that the symplectic form on a projective space is obtained from symplectic reduction of the corresponding vector space at the unit sphere level, so a projective line has area $\pi$.

It is well-known (see, for example, \cite[Proposition 1.5]{Ch}) that an orbit of a Hamiltonian action of a compact Lie group is isotropic if it is contained in the moment map preimage of a fixed point of the coadjoint representation of the group.  In particular, orbits contained in the zero set of the moment map are isotropic.  In \cite{Ch} Chiang considered the case of the above $\SU(2)$-action on $S^dV$ with $d=3$.  In her example the set $\mu^{-1} (0)$ is a single three-dimensional orbit inside $\C\P^3$, and hence is Lagrangian: this is the so-called Chiang Lagrangian.

\subsection{Coordinates on projective space}
\label{sscCoords}

Let $x$ and $y$ be the standard basis vectors for the fundamental representation $V$ of $\SU(2)$, which we now think of as being extended to a representation of $\SL(2, \C)$, with respect to which a group element $\lb \begin{smallmatrix} t & u \\ v & w \end{smallmatrix} \rb \in \SL(2, \C)$ acts as the matrix itself.  We then have an induced basis for $S^d V$ given by $\{ x^i y^j : 0 \leq i, j \text{ and } i+j=d \}$.  We'll refer to this as \emph{the standard basis} for $S^d V$ and the corresponding coordinates (and their projective counterparts) as \emph{standard coordinates} on $S^d V$ (respectively $\P S^dV$).  This is the identification we will always use between $\P V$ and $\C \P^1$.  We will also use the identifications
\[
\C\P^1 \cong \C \cup \{\infty\} \cong \{x \in \R^3 : \norm{x}=1\}
\]
given by viewing a point $\lambda$ in $\C \cup \{\infty\}$ as both the point $[\lambda x + y] = [\lambda:1]$ in $\C\P^1$ and the point on the unit sphere given by stereographic projection through the north pole from the complex (equatorial) plane.  For example $[i:1]$ in $\C\P^1$ corresponds to $i$ in $\C$ and $(0, 1, 0)$ in $\R^3$.

The vectors $x$ and $y$ are orthonormal with respect to the standard inner product $g$, and under our embedding of $S^dV$ in $V^{\otimes d}$ as totally symmetric tensors (normalised, so, for example, $xy$ embeds as $(x\otimes y + y \otimes x)/2$) we see that with respect to $g^{\otimes d}$ the $x^iy^j$ are pairwise orthogonal and satisfy
\[
\norm{x^iy^j}=\sqrt{\frac{i!\,j!}{d!}}.
\]
We thus have a \emph{unitary basis}
\[
\Bigg\{ \sqrt{\frac{d!}{i!\,j!}}  \ x^iy^j : 0 \leq i, j \text{ and } i+j=d \Bigg\}
\]
and corresponding \emph{unitary coordinates}.

A point in $\P S^dV$ is a non-zero homogeneous polynomial of degree $d$ in $x$ and $y$, modulo $\C^*$-scalings.  Since $\C$ is algebraically closed, we can express such a polynomial as a product of $d$ linear combinations of $x$ and $y$, which are uniquely determined up to scaling and reordering.  Moreover, the action of $\SL(2, \C)$ on $\P S^dV$ induced by the representation $S^dV$ corresponds precisely to expressing elements in this factorised form and acting via the fundamental representation on each factor.  In other words, we have an $\SL(2, \C)$-equivariant identification between $\P S^dV$ and $\Sym^d \P V \cong \Sym^d \C \P^1$.

In this way, an unordered $d$-tuple of points on $\C\P^1$ can be viewed as a point of $\P S^d V$ and thus expressed in terms of either standard or unitary coordinates.  As an example, consider the equilateral triangle on the real axis in $\C \P^1$, with one vertex at $\infty$.  Its two other vertices are at $\pm 1/\sqrt{3}$, so it is represented by the point $[x(x+\sqrt{3}y)(x-\sqrt{3}y)]=[x^3-3xy^2]$ in $\P S^3 V$.  It is therefore given by $[1:0:-3:0]$ in standard coordinates and $[1:0:-\sqrt{3}:0]$ in unitary coordinates.  Note that the expression \eqref{eqMomMap} for the moment map is valid only in unitary coordinates.

\subsection{From the triangle to the Platonic solids}
\label{sscQHT}

With these notions fixed, there is another, more geometric, way to describe Chiang's construction.  If we fix a value of $d \geq 3$, and a configuration $C$ of $d$ distinct points in $\C \P^1$, then the $\SL(2, \C)$-orbit of $C$ in $\Sym^d \C \P^1 \cong \P S^d V$ is a three-dimensional complex submanifold, of which the $\SU(2)$-orbit is a three-dimensional totally real submanifold.  In \cite{AF} Aluffi and Faber identified those $C$ for which the $\SL(2, \C)$-orbit has smooth closure $X_C$ in $\P S^d V$.  There are four cases, namely the orbits of the configurations $C$ given by (using the notation of Evans--Lekili): $\tri$, the vertices of an equilateral triangle on a great circle in $\C \P^1$; $T$, $O$ and $I$, respectively the vertices of a regular tetrahedron, octahedron and icosahedron in $\C \P^1$.  These are quasihomogeneous threefolds of $\SL(2, \C)$, in the sense that they carry an $\SL(2, \C)$-action with dense Zariski open orbit.

In each case the restriction of the $\SU(2)$-action to $X_C$ (with the Fubini--Study K\"ahler form) is Hamiltonian with moment map of the form \eqref{eqMomMap}.  The representative configurations $\tri$, $T$, $O$ and $I$ all lie in the zero sets of the respective moment maps, and hence their $\SU(2)$-orbits are Lagrangian; we denote these `Platonic' Lagrangians by $L_C$.  The Chiang Lagrangian itself can then be described as $L_\tri$ in $X_\tri = \P S^3 V \cong \C \P^3$.  The stabiliser of $C$ in $\SL(2, \C)$ is a finite subgroup of $\SU(2)$ which we denote by $\Gamma_C$.

\subsection{Basic properties of the spaces $X_C$}
\label{sscBasProps}

In this subsection we collect together some of the properties of the quasihomogeneous threefolds $X_C$.  Most of the results are contained in \cite[Section 4]{EL1}.  We follow the notation of Evans--Lekili.

For each $C$ let $W_C$ denote the Zariski open $\SL(2, \C)$-orbit in $X_C$, and $Y_C$ its complement, the compactification divisor.  $Y_C$ consists of those $d$-point configurations in $X_C$ where at least $d-1$ of the points coincide.  Inside $Y_C$ we have the subvariety $N_C$ consisting of those configurations where all $d$ points coincide.

If $[z_0 : \dots : z_d]$ are standard coordinates on $\P S^d V$, then the roots of the polynomial
\[
f(T) \coloneqq \sum z_j(-T)^j
\]
correspond (with multiplicity) to the $d$-tuple of points obtained by viewing $[z]$ as a point of $\Sym^d \C \P^1$.  We count $\infty$ as a root with multiplicity $d - \deg f$.  $Y_C$ is therefore defined by the vanishing of the discriminant $\Delta (f)$ of $f$; the `infinite roots' are automatically taken care of by this.

The cohomology ring of $X_C$ is
\[
H^* ( X_C; \Z) = \Z [H, E] / (H^2 = k_C E, E^2=0),
\]
where $k_C$ is $1$, $2$, $5$, $22$ for $C$ equal to $\tri$, $T$, $O$, $I$ respectively, and $H$ is the class of a hyperplane section.  The first Chern class of $X_C$ is $c_1 (X_C) = l_C H$, where $l_C$ is $4$, $3$, $2$, $1$ for the four choices of $C$.  The latter follows from some vanishing order computations we make (see the comment after \protect \MakeUppercase {L}emma\nobreakspace \ref {labsigmaOrd}).

The numbers $1$, $2$, $5$, $22$ come about as follows.  The value of $k_C$ is the triple intersection product of three transverse hyperplane sections of $X_C$.  We can take these hyperplane sections to be of the form $X_C \cap \Pi_z$ for $z$ equal to $0$, $1$ and $\infty$, where $\Pi_z$ consists of those $d$-point configurations containing the point $z \in \C\P^1$, and then each $p \in X_C \cap \Pi_0 \cap \Pi_1 \cap \Pi_\infty$ can be described by choosing three ordered vertices of $C$ to send to $0$, $1$ and $\infty$.  This can be done in $|\Gamma_C|/2$ different ways for each $p$, corresponding to rotating $C$ before choosing the points (we divide by $2$ as we are interested in the image of $\Gamma_C$ in $\mathrm{SO}(3)$).  Any such triple gives rise to some $p$, so we conclude that the triple intersection consists of $2d(d-1)(d-2)/|\Gamma_C|$ points, which works out to be $1$, $2$, $5$, $22$ in the four cases respectively.  This argument appears in \cite[Section 0]{AF}.

In quantum cohomology the product is deformed to give a $\Z/2l_C$-graded ring
\[
QH^* ( X_C; \Z) = \Z [H, E] / (H^2 = k_C E+R_C, E^2=Q_C),
\]
where $R_C$ and $Q_C$ are as given in Table\nobreakspace \ref {tabQH}; see \cite[Section 2]{BM}.  We collapse the grading to $\Z/2$, and the ring is then concentrated in degree $0$.

\begin{center}
\vspace{0.15cm}
\begin{tabular}{c*{4}{>{\centering\arraybackslash}p{2cm}}} \toprule
$C$ & $\tri$ & $T$ & $O$ & $I$ \\ \midrule
$R_C$ & $0$ & $0$ & $3$ & $2H+24$ \\
$Q_C$ & $1$ & $H$ & $E+1$ & $2E+H+4$ \\
\bottomrule
\end{tabular}
\vspace{0.15cm}
\captionof{table}{Quantum corrections to the cup product.}
\label{tabQH}
\end{center}

If we take a basis $\xi_1$, $\xi_2$, $\xi_3$ of $\mathfrak{su}(2)$ then
\[
\sigma \mc p \mapsto (\xi_1 \cdot p) \wedge (\xi_2 \cdot p) \wedge (\xi_3 \cdot p)
\]
(for $p \in X_C$) defines a holomorphic section of $\Lambda_\C^3 TX_C$ which vanishes precisely on $Y_C$, to order $1$ (this is proved in \protect \MakeUppercase {L}emma\nobreakspace \ref {labsigmaOrd}), so $X_C$ is Fano with anticanonical divisor $Y_C$.  Let $\Omega$ be the nowhere-zero holomorphic $3$-form on the Calabi--Yau complement $W_C=X_C \setminus Y_C$ defined by $\Omega=\sigma^{-1}$.  We claim that $L_C$ is special Lagrangian (with phase $0$) in the sense of Auroux \cite[Definition 2.1]{Au}---explicitly this means that $\Omega|_{L_C}$ is real, as a section of $\C \otimes \Lambda^3 T^*L_C$.  To see this, note that for any $p \in L_C$ we get holomorphic coordinates $(z_1, z_2, z_3)$ on $X_C$ about $p$ defined by
\[
(z_1, z_2, z_3) \mapsto e^{z_1 \xi_1 + z_2 \xi_2 + z_3 \xi_3} p,
\]
and the real parts $(x_1, x_2, x_3)$ form local coordinates on $L_C$.  We then have $\sigma (p) = \pd_{z_1} \wedge \pd_{z_2} \wedge \pd_{z_3}$, so $\Omega (p) = \diff z_1 \wedge \diff z_2 \wedge \diff z_3$, and hence
\[
\left. \Omega (p) \right|_{L_C} = \diff x_1 \wedge \diff x_2 \wedge \diff x_3,
\]
which is real.

The importance of this fact lies in the following result of Auroux \cite[Lemma 3.1]{Au}:
\begin{lem} \label{labIndInt}  If $L$ is special Lagrangian in the complement $X \setminus Y$ of an anticanonical divisor in a compact K\"ahler manifold, then the Maslov index of a disc $u \mc (D, \pd D) \rightarrow (X, L)$ is given by twice the algebraic intersection number $[u] \cdot [Y]$.
\end{lem}
It will therefore be important for us to be able to calculate these intersection numbers.  This is the subject of the following subsection.

\subsection{Intersections with the compactification divisor}
\label{sscCompDiv}

The action of $\SL(2, \C)$ restricts to the variety $Y_C$, and is transitive on the dense subset $Y_C \setminus N_C$, so either every $p \in Y_C \setminus N_C$ is a smooth point of $Y_C$ or every such $p$ is singular.  But the set $\mathrm{Sing}(Y_C)$ of singular points of $Y_C$ is a proper Zariski closed subset, so we deduce that $\mathrm{Sing}(Y_C)$ is contained in $N_C$.  Since the action of $\SL(2, \C)$ on $N_C$ is also transitive, we see that in fact $\mathrm{Sing}(Y_C)$ is $N_C$ or empty, and that $N_C$ is itself smooth.

Let $\xi_v$, $\xi_e$ and $\xi_f$ in $\mathfrak{su}(2)$ be generators of rotations about a vertex of $C$, the midpoint of an edge, and the centre of a face respectively, scaled so that $\{ t \in \R : e^{2 \pi t} \cdot C=C\} = \Z$, and directed so that the vertex, midpoint and centre are at the `top' of the axis of rotation.  Here the top of the axis is taken with right-handed convention, so, for example, the rotation $(\theta, z) \mapsto e^{i \theta} z$ has $\infty$ at the top of its axis, whilst $e^{-i \theta} z$ has $0$ at the top (this is right-handed in the sense that when the fingers of the right hand are curled around the axis in the direction of rotation, the outstretched thumb points towards the top).  One choice of such rotations for $C=\tri=[x^3+y^3]$ is shown in Fig.\nobreakspace \ref {figRotGens}.  We think of the triangle as having two faces---one for each side.

\begin{figure}[ht]
\centering
\includegraphics[scale=1]{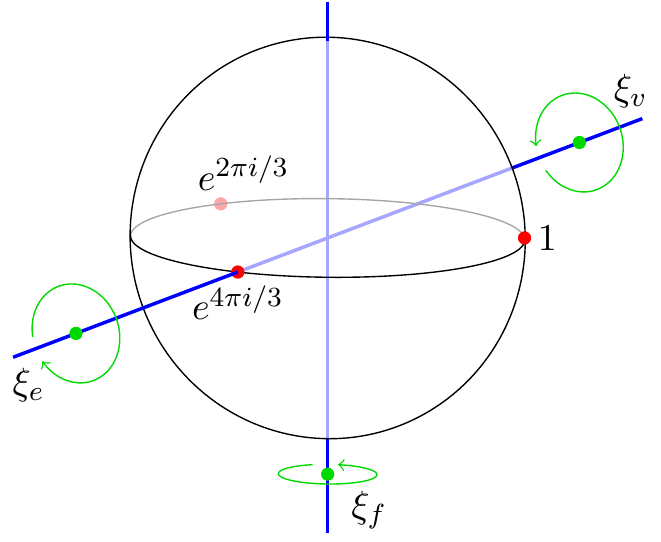}
\caption{Examples of choices for $\xi_v$, $\xi_e$ and $\xi_f$ for $C=\tri$.\label{figRotGens}}
\end{figure}

Let $\xi_g \in \mathfrak{su}(2)$ be the generator of a rotation which is generic, i.e.~not of any of these three forms, scaled in the same way.  Let $r_C$ be $2$, $3$, $4$, $5$ for $C$ equal to $\tri$, $T$, $O$, $I$, denoting the number of faces meeting at a vertex.  Then $e^{2 \pi \xi_v}$ represents a rotation through angle $2 \pi/ r_C$---the smallest angle through which one can rotate $C$ about a vertex to return it to its original position.  Similarly $e^{2\pi \xi_\bullet}$ represents a rotation through angle $2\pi /s$, where $s$ is $2$, $3$ or $1$ for $\bullet$ equal to $e$, $f$ or $g$ respectively.

\begin{defn}\label{labubull}  For $\bullet$ equal to $v$, $e$, $f$ or $g$, let the holomorphic map $u_\bullet \mc \C^* \rightarrow X_C$ be $u_{\xi_\bullet, C}$, using the notation introduced in Section\nobreakspace \ref {sscAxForm}.
\end{defn}

Note that as $z$ winds around the unit circle in $\C^*$ the configuration $u_\bullet (z)$ traces out the rotation generated by $\xi_\bullet$.  And as $z$ moves towards $\infty$ the configuration stretches towards the point $w \in \C\P^1$ at the top of the axis, meaning that all of the points of the configuration, except for the bottom of the axis if this is one of them, move towards $w$.  The model for this stretching (when $w=\infty$) is multiplication by a positive real number $t$ on $\C\P^1$: as $t \rightarrow \infty$ all points except $0$ converge to $\infty$.  Similarly, as $z$ moves towards $0$ the configuration $u_\bullet(z)$ stretches towards the bottom of the axis, corresponding to the limit $t \rightarrow 0$.

\begin{ex}\label{labuv}  For the choices shown in Fig.\nobreakspace \ref {figRotGens}, with vertices at $1$, $\zeta \coloneqq e^{2\pi i/3}$ and  $\zeta^2$, we see that for all $z \in \C^*$ the configuration representing $u_v (z)$ has a vertex at $\zeta^2$.  As $z \rightarrow \infty$, the other two vertices of $u_v(z)$ tend to $\zeta^2$, whilst as $z \rightarrow 0$ the other vertices tend to $-\zeta^2$.  As $z$ moves around $\pd D$, these two vertices rotate around the axis of $\xi_v$ (which fixes $\zeta^2$).
\end{ex}

Returning to the general case, we deduce that $u_v$ patches continuously, and hence holomorphically, over the point $0$ in the domain by a $(d-1)$-fold point at the bottom of the axis of $\xi_v$ and a single point at the top.  For $\bullet$ equal to $e$, $f$ or $g$, the map $u_\bullet$ patches over $0$ simply by a $d$-fold point at the bottom of the axis of $\xi_\bullet$.  The difference between $v$ and $e$, $f$, $g$ is that in the former case $C$ contains the top of the axis of rotation, whilst in the other three it does not.  Similarly the maps $u_\bullet$ all extend over $\infty$, although the exact nature of the limit configuration depends on $C$.  In any case, $u_\bullet$ does indeed define a whole holomorphic sphere.

We now study the intersection of $u_v$ with $Y_C$ at $0$:

\begin{lem}\label{labVanOrd} At the point $0$ in the domain, $u_v$ meets $Y_C$ transversely at a point of $Y_C \setminus N_C$, so if $F$ denotes the discriminant $\Delta (f)$ considered in Section\nobreakspace \ref {sscBasProps} then the vanishing order of $F$ along $Y_C \setminus N_C$ is the vanishing order of $F \circ u_v$ at $0$.  This number is
\[
2 \cdot \binom{d-1}{2} \cdot \frac{1}{r_C} = k_C,
\]
where $k_C$ is the coefficient appearing in the cohomology ring of $X_C$.
\end{lem}
\begin{proof}  Acting by an element of $\SU(2)$ if necessary, we may assume that $C$ contains the vertex $\infty$ and that $\xi_v$ has this vertex at its top.  Then $u_v(0)$ comprises a $(d-1)$-fold point at $0$ and a single point at $\infty$, so lies in $Y_C \setminus N_C$.  This set is connected and contained in the smooth locus of $Y_C = F^{-1}(0)$, so the vanishing order of $F$ is constant along it.

The non-zero entries in the standard coordinates of the point $C \in X_C$ are separated from each other by gaps of exactly $r_C-1$.  This is because if $w \in \C^* \subset \C\P^1$ is a vertex of $C$, and $\zeta$ is a primitive $r_C$th root of unity, then the points $w, \zeta w, \dots, \zeta^{r_C-1} w$ in $C$ contribute a factor $(-1)^{r_C+1}(wx)^{r_C} + y^{r_C}$ to the element of $\P S^d V$ representing $C$.  For example, for the tetrahedron, $C=T$, we could take the vertices to be
\[
\infty\text{, }\frac{1}{\sqrt{2}}\text{, }\frac{e^{2\pi i/3}}{\sqrt{2}}\text{ and }\frac{e^{4 \pi i/3}}{\sqrt{2}},
\]
so that $T$ is represented by
\[
x(x+\sqrt{2}y)(e^{2\pi i/3}x+\sqrt{2}y)(e^{4\pi i/3}x+\sqrt{2}y)=x(x^3+2\sqrt{2}y^3).
\]
Then the standard coordinates of the point $T \in X_T$ are $[1:0:0:2\sqrt{2}:0]$, whose non-zero entries are separated by a gap of $2$.  For the icosahedron, $C=I$, there are two $r_I$-tuples of vertices obtained in this way, giving rise to a factor of $(w_1^5x^5+y^5)(w_2^5x^5+y^5)$ for some $w_1$ and $w_2$, and one needs to check that the coefficient $w_1^5 + w_2^5$ of $x^5y^5$ doesn't vanish, which would give a gap of $9$ rather than $4$.  But this is straightforward once one observes that these two $r_I$-tuples do not lie on the equator---so $|w_j| \neq 1$---and that $w_2$ can be taken to be antipodal to $w_1$, so $w_2 = -1/\conj{w_1}$.  In fact, for each $C$ the standard coordinates of one choice of representative with a vertex at $\infty$ are given in Appendix\nobreakspace \ref {secExplReps} as $C_v$.

The coordinate expression for $u_v$ is then of the form
\[
u_v(z) = [x(y^{d-1}+a_1 z x^{r_C}y^{d-r_C-1} + \dots + a_m z^m x^{mr_C}y^{d-mr_C-1})] \in \P S^d V
\]
for all $z \in \C$, for some non-zero complex constants $a_1, \dots a_m$.  It is notationally easier to work directly with homogeneous polynomials in $x$ and $y$ than with their coefficients, which are the standard coordinates, so we shall use the chart
\begin{equation}
\label{eqChart}
\Bigg[ \sum_{i=0}^d b_i x^i y^{d-i} \Bigg] \mapsto \sum_{j \neq 1} \frac{b_j}{b_1} x^j y^{d-j} \in \C[x, y]_d / \lspan{xy^{d-1}} \cong \C^d,
\end{equation}
where $\C[x, y]_d$ denotes the homogeneous degree $d$ part of the polynomial ring and $\lspan{\cdot}$ denotes linear span as before.  In this chart we have
\[
u_v(z) = \sum_{i=0}^{\lfloor (d-1)/r_C \rfloor} a_i z^ix^{ir_C+1}y^{d-ir_C-1}
\]
for all $z \in \C$, and so
\[
u'_v(z) = \sum_{i=0}^{\lfloor (d-1)/r_C \rfloor} i a_i z^{i-1}x^{ir_C+1}y^{d-ir_C-1}.
\]
From this it is clear that $u_v'(0)$ is non-zero.

If $\eta_X$ and $\eta_Y$ represent the elements of $\mathfrak{sl}(2, \C) \leq \mathrm{Mat}_{2\times2}(\C)$ defined in the proof of \protect \MakeUppercase {L}emma\nobreakspace \ref {labInd2Ax} then (by $\SL(2, \C)$-invariance of $Y_C$) the vectors $\eta_X \cdot u_v(0)$ and $\eta_Y \cdot u_v(0)$ in $T_{u_v(0)}X_C$ actually lie in $T_{u_v(0)}Y_C$.  It is easy to check (in the above chart, for example) that these two vectors, along with $u_v'(0)$, in fact form a basis for $T_{u_v(0)}X_C$ as a complex vector space.  In particular, we have that $u_v'(0)$ defines a complementary direction to $T_{u_v(0)}Y_C$ in $T_{u_v(0)}X_C$, so $u_v$ meets $Y_C$ transversely, proving the first part of the claim.

Geometrically $\eta_X$ generates a translation of $\C \subset \C\P^1$, fixing $\infty$, so the vector $\eta_X \cdot u_v(0)$ corresponds to an infinitesimal translation of the $(d-1)$-fold point at $0$ in the limit configuration $u_v(0)$.  Similarly $\eta_Y \cdot u_v(0)$ corresponds to an infinitesimal translation of the single point at $\infty$ in $u_v(0)$.  The vector $u_v'(0)$, on the other hand, corresponds to infinitesimally `uncollapsing' the $(d-1)$-fold point into $d-1$ distinct points.

Now consider the function $F \circ u_v$.  Strictly $F$ is not a function but a section of $\mathcal{O}(2d-2)$, where $2d-2$ is the degree of the discriminant of a polynomial of degree $d$, but we will happily blur this distinction as we are only concerned with its local properties.  It is proportional to the product of the squares of the differences of the roots of $f \circ u_v$, i.e.~the vertices of the configuration representing $u_v$, with appropriate conventions to deal with the infinities.  These roots are $\infty$ and $d-1$ distinct complex numbers which tend to zero at order $z^{1/r_C}$ as $z \rightarrow 0$.  Therefore $F \circ u_v (z)$ vanishes at order
\[
2 \cdot \binom{d-1}{2} \cdot \frac{1}{r_C}.
\]
Here the binomial coefficient represents the number of pairs of roots which are converging, the $1/r_C$ corresponds to the order of their convergence, and the overall factor of $2$ comes from the fact that we are interested in the \emph{squares} of the differences.

To see that this quantity coincides with $k_C=2d(d-1)(d-2)/|\Gamma_C|$, simply apply the orbit-stabiliser theorem to the action of $\Gamma_C$ on the vertices of $C$.
\end{proof}

We can also use similar considerations to show that the holomorphic section $\sigma$ of $\Lambda_\C^3 TX_C$ constructed in Section\nobreakspace \ref {sscBasProps} vanishes to order $1$ on $Y_C$, and hence that $Y_C$ is anticanonical (rather than some higher multiple of $Y_C$):

\begin{lem}\label{labsigmaOrd}  $\sigma$ vanishes to order $1$ on $Y_C \setminus N_C$.
\end{lem}
\begin{proof}  It is enough to construct a holomorphic map $u \mc \C \rightarrow X_C$, taking $0$ to a point of $Y_C \setminus N_C$, such that $\sigma \circ u$ vanishes to order $1$ at $0$ as a section of $u^* \Lambda_\C^3 TX_C$.  We claim that $u_v$ will do.

For $z$ in a neighbourhood of $0$ in $\C$, the proof of \protect \MakeUppercase {L}emma\nobreakspace \ref {labVanOrd} shows that a basis of $T_{u_v(z)}X_C$ is given by $\eta_X \cdot u_v(z)$, $\eta_Y \cdot u_v(z)$ and $u_v'(z)$, and so we get a holomorphic frame for $u_v^* \Lambda_\C^3 TX_C$ by wedging them together.  Working in the chart \eqref{eqChart}, we saw above that $u_v'(z)$ can be written as
\[
x(a_1 x^{r_C}y^{d-r_C-1} +  2a_2 z x^{2r_C}y^{d-2r_C-1} + \dots + ma_m z^{m-1} x^{mr_C}y^{d-mr_C-1})
\]
for some non-zero $a_1, \dots, a_m \in \C$, but we also have that
\[
\eta_H \cdot u_v(z) = \frac{\diff}{\diff w}\bigg|_{w=0} x(a_1 z e^{2r_Cw} x^{r_C}y^{d-r_C-1} +  \dots + a_m z^m e^{2mr_Cw} x^{mr_C}y^{d-mr_C-1})
\]
for all $z \in \C$.  This is because $e^{w\eta_H}$ acts by multiplying the coefficient $b_j$ of $x^jy^{d-j}$ by $e^{(2j-d)w}$, and hence acts as $e^{2(j-1)w}$ on $b_j/b_1$.  Evaluating the right-hand side we see that $\eta_H \cdot u_v(z) = 2r_Czu'_v(z)$, and hence $\sigma \circ u_v$---which is proportional to $(\eta_H \cdot u_v) \wedge (\eta_X \cdot u_v) \wedge (\eta_Y \cdot u_v)$---vanishes to order $1$ at $0$, as desired.
\end{proof}

Combining \protect \MakeUppercase {L}emma\nobreakspace \ref {labVanOrd} with \protect \MakeUppercase {L}emma\nobreakspace \ref {labsigmaOrd} we deduce that
\[
K_{X_C}^{-k_C} \cong \mathcal{O}(k_CY_C) \cong \mathcal{O}(2d-2),
\]
where $K_{X_C}$ is the canonical bundle of $X_C$.  Therefore the coefficient $l_C$ of $H$ in $c_1(X_C)$ is $(2d-2)/k_C$, which agrees with the values $4$, $3$, $2$ and $1$ quoted earlier.

From the preceding two lemmas we also immediately deduce:

\begin{cor}\label{labIndFor}  The intersection number of a holomorphic disc
\[
u \mc (D, \pd D) \rightarrow (X_C, L_C)
\]
with $Y_C$ is the sum, over the intersection points, of the vanishing order of $\sigma$, which is equal to that of $F \circ u$ divided by $k_C$.
\end{cor}

In order to apply \protect \MakeUppercase {L}emma\nobreakspace \ref {labInd4Ax}, we also need to understand what happens when discs hit $N_C$:

\begin{lem}\label{labIntNC}  A clean intersection of a holomorphic disc $u$ with $N_C$ contributes at least $2$ to the intersection number $[u].[Y_C]$.  A non-clean intersection contributes at least $3$.
\end{lem}
\begin{proof}  It is sufficient to prove that $N_C = \mathrm{Sing}(Y_C)$, or equivalently that the point $p \coloneqq [y^d]$ in $N_C$ is a singular point of $Y_C$.  This will follow if we can find holomorphic maps $u_1, u_2, u_3 \mc  \C\P^1 \rightarrow X_C$ such that $u_i(0)=p$ for each $i$, the $u_i'(0)$ span $T_pX_C$, and $o(\sigma \circ u_i) \geq 2$ for each $i$, where $o(f)$ denotes the vanishing order of a holomorphic function (or section) $f$ defined on a neighbourhood of $0$.  We claim that if $A_e, A_f \in \SU(2)$ are such that $A_e \cdot u_e(0)=A_f \cdot u_f(0)=p$ then $u_1 \mc z \mapsto [(zx+y)^d]$, $u_2=A_e \cdot u_e$ and $u_3=A_f \cdot u_f$ have the required properties.

To see this, we just need to check linear independence of the $u_i'(0)$ and to compute $o(\sigma \circ u_i)$.  We work in the chart
\[
\Bigg[ \sum_{i=0}^d b_i x^i y^{d-i} \Bigg] \mapsto \sum_{j \neq 0} \frac{b_j}{b_0} x^j y^{d-j} \in \C[x, y]_d / \lspan{y^d} \cong \C^d,
\]
analogous to that used earlier but with the $y^d$-component in the denominator, rather than the $xy^{d-1}$-component.  For a similar reason to that at the start of the proof of \protect \MakeUppercase {L}emma\nobreakspace \ref {labVanOrd}, the non-zero entries in the components of $A_e \cdot u_e$ and $A_f \cdot u_f$ in this chart are separated by gaps of $1$ and $2$ respectively (see the explicit coordinates of the configurations $C_e$ and $C_f$ respectively in Appendix\nobreakspace \ref {secExplReps}).  We thus have
\[
u_2(z) = \sum_{i=1}^{\lfloor d/2 \rfloor} a_i^e z^ix^{2i}y^{d-2i}
\]
and
\[
u_3(z) = \sum_{i=1}^{\lfloor d/3 \rfloor} a_i^f z^ix^{3i}y^{d-3i}
\]
for all $z \in \C$, for some non-zero coefficients $a^e_1, a^e_2, \dots$ and $a^f_1, a^f_2, \dots$.  Hence $u_2'(0)\propto x^{d-2}y^2$ and $u_3'(0) \propto x^{d-3}y^3$.  Clearly $u_1'(0) \propto x^{d-1}y$, and so the $u_i'(0)$ are indeed linearly independent.

Finally we compute $o(\sigma \circ u_i)$.  The map $u_1$ is contained in $N_C$, so $\sigma \circ u_1$ is identically zero and we can write $o(\sigma \circ u) = \infty$.  For $u_2$ and $u_3$ we apply \protect \MakeUppercase {C}orollary\nobreakspace \ref {labIndFor}, and reduce the problem to computing the vanishing of $F \circ u_2$ and $F \circ u_3$.  As in the proof of \protect \MakeUppercase {L}emma\nobreakspace \ref {labVanOrd} this comes down to counting pairs of points in the configuration representing $u_2(z)$ or $u_3(z)$ which converge as $z \rightarrow 0$.  In the case of $u_2$ we get
\begin{equation}
\label{equeInd}
2 \cdot \binom{d}{2} \cdot \frac{1}{2} = 3k_C,
\end{equation}
whilst for $u_3$ we get
\begin{equation}
\label{equfInd}
2 \cdot \binom{d}{2} \cdot \frac{1}{3} = 2k_C.
\end{equation}
These are both clearly greater than $k_C$, so the $o(\sigma \circ u_i)$ are all at least $2$.  This completes the proof.
\end{proof}

Given the expression for $k_C$ appearing in the statement of \protect \MakeUppercase {L}emma\nobreakspace \ref {labVanOrd}, the equalities \eqref{equeInd} and \eqref{equfInd} in this proof reduce to $6 (d-2)= r_Cd$, which follows from treating $C$ as a triangulation of $S^2$ and calculating the Euler characteristic, using the fact that the number of edges is $r_Cd/2$ whilst the number of faces is $r_Cd/3$.

\subsection{Maslov indices of axial discs}
\label{sscMasAx}

In view of the results of Section\nobreakspace \ref {sscAxDiscs}, it will be useful to know the Maslov indices of axial discs in $X_C$ bounded by $L_C$, so let $u$ be such a disc.  From Section\nobreakspace \ref {sscAxForm} we know that $u$ can be written as $u(z)=e^{-i\xi \log z}u(1)$ for some $\xi \in \mathfrak{su}(2)$.  Without loss of generality we assume $u$ is non-constant so $\xi \neq 0$.

If one of the vertices in the configuration representing $u(1)$ lies at the top of the axis of $\xi$ then, up to the action of $\SU(2)$, the disc $u$ is equal to $u_v|_D$, or a multiple cover thereof (from now on we will stop writing $|_D$ for the restrictions of axial spheres to discs; we will only make it explicit when confusion could arise).  Similarly, if the top of the axis lies at the mid-point of an edge or the centre of a face then, up to the $\SU(2)$-action and taking multiple covers, $u$ is given by $u_e$ or $u_f$ respectively.  If none of these possibilities occurs then we are in the generic situation, and we may assume $\xi_g$ was chosen so that $u$ coincides with $u_g$ (or a multiple cover), again up to the action of $\SU(2)$.

Since $\SU(2)$ is connected, for any $A \in \SU(2)$ and any continuous disc $u \mc (D, \pd D) \rightarrow (X_C, L_C)$ we have that $u$ and $A \cdot u$ define the same class in $\pi_2 (X_C, L_C)$.  This means that Maslov index is invariant under the action of $\SU(2)$ on discs.  And for any such $u$ the index of the $n$-fold cover of $u$ is $n$ times the index of $u$.  We are therefore left to compute the indices of $u_v, u_e, u_f, u_g$ restricted to $D$.  This is dealt with by:

\begin{lem}\label{labAxInds}  The Maslov indices $\mu(u_\bullet|_D)$ are $2$, $6$, $4$ and $12$ for $\bullet$ equal to $v$, $e$, $f$ and $g$ respectively.
\end{lem}
\begin{proof}  By \protect \MakeUppercase {L}emma\nobreakspace \ref {labIndInt} and \protect \MakeUppercase {C}orollary\nobreakspace \ref {labIndFor} we can equivalently compute $o(F \circ u_\bullet)$, and then multiply by $2/k_C$.  This was done in \protect \MakeUppercase {L}emma\nobreakspace \ref {labVanOrd} for $\bullet = v$ and in the proof of \protect \MakeUppercase {L}emma\nobreakspace \ref {labIntNC} for $\bullet = e$ or $f$---see equations \eqref{equeInd} and \eqref{equfInd}---giving the claimed values of $\mu(u_\bullet)$ in these cases.

We mimic the same arguments for $u_g$.  Now there are $d(d-1)/2$ pairs of vertices converging at order $1$, so
\[
o(F \circ u_g) = 2 \cdot \binom{d}{2} = 6k_C,
\]
and thus $\mu(u_g|_D)=12$.
\end{proof}

Now we can give a result of Evans--Lekili \cite[Lemma 4.4]{EL1}, translating their proof into this language:

\begin{lem} \label{labMonot}
The Lagrangians $L_C$ are monotone with minimal Maslov index
\[
\min \{\mu(u) > 0 : u \in \pi_2(X_C, L_C)\}
\]
equal to $2$.
\end{lem}
\begin{proof}
The holomorphic disc $u_v$ has Maslov index $2$, and since $L_C$ is orientable all Maslov indices of discs bounded by it are even.  This proves the second statement.  To prove the first, note that by the Hurewicz and universal coefficient theorems $H^*(L_C; \Z)$ is $\Z$, $0$, $\Gamma_C^{\mathrm{ab}}$ and $\Z$ in degrees $0$ to $3$, where $\Gamma_C^{\mathrm{ab}}$ is the abelianisation of the fundamental group $\Gamma_C$ of $L_C$.  Then by the long exact sequence in homology for the pair $(X_C, L_C)$ the group $H_2 (X_C, L_C)$ has rank $1$.  Therefore Maslov index and area are proportional, and it suffices to exhibit a disc with both quantities positive---again $u_v$ will do.
\end{proof}

\section{Disc analysis for the Platonic Lagrangians}
\label{secDiscAn}

\subsection{Moduli spaces and evaluation maps}
\label{sscModAndEv}

In this subsection we introduce some notation for various moduli spaces of discs, and their accompanying evaluation maps, that we shall use in the rest of the paper.

Recall from Section\nobreakspace \ref {sscPrelim} that for a non-negative integer $k$ and a non-zero class $A \in H_2(X_C, L_C)$ we have the moduli space $\mathcal{M}_k(A)$ of holomorphic discs $(D, \pd D) \rightarrow (X_C, L_C)$ representing class $A$, with $k$ marked points $z_1, \dots, z_k$ on the boundary, modulo reparametrisation.  By \protect \MakeUppercase {L}emma\nobreakspace \ref {labDiscReg}, this is a smooth manifold of the expected dimension.

\begin{defn}\label{labModSps}  For positive integers $i$, let $M_{2i}$ be the disjoint union of the moduli spaces $\mathcal{M}_i(A)$ over the (finite) collection of classes $A$ with $\mu(A)=2i$.  Note that $i$ occurs as both the number of marked points and half the index.  This manifold carries an evaluation map $\ev_i \mc M_{2i} \rightarrow (L_C)^i$ defined by $[u, z_1, \dots, z_i] \mapsto (u(z_1), \dots, u(z_i))$.  Similarly let $M_{4i}^\mathrm{int}$ be the moduli space of unparametrised index $4i$ discs with $i$ \emph{interior} marked points, which comes with an evaluation map $\ev_i^{\mathrm{int}} \mc M_{4i}^\mathrm{int} \rightarrow X_C^i$.  The space of unparametrised discs of index $\mu$ has dimension $\dim L_C + \mu -3 = \mu$ \cite[Theorem 5.3]{Oh}, and each boundary (respectively interior) marked point increases the dimension by $1$ (respectively $2$).  Therefore $\dim M_{2i} = 3i$ and $\dim M_{4i}^\mathrm{int} = 6i$.

Let $M_{2i}'$ denote the disjoint union of the spaces $\mathcal{M}_0(A)$ over classes $A$ with $\mu(A)=2i$, i.e.~the space of unmarked unparametrised index $2i$ discs.  This has dimension $2i$.
\end{defn}

To emphasise the point, these are the bare uncompactified moduli spaces.  The only ones we expect to be compact are $M_2$ and $M'_2$: since the minimal Maslov index of $L_C$ is $2$ there can be no bubbling from an index $2$ disc with at most one marked point (if the minimal Chern number of $X_C$ is $1$, as it is when $C=I$, then a priori  there could be sphere bubbling but we shall see below that in fact this does not occur).

It is well-known, following de Silva \cite{VdS} and Fukaya--Oh--Ohta--Ono \cite[Chapter 8]{FOOObig}, that a choice of orientation and spin structure on $L_C$ induces orientations on these moduli spaces, so in order to justify working over $\Z$ (rather than $\Z/2$) we claim that $L_C$ is orientable and spin.  To see that this is the case simply note that the infinitesimal action of $\mathfrak{su}(2)$ on $L_C$ trivialises its tangent bundle.  From now on we fix an orientation and spin structure on each $L_C$ (the actual choice is irrelevant to our arguments).  Our general reference for Floer theory, in the form of quantum homology, is \cite{BCQS}, for which the orientation conventions are described in \cite[Appendix A]{BCEG}.

\subsection{Index $2$ discs}
\label{sscInd2}

We now construct the moduli space $M_2'$ of unmarked, unparametrised index $2$ discs, and compute the degree of the evaluation map $\ev_1 \mc M_2 \rightarrow L_C$.  This amounts to counting the number of index $2$ discs through a generic point of $L_C$.

Recall that the points of $Y_C$ represent $d$-point configurations on $\P V$ in which at least $d-1$ of the vertices coincide.  Define the map $\pi_C \mc Y_C \rightarrow \P V$ by letting $\pi_C(p)$ be the position of the multiple point in the configuration $p$.

\begin{prop} \label{labInd2MS}  We have that:
\begen
\item \label{ind2itm1} $M_{2}'$ is diffeomorphic to $S^2$.
\item \label{ind2itm2} $M_2$ is a circle bundle over $M_2'$ and the evaluation map $\ev_1 \mc M_2 \rightarrow L_C$ is a covering map of degree $\mathfrak{m}_0 = \pm d$.
\end{enumerate}
\end{prop}
\begin{proof} \ref{ind2itm1}  By \protect \MakeUppercase {L}emma\nobreakspace \ref {labInd2Ax} an arbitrary index $2$ disc $u$ is axial, so we can parametrise it to be in the form $u \mc z \mapsto e^{-i\xi \log z}p$, with $\xi \in \mathfrak{su}(2)$ and $p \in L_C$.  By \protect \MakeUppercase {L}emma\nobreakspace \ref {labAxInds} the top of the axis of $\xi$ must pass through a vertex of the configuration representing $p$, otherwise $u$ would have index at least $4$ (in fact, unless the top of the axis passed through a vertex, mid-point of an edge or centre of a face the index would be at least $12$).  Moreover $\xi$ must be scaled so that $\{t \in \R : e^{2\pi \xi t}p=p\}$ is $\Z$, otherwise $u$ would be a multiple cover and again have index at least $4$.  Therefore $\xi = A \xi_v A^{-1}$ for some $A \in \SU(2)$ which maps the configuration $u_v(1)$ to $p$, and $u = A \cdot u_v$.  The matrix $A$ is uniquely determined by $u$ and our choice of parametrisation.  The freedom in the latter (once we have decided to put the disc in axial form) consists of reparametrisations of the form $z \mapsto e^{i\theta}z$, which corresponds to multiplying $A$ on the right by elements of the one-parameter subgroup $H$ generated by $\xi_v$.  Thus $M_2'$ is diffeomorphic to $\SU(2) / H$, which is $S^2$ (the quotient map is the Hopf fibration).

Alternatively, we have a smooth map $\phi \mc M_2' \rightarrow N_C$ given by $u \mapsto \pi_C \circ u(0)$, where discs are parametrised so that their unique intersection with $Y_C$ occurs at $0$ in the domain.  Concretely, an index $2$ disc $u$ meets $Y_C$ at a unique point, which corresponds to a configuration on the sphere comprising a $(d-1)$-fold point and a single antipodal point, and $\phi$ sends $u$ to the position of the former.  This map is manifestly $\SU(2)$-equivariant, and the $\SU(2)$-action on $N_C$ is transitive, so $\phi$ is surjective and every point is regular.  Hence $\phi$ is a diffeomorphism $M_2' \rightiso N_C \cong S^2$ if we can show it is injective.  To prove injectivity, note that the generator $\xi$ of an (axial) index $2$ disc $u$ has $\phi(u)$ at the bottom of its axis, and its scaling is determined by the fact that $u$ is not a multiple cover, so $\phi(u)$ uniquely determines $\xi$ and hence the disc $u$ up to reparametrisation.

\ref{ind2itm2}  The once-marked moduli space is always a circle bundle over the unmarked moduli space, and $\ev_1 \mc M_2 \rightarrow L_C$ is $\SU(2)$-equivariant so is a submersion and hence a local diffeomorphism.  Since $M_2$ is compact, $\ev_1$ is therefore a covering map.  To see that the degree is $d$, up to an overall sign, note that for $p \in L_C$ and a disc $u \in M_2'$, the fibre of $M_2$ over $u$ hits $p$ under $\ev_1$ if and only if $\phi(u)$ is a vertex of the configuration representing $p$.  There are precisely $d$ such choices of $u$ for a given $p$, and in each case there is a unique point in the corresponding fibre of $M_2$ which maps to $p$.  (The reason that all discs count with the same sign is that $M_2$ is connected, so $\ev_1$ is either everywhere orientation-preserving or everywhere orientation-reversing.)

Another approach is to view $M_2'$ as $\SU(2) / H$.  Then $M_2$ is $\SU(2) / \Gamma_C'$, where $\Gamma_C'$ is the subgroup $\{e^{2 \pi k \xi_v} : k \in \Z \}$ of $\Gamma_C$, which is manifestly a circle bundle over $M_2'$.  Thinking of $L_C$ as $\SU(2)/\Gamma_C$, we see that the degree of the evaluation map, up to sign, is the index of $\Gamma_C'$ in $\Gamma_C$, which is $d$.
\end{proof}

\subsection{The antiholomorphic involution I}
\label{sscAntInv}

The purpose of the present subsection is to introduce the key tool for simplifying computations with holomorphic discs on $L_C$---a method for completing such discs to spheres, based on a partially-defined antiholomorphic involution of $X_C$.  Global antiholomorphic involutions have previously appeared in Floer theory, for example in the work of Fukaya--Oh--Ohta--Ono \cite{FOOOinv} and Haug \cite{Ha}, and we shall apply some of their ideas later.

We begin with the following observation:

\begin{prop} \label{labAntInv}  There exists an antiholomorphic involution $\tau$ of $W_C$ whose fixed-point set is precisely $L_C$.  If $C=O$ or $I$ then $\tau$ extends to the whole of $X_C$, preserving $Y_C \setminus N_C$ and $N_C$ setwise.
\end{prop}
\begin{proof}  Given a point $p \in W_C$ there exists an $A \in \SL(2, \C)$ such that $p = A \cdot C$.  $A$ is unique up to multiplication on the right by elements of $\Gamma_C$, and $p$ lies in $L_C$ if and only if $A$ is in $\SU(2)$.  Letting $\ddag$ denote conjugate-transpose-inverse (which is an antiholomorphic group involution on $\mathrm{SL(2, \C)}$, fixing $\SU(2)$), define $\tau (p) = A^\ddag \cdot C$.  Since $\ddag$ is a group homomorphism and fixes $\SU(2)$, and hence also $\Gamma_C$, this is independent of the choice of $A$, i.e.~it depends only on the underlying point $p$.  Thus $\tau$ is well-defined.  It's manifestly antiholomorphic and involutive.

We now interpret this algebraic construction geometrically.  First note that if we define $J_0 = \lb \begin{smallmatrix} 0 & -1 \\ 1 & 0 \end{smallmatrix} \rb$ then for any $A \in \SL(2, \C)$ we have
\[
A^\ddag = J_0 \conj{A} J_0^{-1}.
\]
And for $z \in \C\P^1$, the map $z \mapsto J_0^{\pm 1} \cdot \conj{z}$ is precisely the antipodal map $\alpha \mc z \mapsto -1/\conj{z}$.  So if $p \in W_C$ is described by $A \cdot C$ for some $A \in \SL(2, \C)$ then $A^\ddag \cdot C$ is obtained by taking $C$, applying the antipodal map $\alpha$ (to each factor of $\Sym^d \C \P^1 \cong \P S^d V$), acting by $A$, and then applying $\alpha$ again.

The configurations $O$ and $I$ are invariant under $\alpha$, so $\tau$ acts on $W_C$ simply as (the restriction of) the antipodal map itself.  Therefore $\tau$ extends to all of $X_C$ and clearly preserves coincidences of points, so fixes $Y_C \setminus N_C$ and $N_C$ setwise.
\end{proof}

For the triangle and tetrahedron, which are not preserved by $\alpha$, $\tau$ is rather more subtle.  It can be extended to $Y_C \setminus N_C$, which it collapses down to $N_C$, but then it cannot possibly extend further to a global involution since it is not injective.  Evans--Lekili remark that $L_\tri$ can't be the fixed-point set of \emph{any} antiholomorphic involution, since by \protect \MakeUppercase {P}roposition\nobreakspace \ref {labInd2MS}\ref{ind2itm2} the count of index $2$ discs is odd (this count was also computed by Evans--Lekili \cite[Lemma 6.2]{EL1}).

To see that $\tau$ extends over $Y_C \setminus N_C$, recall the proof of \protect \MakeUppercase {L}emma\nobreakspace \ref {labVanOrd} where we saw that the vectors $\eta_X \cdot u_v(0)$, $\eta_Y \cdot u_v(0)$ and $u_v'(0)$ form a basis for the tangent space $T_{u_v(0)}X_C$.  Therefore for any $A \in \SL(2, \C)$ the map
\[
\phi \mc (z_X, z_Y, z) \mapsto Ae^{z_X \eta_X + z_Y \eta_Y} u_v(z)
\]
gives a holomorphic parametrisation of a neighbourhood of $Au_v(0)$ in $X_C$ by a neighbourhood $U$ of $0$ in $\C^3$.  It is straightforward to check by hand that $\tau \circ u_v(z) = u_v(1/\conj{z})$ for all $z \in \C^*$, so for $(z_X, z_Y, z)$ in $U$ with $z \neq 0$ we have
\[
\tau \circ \phi (z_X, z_Y, z) = A^\ddag e^{-\conj{z}_X \eta_X^\dag - \conj{z}_Y \eta_Y^\dag}u_v(1/\conj{z}).
\]
Since $u_v(1/\conj{z})$ extends smoothly and antiholomorphically over $0$, sending $0$ to a $d$-fold point antipodal to the $(d-1)$-fold point in the configuration $u_v(0)$, we see that $\tau \circ \phi$ extends smoothly and antiholomorphically over $U$, mapping $U \cap \{z=0\}$ to $N_C$.  Hence $\tau$ itself extends smoothly and antiholomorphically over a neighbourhood of $Au_v(0)$, collapsing the intersection of this neighbourhood with $Y_C$ to $N_C$.  Since $Au_v(0)$ takes every value in $Y_C \setminus N_C$ as $A$ varies over $\SL(2, \C)$, we see that $\tau$ can be defined on all of $X_C \setminus N_C$.

For $C$ equal to $O$ or $I$, the involution on $X_C$ is the restriction of the antipodal involution on $\P S^d V$ and it is easy to see in coordinates that it is antisymplectic: the point (i.e.~homogeneous polynomial of degree $d$ in $x$ and $y$, modulo scaling)
\[
[(a_1 x + b_1 y) \dots (a_d x + b_d y)]
\]
maps to
\[
[(\conj{b}_1 x - \conj{a}_1 y) \dots (\conj{b}_d x - \conj{a}_d y)]
\]
so in standard coordinates we have
\[
[z_0 : z_1 : z_2 : \dots : z_d] \mapsto [\conj{z}_d : -\conj{z}_{d-1} : \conj{z}_{d-2} : \dots : (-1)^{d} \conj{z}_0],
\]
which flips the sign of the Fubini--Study form.  For $C$  equal to $\tri$ or $T$, however, the involution is \emph{not} antisymplectic.  In fact we shall see shortly that given a holomorphic disc $u$ on $L_C$, the reflection of $u$ by $\tau$ often has different Maslov index from $u$ itself.  By monotonicity of $L_C$, this means the reflected disc has different area.

We next take a slight detour to prepare us to deal with the points where $\tau$ is not defined.

\begin{lem} \label{labMatLim}  If $U$ is a punctured open neighbourhood of $0$ in $\C$, $E$ is an $n$-dimensional complex vector space, $e_0, \dots, e_n$ is a sequence of vectors in $E$ such that any proper subsequence is linearly independent, and $A \mc U \rightarrow \mathrm{GL}(E)$ is a holomorphic map with the property that for each $i$ the limit $\lim_{z \rightarrow 0} [A(z) \cdot e_i]$ exists in $\P E$, then:
\begen
\item \label{EndLimitm1} Shrinking $U$ if necessary, there exists a holomorphic function $\kappa \mc U \rightarrow \C^*$ such that $\kappa A$ extends continuously (and thus holomorphically) over $0$ as a map to $\End E$.
\item \label{EndLimitm2} If $A$ actually maps to $\SL(E)$ then its matrix components (with respect to any basis) are meromorphic over $0$, i.e.~they have at worst poles at $0$.
\end{enumerate}

\end{lem}
\begin{proof}  \ref{EndLimitm1} For each $i$, let $f_i \mc U \cup \{0 \} \rightarrow \P E$ denote the map $z \mapsto [A(z) \cdot e_i]$, with $f_i(0)$ defined to be the limit $\lim_{z \rightarrow 0} f_i(z)$.  Taking $e_1, \dots, e_n$ as a basis for $E$, we can view $A$ as a matrix-valued function with components $(a_{ij})$, and for each $i$ (including $0$) we can pick an index $k_i$ such that the $e_{k_i}$-component of $f_i(0)$ is non-zero.  For $i \geq 1$ let $\lambda_i$ denote $a_{k_ii}$, and analogously let $\lambda_0$ denote the $e_{k_0}$-component of $A(z) \cdot e_0$.

By shrinking $U$ if necessary we may assume that the $\lambda_i$ are nowhere zero on $U$ (by choice of the $k_i$) and so the map $B \mc U \rightarrow \mathrm{GL}(n, \C)$ given by
\[
B = A \begin{pmatrix} \lambda_1 & & \\ & \ddots & \\ & & \lambda_n \end{pmatrix}^{-1}
\]
is well-defined.  Note that for all $i, j \geq 1$ the limit $a_{ij}/\lambda_j$ exists as $z \rightarrow 0$---it is just the ratio of the $i$th and $k_j$th components of $f_i(0)$---and so $B$ extends over $0$, as a map to $\End \C^n$.

Let $B$ have components $b_{ij}$ and $e_0$ have components $\mu_i$.  The statement that $A(z) \cdot e_0$ tends to $f_0(0)$ in $\P E$ tells us that ${\lambda_0(z)}^{-1} A(z) \cdot e_0$ tends to a limit in $\C^n$ (namely the lift of $f_0(0)$ to $\C^n$ with $k_0$-component equal to $1$) as $z \rightarrow 0$, so
\[
\frac{1}{\lambda_0} B \begin{pmatrix} \lambda_1 & & \\ & \ddots & \\ & & \lambda_n \end{pmatrix} \begin{pmatrix} \mu_1 \\ \vdots \\ \mu_n \end{pmatrix} = \begin{pmatrix} \nu_1 \\ \vdots \\ \nu_n \end{pmatrix}
\]
for some holomorphic functions $\nu_i$ which extend over $0$.  We therefore have
\[
\frac{\det B}{\lambda_0}\begin{pmatrix} \lambda_1 \mu_1 \\ \vdots \\ \lambda_n \mu_n \end{pmatrix} = \adj B \begin{pmatrix} \nu_1 \\ \vdots \\ \nu_n \end{pmatrix},
\]
where $\adj B$ denotes the adjugate of $B$, and the right-hand side extends over $0$ (because $B$ and the $\nu_i$ do).  Let the components of the right-hand side be $\nu_i'$.

Now, since proper subsequences of $e_0, \dots, e_n$ are linearly independent, the $\mu_i$ must all be non-zero.  And $B$ is non-singular on $U$ so
\[
\kappa \mc z \mapsto \det B(z) / \lambda_0(z)
\]
defines a holomorphic function $U \rightarrow \C^*$.  We then have
\begin{equation}
\label{eqEndLim}
\kappa A = B \begin{pmatrix} \frac{\lambda_1 \det B}{\lambda_0} & & \\ & \ddots & \\ & & \frac{\lambda_n \det B}{\lambda_0} \end{pmatrix} = B \begin{pmatrix} \frac{\nu_1'}{\mu_1} & & \\ & \ddots & \\ & & \frac{\nu_n'}{\mu_n} \end{pmatrix},
\end{equation}
and the latter extends over $0$ (since $B$ and the $\nu_i'$ extend over $0$ and the $\mu_i$ are non-zero).  This proves \ref{EndLimitm1}.

\ref{EndLimitm2} Take determinants in \eqref{eqEndLim} to see that $\kappa$ is holomorphic over $0$.  Dividing through by $\kappa$, we thus deduce that the matrix components of $A$ with respect to our chosen basis are meromorphic over $0$.  Changing basis clearly preserves this property.
\end{proof}

Combining \protect \MakeUppercase {P}roposition\nobreakspace \ref {labAntInv} and \protect \MakeUppercase {L}emma\nobreakspace \ref {labMatLim} allows us to reflect holomorphic maps using $\tau$:

\begin{cor} \label{labExt}  If $U$ is a punctured open neighbourhood of $0$ in $\C$, and
\[
u \mc U \cup \{0\} \rightarrow X_C
\]
is a holomorphic map with $u(U) \subset W_C$, then $\tau \circ u|_U$ extends continuously over $0$.  In particular, holomorphic discs with boundary on $L_C$ extend to holomorphic spheres.
\end{cor}
\begin{proof}  Since the map $\SL(2, C) \rightarrow W_C$, $A \mapsto A \cdot C$, is a covering map, we can lift $u$ on simply connected open sets to $\SL(2, \C)$.  Lifting along a path in $U$ which encircles $0$ we may pick up some non-trivial monodromy, but since $\Gamma_C$ is finite this monodromy has finite order, $N$ say.  Defining $v = u \circ (z \mapsto z^N)$, we thus see that $v$ lifts to a map $A: U' \rightarrow \SL(2, \C)$ on some small punctured neighbourhood $U'$ of $0$.  Clearly it is enough to show that $\tau \circ v$ extends continuously over $0$.  By the definition of $\tau$, we have that $\tau \circ v$ is given by $z \mapsto A(z)^\ddag \cdot C$.  We thus need to show that ${A(z)}^\ddag \cdot C$ tends to some limit (in $X_C$, or equivalently in $\P S^d V$) as $z \rightarrow 0$.

Now, since $u(z)$ tends to a limit in $X_C$ as $z \rightarrow 0$, if we pick three distinct points $w_0, w_1, w_2 \in C \subset \C\P^1$ then for each $j$ there exists $\hat{w}_j \in \C\P^1$ with $A(z) \cdot w_j \rightarrow \hat{w}_j$ as $z \rightarrow 0$.  Letting $E = \C^2$, and picking lifts $e_0$, $e_1$ and $e_2$ of $w_0$, $w_1$ and $w_2$ to $E$, we can apply \protect \MakeUppercase {L}emma\nobreakspace \ref {labMatLim}\ref{EndLimitm1}, noting that the linear independence hypothesis holds since the $w_i$ are distinct.  The conclusion is that there exists a holomorphic $\kappa \mc U \rightarrow \C^*$ such that $B \coloneqq \kappa A$ extends over $0$.

We then have for all $z \in U'$ and all $w \in C$ that
\[
[{A(z)}^\ddag \cdot w] = [{B(z)}^\ddag \cdot w]
\]
in $\C\P^1$, and the homogeneous coordinates of the right-hand side are antiholomorphic functions of $z$ which never both vanish and which extend over $0$.  Cancelling off $\conj{z}^m$ from both coordinates, where $m$ is the minimum of their vanishing orders at $z=0$ (which may be $0$), we see that there is a well-defined limit in $\C\P^1$ as $z \rightarrow 0$.  Since this holds for all $w \in C$ we're done: $v$, and hence, $u$ extends continuously over $0$.

Now suppose $u$ is a holomorphic disc $(D, \pd D) \rightarrow (X_C, L_C)$, and let
\[
P = u^{-1}(Y_C) \subset D \setminus \pd D.
\]
Note that $P$ is discrete and hence finite.  By the standard Schwarz reflection argument, if $c \mc \C\P^1 \rightarrow \C\P^1$ denotes $z \mapsto 1/\conj{z}$ then we can extend $u$ to a holomorphic map $\double{u} \mc \C \P^1 \setminus c(P) \rightarrow X_C$ by defining
\[
\double{u}(z) = \begin{cases} u(z) & \text{if } z \in D \\ \tau \circ u \circ c(z) & \text{if } z \in c(D \setminus P).
\end{cases}
\]
The only question now is whether $\double{u}$ extends holomorphically (or, equivalently, continuously) over $c(P)$.  But this is precisely what we just showed.  Hence the disc extends to a sphere as claimed.
\end{proof}

To study holomorphic discs bounded by $L_C$, we can therefore now restrict our attention to holomorphic spheres with equator on $L_C$.  This is extremely useful as holomorphic maps from $\C \P^1$ into $X_C$ are necessarily algebraic (pull back $\mathcal{O}_{\C\P^d}(1)$ from $X_C$ and use the fact that holomorphic line bundles on $\C \P^1$ are all of the form $\mathcal{O}_{\C\P^1}(m)$, for $m \in \Z$, and thus are algebraic).  We shall frequently use the notation $\double{u}$ for the completion of a disc $u$ to a sphere, without explicit warning.  Following Fukaya et al.~\cite{FOOOinv} and Haug \cite{Ha}, we will refer to this sphere as the \emph{double} of $u$.

Note that in the proof of \protect \MakeUppercase {C}orollary\nobreakspace \ref {labExt} it is important that we can use the finiteness of the order of the monodromy to lift the map $u$ to $\SL(2, \C)$ (after composing with an appropriate $z \mapsto z^N$) on a whole punctured neighbourhood of $0$---there exist holomorphic maps $A \mc \C \setminus \R_{\geq 0} \rightarrow \SL(2, \C)$, for example, such that $A(z) \cdot w \rightarrow 0$ as $z \rightarrow 0$ for $w$ equal to $0$, $1$ or $\infty$, but with
\[
{A(z)}^\ddag \cdot 1
\]
\emph{not} tending to any limit as $z \rightarrow 0$.  An example of such a map is given by
\[
A(e^z) = \frac{e^{-iz}}{\sqrt{2} z} \begin{pmatrix} e^{2iz} && -e^{2iz} \\ z^2 + 1 && z^2 - 1 \end{pmatrix},
\]
with $\im{z}$ taken in $(0, 2 \pi)$.

\subsection{Poles}
\label{sscPolesI}

We have already seen several examples of the importance of the intersections of a holomorphic disc $u \mc (D, \pd D) \rightarrow (X_C, L_C)$ with the compactification divisor $Y_C$.  We call such points \emph{poles of $u$}; in analogy with the study of meromorphic functions, the term will be used quite loosely to refer to both the position of such points (in $D$) and to various aspects of the local behaviour of $u$ there.  We can similarly speak of the poles of the double $\double{u}$ of $u$, which occur precisely at the poles of $u$ and their reflections across $\pd D$, or indeed of any holomorphic map from a Riemann surface to $X_C$ (as long as no component of the map is contained in $Y_C$).

In this subsection we study these poles systematically, developing the analogy with meromorphic functions.  Of course $\C \P^1$ can be viewed as $\C \cup \{\infty\}$, carrying an obvious action of the additive group $\C$ with dense open orbit compactified by the divisor $\{\infty\}$.  A meromorphic function $f$ on a Riemann surface $\Sigma$ corresponds to a holomorphic map $\Sigma \rightarrow \C \P^1$ and the poles of $f$ as a function are then precisely the intersections of the corresponding map with the compactification divisor, so in this sense our new definition extends the existing one.

We begin the discussion proper with the key definitions:

\begin{defn}  A \emph{pole germ} is the germ (at $0$) of a holomorphic map $u$, from an open neighbourhood of $0$ in $\C$ to $X_C$, such that $u^{-1}(Y_C)$ contains $0$ as an isolated point.  More generally, for a Riemann surface $\Sigma$ and a point $a \in \Sigma$, one can speak of a \emph{pole germ at $a$}.  If we don't specify `at $a$' then we are implicitly working at $0$ in $\C$.  We define an equivalence relation on pole germs at $a$ by $u_1 \sim u_2$ if and only if there exists a germ of  holomorphic map $A$, from a neighbourhood of $a$ in $\Sigma$ to $\SL(2, \C)$, such that $u_2 = A \cdot u_1$, and the \emph{principal part} of a pole germ $u$ is its equivalence class $[u]_a$ under this relation.

We say a pole germ $u$ is of \emph{type $\xi \in \mathfrak{su}(2)$} and \emph{order $k \in \Z_{\geq 1}$} if its principal part is
\[
[z \mapsto e^{-ik \xi \log z} \cdot C]_0,
\]
and $\xi$ is scaled so that $\{t \in \R : e^{2 \pi \xi t} \in \Gamma_C\} = \Z$.  We say that $u$ is \emph{quasi-axial} if it is of type $\xi$ and order $k$ for some $\xi$ and $k$.  The index $\mu_a(u)$ of a pole germ $u$ at $a$ is defined to be twice the intersection multiplicity of $u$ with $Y_C$ at $a$.
\end{defn}

A priori the notion of being of type $\xi$ and order $k$ only makes sense for pole germs at $0$ in $\C$, or after fixing a local coordinate about the base point $a$ if working on an arbitrary Riemann surface $\Sigma$, but we will show in \protect \MakeUppercase {L}emma\nobreakspace \ref {labRigQA} that in fact it is independent of such a choice of coordinate.  Note that if $u$ is a quasi-axial pole germ of type $\xi_1$ then it is also of type $\xi_2$ whenever $\xi_1$ and $\xi_2$ are conjugate by an element of $\Gamma_C$.  \protect \MakeUppercase {L}emma\nobreakspace \ref {labRigQA} also shows that the converse holds, i.e.~if $u$ is of types $\xi_1$ and $\xi_2$ then the $\xi_i$ are conjugate by an element of $\Gamma_C$.

Clearly if $u \mc \Sigma \rightarrow X_C$ is a holomorphic map from a Riemann surface $\Sigma$, with a pole at $a \in \Sigma$ (i.e.~$a$ is an isolated point of $u^{-1}(Y_C)$), then $u$ defines a pole germ at $a$.  We can therefore apply the terms defined for pole germs at $a$ to poles of actual maps $u$, as opposed to just germs.  For example, we can say that the Maslov index of a holomorphic disc is the sum of the indices of its poles.

Next we prove a simple lemma:

\begin{lem} \label{labPolDeg}  The index of a pole germ $u$ at a point $a$ in a Riemann surface $\Sigma$ is determined by its principal part $[u]_a$.
\end{lem}
\begin{proof}  Suppose $A$ is a holomorphic map from an open neighbourhood of $a$ in $\Sigma$ to $\SL(2, \C)$.  We want to show that $\mu_a(u)=\mu_a(A\cdot u)$.  By taking a local coordinate about $a$ we may assume that we are working at $0$ in $\C$.

Recall that the divisor $Y_C$ is defined by the vanishing (to order $k_C$) of the discriminant $F$ in $X_C$.  For a point $
[(u_1 x + v_1 y) \dots (u_d x + v_d y)] \in X_C$ we have
\[
F\big([(u_1 x + v_1 y) \dots (u_d x + v_d y)]\big) = \begin{vmatrix} u_1^{d-1} & u_2^{d-1} & \dots & u_d^{d-1} \\ u_1^{d-2}v_1 & u_2^{d-2}v_2 & \dots & u_d^{d-2}v_d \\ \vdots & \vdots & \ddots & \vdots \\ v_1^{d-1} & v_2^{d-1} & \dots & v_d^{d-1} \end{vmatrix}^2,
\]
by the Vandermonde determinant, so if $\rho \mc \SL(2, \C) \rightarrow \mathrm{GL}(d, \C)$ denotes the representation $S^{d-1} V$ (which describes the action on the columns of the above matrix) then we have
\[
F \circ (A \cdot u) = \det\big( \rho (A) \big)^2 F \circ u.
\]
Hence $F \circ (A \cdot u)$ and $F \circ u$ vanish to the same order at $0$.
\end{proof}

\begin{ex} \label{labAxPol}  As an illustration, recall the axial spheres $u_v$, $u_e$ and $u_f$ defined in Section\nobreakspace \ref {sscCompDiv}.  Their poles at $0$ are of type $\xi_v$, $\xi_e$ and $\xi_f$ respectively, and order $1$.  For $C$ equal to $O$ or $I$, the poles at $\infty$ are of the same type and order.  For $C=\tri$, the poles at $\infty$ are of type $\xi_e$, $\xi_v$, $\xi_f$ respectively (all of order $1$), since a vertex of the triangle is opposite the mid-point of an edge whilst the two faces are `opposite' each other.  Similarly, for $C=T$ they are of type $\xi_f$, $\xi_e$, $\xi_v$ (and order $1$), since a vertex of the tetrahedron is opposite the centre of a face whilst mid-points of edges are opposite each other.  By \protect \MakeUppercase {L}emma\nobreakspace \ref {labPolDeg} the index of a quasi-axial pole is determined by its type and order, so from \protect \MakeUppercase {L}emma\nobreakspace \ref {labAxInds} we see that poles of type $\xi_v$, $\xi_e$, $\xi_f$ and $\xi_g$ of order $1$ have indices $2$, $6$, $4$ and $12$ respectively.  A pole of type $\xi$ and order $k$ is equivalent to a $k$-fold cover of a pole of type $\xi$ and order $1$ so its index is $k$ times the index of the order $1$ pole.
\end{ex}

For a positive integer $N$, let $\psi_N$ denote the map $z \mapsto z^N$ or its germ at $0$.  If $u_1$ and $u_2$ are two pole germs with the same principal part (i.e. $u_1 \sim u_2$) then it is clear that for all positive integers $N$ we have $u_1 \circ \psi_N \sim u_2 \circ \psi_N$.  A converse is also true, which allows us to lift questions about principal parts to multiple covers:

\begin{lem} \label{labPolComp}  If $u_1$ and $u_2$ are pole germs such that for some positive integer $N$ we have $u_1 \circ \psi_N \sim u_2 \circ \psi_N$, then $u_1 \sim u_2$.  (Clearly a similar result is valid for pole germs at arbitrary points $a$, if $\psi_N$ is replaced by an appropriate local $N$-fold cover.)
\end{lem}
\begin{proof}  Replacing $N$ by a multiple if necessary, we may assume that away from $0$ the pole germs $u_1 \circ \psi_N$ and $u_2 \circ \psi_N$ lift to maps $B_1$ and $B_2$ from a punctured neighbourhood of $0$ to $\SL(2, \C)$.  Since $u_1 \circ \psi_N \sim u_2 \circ \psi_N$ there exists a map $A$ from a (non-punctured) neighbourhood of $0$ to $\SL(2, \C)$ such that $B_2^{-1}AB_1 \in \Gamma_C$ (on a small punctured neighbourhood of $0$).  If we can show that $A(z)$ is invariant under $z \mapsto e^{2 \pi i / N} z$ then we have that $A = \tilde{A} \circ \psi_N$ for some holomorphic map $\tilde{A}$, and that $u_2 = \tilde{A} \cdot u_1$, so $u_1 \sim u_2$.

Well, since $\Gamma_C$ is discrete, there exists $M \in \Gamma_C$ such that $B_2=AB_1M$ near $0$; replacing $B_1$ by $B_1M$ we may assume that $M$ is the identity.  By the construction of $B_1$ and $B_2$ as lifts of an $N$-fold cover, there exist $D_1, D_2 \in \Gamma_C$ such that $B_i (\zeta z)=B_i(z) D_i$ for all $z$ in a punctured neighbourhood of $0$, where $\zeta \coloneqq e^{2 \pi i / N}$.  We then have
\[
A(z)=B_2(z)B_1(z)^{-1} \text{ and } A(\zeta z)=B_2(z)D_2D_1^{-1}B_1(z)^{-1}
\]
so
\[
A(z)^{-1}A(\zeta z)=B_1(z)D_2D_1^{-1}B_1(z)^{-1}
\]
on a punctured neighbourhood of $0$.  Taking characteristic polynomials and letting $z \rightarrow 0$, we see that $D_2D_1^{-1}$ has characteristic polynomial $(T-1)^2$.  We also know that $D_2D_1^{-1}$ is diagonalisable, since it lies in $\Gamma_C \subset \SU(2)$, so it must be the identity, $I$.  Hence $A(z)^{-1}A(\zeta z)=I$ on a punctured neighbourhood of $0$, and thus $A(z)$ is invariant under $z \mapsto \zeta z$, as required.
\end{proof}

In light of this result and \protect \MakeUppercase {L}emma\nobreakspace \ref {labMatLim}\ref{EndLimitm2}, we can reduce the study of poles to that of meromorphic maps to $\SL(2, \C)$ with poles in the ordinary sense.  We briefly remark that it is important that $D_2D_1^{-1}$ is diagonalisable in the last step of the above proof.  Otherwise we could have, say,
\[
D_2D_1^{-1} = \begin{pmatrix} 1 & 1 \\ 0 & 1 \end{pmatrix} \text{ and } B_1 = \begin{pmatrix} z & 0 \\ 0 & 1/z \end{pmatrix}.
\]
Then $B_1D_2D_1^{-1}B_1^{-1} \rightarrow I$ as $z \rightarrow 0$ but clearly $D_2D_1^{-1} \neq I$.

We now characterise the simplest type of pole.

\begin{lem} \label{labVertPol}  A pole germ $u$ is of type $\xi_v$ if and only if $u(0) \in Y_C \setminus N_C$.  In this case, the order of $u$ is $\mu_0(u)/2$.
\end{lem}
\begin{proof}  If $u$ is of type $\xi_v$ then it is easy to see that the limit configuration $u(0)$ consists of a $(d-1)$-fold point and a separate single point.  Hence $u(0) \in Y_C \setminus N_C$.  The statement about the order follows immediately from the comments at the end of \protect \MakeUppercase {E}xample\nobreakspace \ref {labAxPol}.

Conversely suppose that $u(0)$ is of this form.  By replacing $u$ by $A_0 \cdot u$ for a suitable $A_0 \in \SL(2, \C)$ (which doesn't change the principal part), we may assume that the $(d-1)$-fold point is at $0$, and the single point is at $\infty$.  For appropriate $N$ we can lift $u \circ \psi_N$ to a map $B$ from a punctured neighbourhood of $0$ to $\SL(2, \C)$.  Let $w_\infty \in C$ be the point with $B(z) \cdot w_\infty \rightarrow \infty$ as $z \rightarrow 0$, and let $R \in \SU(2)$ be a rotation sending $w_\infty$ to $\infty$.

Now consider the map $\tilde{B} \coloneqq BR^{-1}$.  This has the property that $\tilde{B}(z) \cdot \infty \rightarrow \infty$ as $z \rightarrow 0$, but for $d-1$ other points $p_1, \dots, p_{d-1} \in \C\P^1$ (namely the points of $(R\cdot C) \setminus \{\infty\}$) we have $\tilde{B}(z) \cdot p_i \rightarrow 0$.  Let
\[
\tilde{B}=\begin{pmatrix} a & b \\ c & d \end{pmatrix}.
\]
By \protect \MakeUppercase {L}emma\nobreakspace \ref {labMatLim}\ref{EndLimitm2}, the functions $a$, $b$, $c$ and $d$ are meromorphic over $0$.  From our knowledge of the limit behaviour, we have $c/a \rightarrow 0$ as $z \rightarrow 0$, and $(ap_i+b)/(cp_i+d) \rightarrow 0$ for $i=1, \dots, d-1$.

For a meromorphic function $f$ defined on a neighbourhood of $0$, let $o(f)$ denote the vanishing order of $f$ at $0$---this may be $\infty$, if $f$ is identically $0$, or negative if $f$ has a pole (this extends our earlier definition from the proof of \protect \MakeUppercase {L}emma\nobreakspace \ref {labIntNC}).  The statements about the limits above can be expressed as $o(c)>o(a)$ and $o(ap_i+b)>o(cp_i+d)$.  Note that for all $i$ we have $o(ap_i+b)\geq \min\{o(a), o(b)\}$, and for all but at most one $i$ we have equality; similarly for $o(cp_i+d)$.  As $d\geq3$, we can pick an index $j$ so that we have equality for $ap_j+b$, and then
\[
\min\{o(a), o(b)\} = o(ap_j+b) > o(cp_j+d) \geq \min\{o(c), o(d)\}.
\]
Since $o(c)>o(a)$, we must have $o(a), o(b)>o(d)$.

By considering $\det \tilde{B}$, we see that $ad-bc=1$.  And since $o(a)<o(c)$ and $o(d)<o(b)$, we have $o(ad)<o(bc)$.  Therefore $ad \rightarrow 1$ and $bc \rightarrow 0$ as $z \rightarrow 0$, so $o(a)=-o(d)$.  Since $o(d)<o(a)$, we must have $o(a)>0$.  Letting $\kappa=o(a)$ we see that $az^{-\kappa}$ and $dz^\kappa$ are holomorphic over $0$, as are $bz^\kappa$ and $cz^{-\kappa}$.  In other words
\[
\tilde{B}=A\begin{pmatrix} z^\kappa & 0 \\ 0 & z^{-\kappa} \end{pmatrix}
\]
for a holomorphic $\SL(2, \C)$-valued function $A$ (with entries $az^{-\kappa}$, $bz^\kappa$, $cz^{-\kappa}$, $dz^\kappa$), so
\[
u \circ \psi_N = B \cdot C = \tilde{B} R \cdot C \sim \begin{pmatrix} z^\kappa & 0 \\ 0 & z^{-\kappa} \end{pmatrix} R \cdot C \sim e^{-i (2r_C \kappa) \xi_v \log z} \cdot C.
\]
The final equivalence holds because the configuration $R \cdot C$ has a vertex at $\infty$ and as $z$ moves around the unit circle the matrix
\[
\begin{pmatrix} z^\kappa & 0 \\ 0 & z^{-\kappa} \end{pmatrix}
\]
sweeps out a rotation about this vertex through angle $4\pi$, i.e.~$2r_C$ times the smallest angle needed to bring $R \cdot C$ back to its initial position.

Taking indices of poles, we get $N \mu_0(u) = 4r_C\kappa$ so
\[
m \coloneqq 2r_C\kappa/N = \mu_0(u)/2
\]
is an integer (it's half of twice an intersection number).  We can therefore write
\[
u \circ \psi_N \sim \big(e^{-im \xi_v \log z} \cdot C\big)\circ \psi_N
\]
and deduce by \protect \MakeUppercase {L}emma\nobreakspace \ref {labPolComp} that $u$ is of type $\xi_v$ and order $m = \mu_0(u)/2$, as claimed.
\end{proof}

Clearly the value of $u(0)$ is independent of the choice of local coordinate about $0$ (as long as it is centred at $0$ of course), so the property of being of type $\xi_v$ is also independent of this choice; this gives the first hint at the rigidity of quasi-axial poles, which is explored further in the next result:

\begin{lem} \label{labRigQA}  Suppose $u$ is a pole germ of type $\xi$ and order $k$.
\begen
\item \label{Rigitm1} If $u$ is also of type $\hat{\xi}$ and order $\hat{k}$ then $\hat{k}=k$ and $\hat{\xi}$ is conjugate to $\xi$ by an element of $\Gamma_C$.
\item \label{Rigitm2} If $\phi$ is a holomorphic function defining a change of coordinates about $0$, with $\phi(0)=0$, then $u \circ \phi$ is also of type $\xi$ and order $k$.
\end{enumerate}
So given a pole germ at an arbitrary point on a Riemann surface, it makes sense to say that it is quasi-axial (by choosing a local coordinate).  The order of such a pole is uniquely defined, and its type is well-defined up to conjugation by $\Gamma_C$.  With this in place we can state:
\begen
\setcounter{enumi}{2}
\item \label{Rigitm3} Given a holomorphic disc $u \mc (D, \pd D) \rightarrow (X_C, L_C)$ with a pole at $a$ of type $\xi$ and order $k$, the corresponding pole of $\double{u}$ at $1/\conj{a}$ is of type $-\xi$ and order $k$.
\end{enumerate}
\end{lem}
\begin{proof} \ref{Rigitm1} Identifying $\mathfrak{su}(2)$ with the trace-free skew-hermitian $2\times2$ matrices in the standard way, there exists $R \in \SU(2)$ such that
\[
\xi = R \begin{pmatrix} \kappa i & 0 \\ 0 & -\kappa i \end{pmatrix} R^{-1}
\]
for some positive real number $\kappa$.  Then $u$ can be written in the form
\begin{equation}
\label{eqAxPol}
u(z)=A(z) R \begin{pmatrix} z^{k\kappa} & 0 \\ 0 & z^{-k\kappa} \end{pmatrix} R^{-1} \cdot C
\end{equation}
for some holomorphic map $A$ from a neighbourhood of $0$ to $\SL(2, \C)$.  We also deduce that $\kappa$ is rational since some multiple cover of $u$ lifts to $\SL(2, \C)$.  We can do exactly the same for $\hat{\xi}$, with some $\hat{R}$, $\hat{\kappa}$ and $\hat{A}$.

Note that for any $N \in \Z_{\geq1}$ the disc $u \circ \psi_N$ is of types $\xi$ and $\hat{\xi}$ and orders $Nk$ and $N\hat{k}$, so it suffices to prove that for some $N$ the result holds with $u \circ \psi_N$ in place of $u$.  Choosing $N$ so that $Nk\kappa, N\hat{k}\hat{\kappa} \in \Z$, we may therefore assume that $k\kappa$ and $\hat{k}\hat{\kappa}$ are integers, and hence that $z^{\pm k \kappa}$ and $z^{\pm \hat{k}\hat{\kappa}}$ define genuine holomorphic functions.

Since
\[
A(z) R \begin{pmatrix} z^{k\kappa} & 0 \\ 0 & z^{-k\kappa} \end{pmatrix} R^{-1} \cdot C = \hat{A}(z) \hat{R} \begin{pmatrix} z^{\hat{k}\hat{\kappa}} & 0 \\ 0 & z^{-\hat{k}\hat{\kappa}} \end{pmatrix} \hat{R}^{-1} \cdot C
\]
for all $z$ in a punctured neighbourhood of $0$, there exists $D \in \Gamma_C$ such that
\[
A(z) R \begin{pmatrix} z^{k\kappa} & 0 \\ 0 & z^{-k\kappa} \end{pmatrix} R^{-1} = \hat{A}(z) \hat{R} \begin{pmatrix} z^{\hat{k}\hat{\kappa}} & 0 \\ 0 & z^{-\hat{k}\hat{\kappa}} \end{pmatrix} \hat{R}^{-1} D
\]
near $0$.  Letting $S = \hat{R}^{-1}DR \in \SU(2)$, we therefore have that
\[
\begin{pmatrix} z^{\hat{k}\hat{\kappa}} & 0 \\ 0 & z^{-\hat{k}\hat{\kappa}} \end{pmatrix} S \begin{pmatrix} z^{-k\kappa} & 0 \\ 0 & z^{k\kappa} \end{pmatrix}
\]
is holomorphic over $0$.  Recalling that $\kappa$ and $\hat{\kappa}$ are positive, and by definition so are $k$ and $\hat{k}$, this is only possible if $\hat{k}\hat{\kappa}=k\kappa$ and $S$ is diagonal (and hence commutes with $\lb \begin{smallmatrix} 1 & 0 \\ 0 & -1 \end{smallmatrix} \rb$).  We thus have
\[
\hat{k}\hat{\xi} = \hat{R} \begin{pmatrix} \hat{k}\hat{\kappa}i & 0 \\ 0 & -\hat{k}\hat{\kappa}i \end{pmatrix} \hat{R}^{-1} = DRS^{-1} \begin{pmatrix} k\kappa i & 0 \\ 0 & -k\kappa i \end{pmatrix} SR^{-1}D^{-1} = D k\xi D^{-1},
\]
so $\hat{k}\hat{\xi}$ and $k\xi$ are conjugate by an element of $\Gamma_C$.

By our scaling convention, we have
\[
\{t \in \R : e^{2 \pi k\xi t} \in \Gamma_C\} = \frac{1}{k}\Z \text{ and } \{t \in \R : e^{2 \pi \hat{k}\hat{\xi} t} \in \Gamma_C\} = \frac{1}{\hat{k}}\Z.
\]
Since $\hat{k}\hat{\xi}$ and $k\xi$ are conjugate by $\Gamma_C$, we must therefore have that $\hat{k}=k$ and hence that $\hat{\xi}$ is conjugate to $\xi$ by $\Gamma_C$, as claimed.

\ref{Rigitm2} Let $\kappa \in \mathbb{Q}, R \in \SU(2)$ be as in the previous part.  Note that $\phi$ vanishes to order $1$ at the origin, so there exists a holomorphic $\kappa$th power $\chi$ of $\phi(z)/z$ defined about $z=0$.  We then have (using the expression \eqref{eqAxPol})
\[
u\circ \phi(z)=\bigg((A\circ \phi(z)) R \begin{pmatrix} \chi(z)^k & 0 \\ 0 & \chi(z)^{-k} \end{pmatrix} R^{-1}\bigg) R \begin{pmatrix} z^{k\kappa} & 0 \\ 0 & z^{-k\kappa} \end{pmatrix} R^{-1} \cdot C
\]
near $0$, and the expression in the large brackets is holomorphic.  So $u\circ \phi$ is quasi-axial, of type $\xi$ and order $k$.

\ref{Rigitm3} By applying the change of coordinate $z \mapsto (z-a)/(\conj{a}z-1)$ (which commutes with the reflection $c \mc z \rightarrow 1/\conj{z}$) we may assume $a=0$.  For $z$ near $0$ we have $u(z)=A(z)e^{-ik\xi \log z}\cdot C$ for some holomorphic map $A$ from a neighbourhood of $0$ to $\SL(2, \C)$.  Then for $z$ near $\infty$ we have
\[
u(z)=\Big( A(1/\conj{z})e^{ik\xi \log \conj{z}}\Big)^\ddag \cdot C=A(1/\conj{z})^\ddag e^{-ik\xi \log z} \cdot C,
\]
using $\xi^\dag=-\xi$.

Now let $\phi \mc \C\P^1 \rightarrow \C\P^1$ be $z \mapsto 1/z$.  For $z \in \C^*$ small we have
\[
u \circ \phi (z) = A(\conj{z})^\ddag e^{ik\xi \log z} \cdot C.
\]
Therefore the pole of $u \circ \phi$ at $0$---and hence that of $u$ at $\infty$---is of type $-\xi$ and order $k$, completing the proof.
\end{proof}

This result shows that quasi-axial poles are rather well-behaved, and we can make the following definition:

\begin{defn} \label{labQA} A disc $u$ is \emph{quasi-axial} if all of its poles are quasi-axial.
\end{defn}

Armed with \protect \MakeUppercase {L}emma\nobreakspace \ref {labVertPol}, and the sanity check of \protect \MakeUppercase {L}emma\nobreakspace \ref {labRigQA}, we can now classify poles and discs of index $2$, and obtain a new proof of \protect \MakeUppercase {L}emma\nobreakspace \ref {labInd2Ax} in this setting:

\begin{cor} \label{labInd2Pol}  All index $2$ poles are of type $\xi_v$ and order $1$.  All index $2$ discs with boundary on $L_C$ are, up to reparametrisation, of the form $A \cdot u_v$ for $A \in \SU(2)$.  In particular they are all axial.
\end{cor}
\begin{proof}  If $u$ is a pole germ at a point $a$ on a Riemann surface $\Sigma$, with $\mu_a(u)=2$, then $u$ intersects $Y_C$ with multiplicity $1$ at $a$ and hence $u(a) \in Y_C \setminus N_C$ by \protect \MakeUppercase {L}emma\nobreakspace \ref {labIntNC}.  So by \protect \MakeUppercase {L}emma\nobreakspace \ref {labVertPol} $u$ is of type $\xi_v$ and order $1$.

Now suppose $u$ is an index $2$ disc.  Since $\mu(u)$ is the sum of the indices of the poles of $u$, all of which are positive and even, we see that $u$ has a single pole, of index $2$.  Reparametrising $u$ if necessary, we may assume that the pole is at the point $0 \in D$.  By the above, we know that the pole is of type $\xi_v$ and order $1$.  Thus $u\circ \psi_{2r_C}$ lifts to a map $B \mc D \setminus \{0\} \rightarrow \SL(2, \C)$ which lands in $\SU(2)$ when restricted to the boundary $\pd D$ and is such that $B(z)e^{i(2r_C)\xi_v \log z}$ is holomorphic over $0$.

Therefore
\[
z \mapsto B(z)e^{i(2r_C)\xi_v \log z} \cdot C
\]
defines a holomorphic map $D \rightarrow W_C$ with boundary on $L_C$; in other words it's a holomorphic disc on $L_C$ of index $0$ (it doesn't meet $Y_C$), so by monotonicity is constant---say $A \cdot C$ for $A \in \SU(2)$.  Then, multiplying $A$ on the right by an element of $\Gamma_C$ if necessary, we have $B(z)=Ae^{-i(2r_C)\xi_v \log z}$, so $u$ is
\[
z \mapsto Ae^{-i\xi_v \log z} \cdot C,
\]
which is precisely $A \cdot u_v$.  In particular, we have $u(e^{i\theta}z)=e^{\theta A\xi_vA^{-1}}u(z)$ for all $z\in D$ and $\theta \in \R$, so $u$ is axial.
\end{proof}

An alternative way to see that a disc without poles is constant, not using monotonicity, would be to reflect it to a sphere (also without poles) and lift it to a holomorphic map $\C\P^1 \rightarrow \SL(2, \C)$.  Any such map is constant (compose it with the embedding $\SL(2, \C)\hookrightarrow \C^4$ and observe that the composite must be constant) and so the disc itself is constant.

We can also classify the poles on index $4$ discs, although now we \emph{do} rely on the general results of Section\nobreakspace \ref {sscAxDiscs}:

\begin{cor} \label{labInd4Pol}  Suppose $u$ is an index $4$ disc.  Either $u$ has two poles of type $\xi_v$ and order $1$, one pole of type $\xi_v$ and order $2$, or one pole of type $\xi_f$ and order $1$.  In the latter two cases the disc is axial and is an $\SU(2)$-translate of $u_v \circ \psi_2$ or $u_f$ respectively.
\end{cor}
\begin{proof}  The poles of $u$ have positive even indices which sum to $4$, so either there are two poles of index $2$ (and thus of type $\xi_v$ and order $1$ by \protect \MakeUppercase {C}orollary\nobreakspace \ref {labInd2Pol}), or one pole of index $4$.  In the latter case, if the disc hits $Y_C$ on $Y_C \setminus N_C$ then by \protect \MakeUppercase {L}emma\nobreakspace \ref {labVertPol} the pole is of type $\xi_v$ and order $2$, and arguing as in \protect \MakeUppercase {C}orollary\nobreakspace \ref {labInd2Pol} the disc is axial of the form $A \cdot u_v \circ \psi_2$ (clearly this argument generalises to show that if $u$ is a quasi-axial disc with a single pole of type $\xi$ and order $k$ then it is a translate of $e^{-i\xi \log z} \circ \psi_k$).

Otherwise the disc hits $N_C$, and by \protect \MakeUppercase {L}emma\nobreakspace \ref {labIntNC} this intersection is clean.  Applying \protect \MakeUppercase {L}emma\nobreakspace \ref {labInd4Ax} with $Z=N_C$ we deduce that $u$ is again axial.  From \protect \MakeUppercase {L}emma\nobreakspace \ref {labAxInds} we know that the only axial discs of index $4$ which hit $N_C$ are translates of $u_f$ under the action of $\SU(2)$.  Hence $u$ has a single pole of type $\xi_f$ and order $1$.
\end{proof}

\subsection{Group derivatives}
\label{sscGrpDer}

In this subsection we define a meromorphic Lie algebra-valued notion of the derivative of a holomorphic curve in $X_C$, which is closely related to the logarithmic (or Darboux) derivative of a smooth map to a Lie group, and thus to the pullback of the Maurer--Cartan form on $\SL(2, \C)$ \cite[page 311]{HiNe}.  In \cite{Hit} and subsequent papers, Hitchin constructed holomorphic curves in quasihomogeneous threefolds of $\SL(2, \C)$ and used the Maurer--Cartan pullback to produce meromorphic connections on the Riemann sphere, in order to build solutions to isomonodromic deformation problems and the Painlev\'e equations.  Our approach here is in the opposite direction---we use properties of our derivative to constrain holomorphic curves---although we hope that some of our ideas may be applicable to the study of related isomonodromic deformations.

\begin{defn} \label{labGrpDer} Let $u \mc \C\P^1 \rightarrow X_C$ be a parametrised holomorphic curve not contained in $Y_C$, so it has isolated poles.  The \emph{group derivative} $\D u$ is the meromorphic $\mathfrak{sl}(2, \C)$-valued function on $\C\P^1$ defined as follows.  For $p\in \C\P^1 \setminus (u^{-1}(Y_C) \cup \{\infty\})$ pick a lift $B$ of $u$ to $\SL(2, \C)$ on an open neighbourhood $U$ of $p$, and define $\D u|_U$ to be $B'B^{-1}$ (where $'$ denotes $\pd/\pd z$, and $z$ is our coordinate on $\C\subset \C\P^1$).  If $B_1$ and $B_2$ are two different lifts of $u$ on $U$ then there exists a locally constant map $M \mc U \rightarrow \Gamma_C$ such that $B_2 = B_1 M$, so then $B_2'B_2^{-1} = B_1'MM^{-1}B_1^{-1} = B_1'B^{-1}$.  Therefore $\D u$ is well-defined on $U$ and these local definitions glue together to give a holomorphic map $\C\P^1 \setminus (u^{-1}(Y_C) \cup \{\infty\}) \rightarrow \mathfrak{sl}(2, \C)$.

For $p \in u^{-1}(Y_C) \cap \C$ we can compose $u$ with an appropriate local multiple cover $\psi$ near $p$ so that it lifts to a holomorphic map $\tilde{B}$ from a punctured neighbourhood of $p$ to $\SL(2, \C)$.  By \protect \MakeUppercase {L}emma\nobreakspace \ref {labMatLim}\ref{EndLimitm2}, the components of $\tilde{B}$ are meromorphic over $0$, and hence the components of $\tilde{B}'\tilde{B}^{-1} = \psi'\cdot(\D u\circ \psi)$ are meromorphic over $p$.  Thus $\D u$ itself has at worst a pole at $p$.  To see that $\D u$ is meromorphic over $\infty$, simply make a change of coordinate $w=1/z$ and use the chain rule and the fact that the group derivative of this reparametrised curve is meromorphic over $0$.

If $u$ is also non-constant, so that $\D u$ is not identically zero, we get a holomorphic map $[\D u] \mc \C\P^1 \rightarrow \P\mathfrak{sl}(2, \C) \cong \C\P^2$---the \emph{projectivised group derivative}.  If $\psi$ is an automorphism of $\C\P^1$ then $\D (u\circ \psi)=\psi'\cdot(\D u\circ \psi)$ so $[\D (u\circ \psi)]=[\D u]\circ \psi$.
\end{defn}

Note that by construction we have $u'=\D u \cdot u$.

The group derivative is easily understood at quasi-axial poles:

\begin{lem} \label{labGrpDerRes} Suppose $u \mc \C\P^1 \rightarrow X_C$ is a holomorphic curve not contained in $Y_C$, which has a pole of type $\xi$ and order $k$ at the point $a \in \C \subset \C\P^1$.  So near $a$ there exists a holomorphic map $A$ to $\SL(2, \C)$ such that $u$ is given locally by
\[
u(z) = A(z)e^{-i k \xi \log (z-a)} \cdot C.
\]
Then $\D u$ has a simple pole at $a$ with residue $\Res_a \D u=-ikA(a)\xi A(a)^{-1}$.
\end{lem}
\begin{proof}  This is a straightforward explicit computation: we have near $a$ that
\begin{align*}
\D u &= \Big(Ae^{-ik\xi \log (z-a)}\Big)'e^{ik \xi \log(z-a)}A^{-1}
\\ &= A'A^{-1} -\frac{ik}{z-a} A \xi A^{-1},
\end{align*}
and $A'A^{-1}$ and $A \xi A^{-1}$ are both regular at $a$, so the result follows immediately.
\end{proof}

It also has the following properties:

\begin{lem} \label{labGrpDerProps}  Let $u \mc \C\P^1 \rightarrow X_C$ be a holomorphic curve not contained in $Y_C$, and $\psi$ an antiholomorphic involution of $\C\P^1$.
\begen
\item \label{grpitm1} If $u$ intertwines $\psi$ with the antiholomorphic involution $\tau$ on $X_C$ (or $X_C \setminus Y_C$) then
\[
\D u \circ \psi = -\Big(\conj{\psi}'\circ \psi \Big) \cdot \D u^\dagger
\]
as meromorphic maps from $\C\P^1$ to $\mathfrak{sl}(2, \C)$.  Here $\dag$ denotes conjugate transpose as usual, whilst the derivative of $\conj{\psi}$ is computed by viewing it as a meromorphic function on $\C\P^1$.  In particular, if $\psi$ is the reflection $c$ in the equator then
\[
\D u \circ c = \conj{z}^2 \D u^\dagger.
\]
\item \label{grpitm3} If $u$ intertwines $c$ and $\tau$, and has a quasi-axial pole at $a \in \C^* \subset \C\P^1$, then
\[
\Res_{1/\conj{a}} \D u = -(\Res_a \D u)^\dagger.
\]
\item \label{grpitm2} If $u$ is quasi-axial and non-constant, with $n \geq 2$ poles, then either $[\D u]$ has degree $n-2$ or the image of $u$ is contained in a linear subspace of $\P S^dV$ of dimension less than $\deg u$.
\end{enumerate}
\end{lem}
\begin{proof} \ref{grpitm1} Throughout the proof the notation ${}^{-1}$ will always denote an inverse matrix, rather than inverse function.  Since both sides are antiholomorphic (away from their poles) it suffices to prove the result on the dense open set of points $p$ for which $p$ and $\psi(p)$ are not in $u^{-1}(Y_C) \cup \{\infty\}$, so fix such a $p$.  Near $p$ we can lift $u$ to some holomorphic map $B$ to $\SL(2, \C)$ and then $B^\ddag \circ \psi$ lifts $u$ near $\psi(p)$.  We then have near $\psi(p)$ that
\begin{equation}
\label{eqDuLine1}
\D u=\big( B^\ddag \circ \psi \big)' \big( B^\ddag \circ \psi \big)^{-1} = \big( B^\ddag \circ \psi \big)' \big( B \circ \psi \big)^\dag.
\end{equation}
Letting $g$ denote complex conjugation on $\C\P^1$, the chain rule gives
\[
\big( B^\ddag \circ \psi\big)' = \frac{\pd}{\pd z}\big(B^\ddag \circ g \circ \conj{\psi}\big) = \conj{\psi}' \cdot \lb\frac{\pd B^\ddag \circ g}{\pd z} \circ \conj{\psi}\rb = \conj{\psi}' \cdot \lb\frac{\pd B^\ddag}{\pd \conj{z}} \circ \psi\rb,
\]
and therefore
\[
\big(B^\ddag \circ \psi\big)' = \conj{\psi}' \cdot \lb \frac{\pd B^{-1}}{\pd z} \circ \psi\rb^\dag = -\conj{\psi}' \cdot \lb \big(B^{-1} B' B^{-1}\big) \circ \psi\rb^\dag.
\]
Plugging this into \eqref{eqDuLine1} we get
\[
\D u = -\conj{\psi}' \cdot \big(\big(B'B^{-1}\big)^\dagger\circ \psi\big),
\]
and composing both sides with $\psi$ (which is an involution) gives the first result.  Reflection in the equator is given by $z \mapsto 1/\conj{z}$, and the second result follows from an easy calculation.

\ref{grpitm3} Simply apply the previous part to get
\begin{align*}
\Res_{1/\conj{a}}{\D u} &= \lim_{z \rightarrow 1/\conj{a}} (z-1/\conj{a}) \cdot \D u(z)
\\ &= \lim_{z \rightarrow a} (1/\conj{z}-1/\conj{a}) \cdot \D u \circ c(z)
\\ &= \lim_{z \rightarrow a} -\lb \frac{z(z-a)}{a} \cdot \D u (z)\rb^\dagger
\\ &= - ( \Res_a \D u )^\dagger.
\end{align*}

\ref{grpitm2} Suppose that the image of $u$ is not contained in a linear subspace of dimension less than $\deg u$.  Then we can choose homogeneous coordinates on $\P S^d V$ in which $u$ is given by
\[
[1: z: z^2: \dots : z^{\deg u}:0 : \dots :0]
\]
and it is easy to check that $u$ is an immersion (it's even an embedding: the rational normal curve in the subspace it spans).  Reparametrising $u$ if necessary, we may also assume that $\infty$ is not a pole.  Since $u'=\D u\cdot u$, the fact that $u$ is an immersion ensures that $\D u$ has no zeros in $\C\P^1 \setminus \{\infty\}$, and by a change of coordinate $w=1/z$ we see that $\D u$ vanishes to order $2$ at $\infty$.  By \protect \MakeUppercase {L}emma\nobreakspace \ref {labGrpDerRes}, $\D u$ has a simple pole at each pole of $u$.

Therefore $\D u$ has $n$ poles of order $1$, at $a_1, \dots, a_n \in \C$ say, and a single zero of order $2$, at $\infty$.  So the components of $(z-a_1)\dots(z-a_n)\D u$, with respect to an arbitrary basis of $\mathfrak{sl}(2, \C)$, are polynomials $f_1$, $f_2$ and $f_3$ in $z$ such that $\gcd \{f_1, f_2, f_3\}=1$ and $\max_i \deg f_i = n-2$.  Reordering our basis if necessary, we may assume that $\deg f_1 = n-2$, then $[\D u]$ intersects the line in $\P \mathfrak{sl}(2, \C)$ given by the vanishing of the first component with total multiplicity $n-2$.  Hence $\deg [\D u] = n-2$.
\end{proof}

\subsection{Partial indices and transversality}
\label{sscParInd}

Recall that a Riemann--Hilbert pair $(E, F)$ comprises a holomorphic rank $n$ vector bundle $E$ over the disc $D$, along with a smooth totally real rank $n$ subbundle $F$ over the boundary $\pd D$.  By a result of Oh \cite[Theorem I]{Oh} (following Vekua \cite{Vek} and Globevnik \cite[Lemma 5.1]{Glob}), such a pair can be split as a direct sum $(E, F) \cong \oplus_{i=1}^n (E_i, F_i)$ of rank $1$ Riemann--Hilbert pairs (where the $\cong$ indicates an isomorphism of holomorphic bundles over $D$ preserving the subbundles over $\pd D$), and for each $i$ there exists a partial index $\kappa_i \in \Z$ and a holomorphic trivialisation of $E_i$ in which the fibre of $F_i$ at the point $z \in \pd D$ is given by $z^{\kappa_i/2}\R \subset \C$.  See \cite[Section 2]{EL1} for a fuller discussion, on which our treatment is based.

There is a Cauchy--Riemann operator $\conj{\pd}$, taking smooth sections of $E$ which lie in $F$ when restricted to $\pd D$ to $E$-valued $(0, 1)$-forms on $D$; strictly we need to pass to appropriate Sobolev completions to do the analysis, but this will not concern us.  For rank $1$ pairs in the standard form $(\C, z^{\kappa/2}\R)$, with $\kappa \geq 0$, we can explicitly write down the kernel of $\conj{\pd}$:
\[
\ker \conj{\pd} = \Bigg\{\sum_{r=0}^\kappa a_r z^r : a_r \in \C \text{ with } a_r=\conj{a}_{\kappa-r} \text{ for all } r \Bigg\}.
\]
For example, for $\kappa=0$ the only solutions are real constants, for $\kappa=1$ a basis is given by $(1+z)$ and $i(1-z)$, whilst for $\kappa=2$ a basis is given by $z$, $(1+z^2)$ and $i(1-z^2)$.  More generally, for any integer $\kappa$ we have
\[
\dim \ker \conj{\pd} = \max \{\kappa + 1, 0\} \text{ and } \dim \coker \conj{\pd} = \max \{-(\kappa+1), 0\},
\]
and hence the index of $\conj{\pd}$ is $\kappa+1$.  Note that all dimensions here are over $\R$---the boundary condition imposed by $F$ means that the spaces of sections involved do not have natural complex structures.  Returning to the case of a general Riemann--Hilbert pair $(E, F)$, of arbitrary rank, the operator $\conj{\pd}$ splits into operators $\conj{\pd}_i$ of index $\kappa_i + 1$ on each summand $(E_i, F_i) \cong (\C, z^{\kappa_i/2}\R)$, so the total $\conj{\pd}$-operator has index $n + \sum_i \kappa_i$.

Note that the fibrewise $\C$-linear span of the elements of $\ker \conj{\pd}$ is precisely the span of the summands $E_i$ of non-negative partial index (similarly, the $\R$-linear span of their boundary values is the span of the corresponding $F_i$).  More generally, if we fix an integer $\kappa$ and consider the Riemann--Hilbert pair $(E, z^{-\kappa/2}F)$ then the fibrewise span of the elements of $\ker \conj{\pd}$ for this pair is the span of the summands of the \emph{original} pair whose partial indices are at least $\kappa$.  In this way, we see that the filtration
\[
\cdots \supset \bigoplus_{i : \kappa_i \geq -1} (E_i, F_i) \supset \bigoplus_{i : \kappa_i \geq 0} (E_i, F_i) \supset \bigoplus_{i : \kappa_i \geq 1} (E_i, F_i) \supset \cdots
\]
is uniquely determined by $(E, F)$, and hence so are the tuple of partial indices and the spaces
\[
\bigoplus_{i : \kappa_i \geq \kappa} (E_i, F_i) \bigg/ \bigoplus_{i : \kappa_i \geq \kappa+1} (E_i, F_i).
\]
This filtration was exploited by Evans--Lekili in the proof of \cite[Lemma 3.12]{EL1}.  However, the span of the summands of a given partial index is not determined in general: consider for example $(E, F) = (\C^2, \R \oplus z^{1/2}\R)$, which has one obvious splitting by the natural basis $e_1$ and $e_2$ of $\C^2$, but can in fact be split by the basis
\[
e_1 + \lb a(1+z)+bi(1-z)\rb e_2 \text{ and } e_2,
\]
for any real numbers $a$ and $b$.

Given a holomorphic disc $u \mc (D, \pd D) \rightarrow (X_C, L_C)$, there is an associated rank $3$ Riemann--Hilbert pair $(E, F)=(u^*TX_C, u|_{\pd D}^* TL_C)$, and we refer to the partial indices of this pair as the partial indices of $u$.  It is easy to see directly from the definitions that the sum of the partial indices of $u$ is its Maslov index $\mu(u)$.  The disc $u$ is regular if and only if $\coker \conj{\pd} = 0$, i.e.~if and only if all of its partial indices are at least $-1$.  In this case the moduli space
\[
\widetilde{\mathcal{M}}_0([u])
\]
of unmarked parametrised holomorphic discs in the same homology class as $u$ is a smooth manifold near $u$, of the correct dimension, with tangent space
\[
T_u \widetilde{\mathcal{M}}_0([u]) = \ker \conj{\pd}.
\]
Evans--Lekili \cite[Lemma 2.11 and Lemma 3.2]{EL1} showed using homogeneity that in fact all $\kappa_i$ are non-negative, an argument which we review shortly.

Our motivation for analysing partial indices is to prove transversality results for various evaluation maps on moduli spaces of discs.  In particular, we are interested in showing that the maps $\ev_2 \mc M_4 \rightarrow L_C^2$ and $\ev_1^\mathrm{int} \mc M_4^{\mathrm{int}} \rightarrow X_C$ (as defined in Section\nobreakspace \ref {sscModAndEv}) are submersions at certain points.  To answer these questions we pull back $\ev_2$ and $\ev_1^{\mathrm{int}}$ under the (surjective) projections
\[
\coprod_{\{A \in H_2(X_C, L_C) : \mu(A)= 4\}} \widetilde{\mathcal{M}}_0(A) \rightarrow M_4, \quad u \mapsto [u, -1, 1]
\]
and
\[
\coprod_{\{A \in H_2(X_C, L_C) : \mu(A)= 4\}} \widetilde{\mathcal{M}}_0(A) \rightarrow M_4^{\mathrm{int}}, \quad u \mapsto [u, 0],
\]
which allows us to work with moduli spaces of parametrised discs with fixed marked points, namely $-1$ and $1$ in the first case and $0$ in the second.  Using this simplification, it is easy to see from the explicit form of $\ker \conj{\pd}$ above that $\ev_2$ and $\ev_1^\mathrm{int}$ are submersions at a parametrised disc $u$ if and only if all partial indices of $u$ are at least $1$ (cf.~\cite[Lemma 2.12]{EL1}), and in fact the positions of the marked points, which we chose to be $\pm 1$ and $0$, are irrelevant.

Given a holomorphic disc $u$ and a meromorphic map $\xi$ from $D$ to $\mathfrak{sl}(2, \C)$, let $\xi \cdot u$ denote the meromorphic section of $E=u^*TX_C$ defined by $z \mapsto \xi(z) \cdot u(z) \in T_{u(z)}X_C$.  For any $u$ and any basis $\alpha$, $\beta$, $\gamma$ of $\mathfrak{su}(2)$ we then have holomorphic sections $\alpha \cdot u$, $\beta \cdot u$ and $\gamma \cdot u$ of $u^*TX_C$ which form a global frame for $F=u|_{\pd D}^*TL_C$ when restricted to $\pd D$.  In particular, the fibrewise $\R$-linear span of the boundary values of the elements of $\ker \conj{\pd}$ is the whole of $F$, which shows that all partial indices are non-negative.  This is roughly the argument used by Evans--Lekili.

Before looking at index $4$ discs we warm up by considering an index $2$ disc:

\begin{lem} \label{labInd2Pars} The partial indices of an index $2$ disc $u$ are $0$, $0$ and $2$.
\end{lem}
\begin{proof} We have seen that $u$ is axial of type $\xi_v$, so (up to reparametrisation) is of the form $z \mapsto A e^{-i \xi_v \log z} \cdot C$ for some $A \in \SU(2)$.  Acting by $A^{-1}$, which clearly doesn't change the isomorphism class of the corresponding Riemann--Hilbert pair $(E, F)$, we may assume that in fact $A$ is the identity.  The infinitesimal action of $\mathfrak{sl}(2, \C)$ at $u(z)$ is surjective except at $u(0) \in Y_C \setminus N_C$, where it has rank $2$, with kernel spanned by $\xi_v$.  If we take a basis $\alpha$, $\beta$, $\gamma = \xi_v$ of $\mathfrak{su}(2)$, we therefore have holomorphic sections $\alpha \cdot u$, $\beta \cdot u$ and $\gamma \cdot u$ of $E$ which are $\C$-linearly independent everywhere except at $0$, where $\gamma \cdot u$ vanishes.

By viewing the Maslov index of $u$ as twice the intersection with an anticanonical divisor, as in \protect \MakeUppercase {L}emma\nobreakspace \ref {labIndInt}, we have that
\[
(\alpha \cdot u) \wedge (\beta \cdot u) \wedge (\gamma \cdot u)
\]
vanishes to order $1$ at $0$.  Therefore
\[
v_1 \coloneqq \alpha \cdot u \text{, } v_2 \coloneqq \beta \cdot u \text{ and } v_3 \coloneqq \frac{\gamma}{z} \cdot u
\]
are linearly independent on the whole of $D$, including $0$, so form a holomorphic global frame for $E$.  Moreover they induce a splitting of the Riemann--Hilbert pair, meaning that for each $i$ the holomorphic line bundle spanned by $v_i$ meets $F$ in a real line bundle over $\pd D$.  One can immediately read off that $F$ is given by the real span of $v_1$, $v_2$ and $zv_3$ (restricted to $\pd D$), so the partial indices are $0$, $0$ and $2$.
\end{proof}

We note in passing that the index $2$ subbundle is the tangent bundle to the disc, generated by reparametrisations.  To see this, note that $u$ lifts to $e^{-i \xi_v \log z}$ away from zero, so
\[
\D u = -\frac{i \xi_v}{z},
\]
and hence $v_3 = i\D u \cdot u=iu'$ (strictly we defined the group derivative for closed curves but clearly we could have just as well defined it for discs; alternatively we could work with the double $\double{u}$ of our disc).

Next we consider axial index $4$ discs.  Such discs are either of type $\xi_v$ and order $2$ or of type $\xi_f$ and order $1$.  In the former case (which we are not really concerned with) the disc is a double cover of an axial index $2$ disc, so the partial indices are $0$, $0$ and $4$ from \protect \MakeUppercase {L}emma\nobreakspace \ref {labInd2Pars}.  In contrast, for the latter case we have the following result:

\begin{lem} \label{labInd4AxPars}  If $u$ is an axial index $4$ disc of type $\xi_f$ (meaning that its pole is of type $\xi_f$) then the partial indices of $u$ are $1$, $1$ and $2$.
\end{lem}
\begin{proof}  Now we may assume that $u$ is of the form $z \mapsto e^{-i \xi_f \log z} \cdot C$, so the infinitesimal action of $\mathfrak{sl}(2, \C)$ has rank $1$ at $u(0) \in N_C$ (but is surjective at $u(z)$ for all non-zero $z$).  The kernel at $u(0)$ is spanned by $\gamma = \xi_f$ and by $\alpha + i\beta$, where $\alpha$, $\beta$, $\gamma$ is a basis of $\mathfrak{su}(2)$ corresponding to infinitesimal right-handed rotations about a right-handed set of orthogonal axes.  Now $(\alpha + i \beta)\cdot u$ and $\gamma \cdot u$ both vanish at $0$, and similar Maslov index considerations show that the sections
\[
\frac{\alpha+i \beta}{z} \cdot u\text{, }(\alpha-i\beta)\cdot u\text{ and } \frac{\gamma}{z} \cdot u
\]
form a holomorphic global frame for $E$.  However, they do not induce a splitting of the Riemann--Hilbert pair.  Instead we must take linear combinations of the first two in order to get a frame which interacts well with $F$:
\[
v_1 \coloneqq \frac{(1+z)\alpha+i(1-z)\beta}{z} \cdot u\text{, }v_2 \coloneqq \frac{i(1-z)\alpha-(1+z)\beta}{z} \cdot u \text{ and } v_3 \coloneqq \frac{\gamma}{z} \cdot u.
\]
Now $F$ is the real span of $z^{1/2}v_1$, $z^{1/2}v_2$ and $zv_3$, so the partial indices are $1$, $1$ and $2$.
\end{proof}

Analogously to the index $2$ case we have that the index $2$ subbundle is the tangent bundle to $u$.

Now we deal with non-axial index $4$ discs:

\begin{lem} \label{labInd4Pars} Suppose $u$ is a non-axial index $4$ disc whose double $\double{u}$ is not contained in a linear subspace of dimension less than $\deg \double{u}$.  If $[\D \double{u}]$ is not a double cover of a line then $u$ has partial indices $1$, $1$, and $2$.
\end{lem}
\begin{proof}  By \protect \MakeUppercase {C}orollary\nobreakspace \ref {labInd4Pol} $u$ must have two poles of type $\xi_v$ and order $1$, so they evaluate to $Y_C \setminus N_C$ and $u'$ is transverse to $Y_C$ there.  The infinitesimal action of $\mathfrak{sl}(2, \C)$ at the poles has rank $2$ and the kernel is spanned by the residue of $\D u$, otherwise $u'$ would blow up.  After reparametrising we may assume that one of the poles is at $0$, where the residue of $\D u$ is $\xi_R + i \xi_I$, with $\xi_R$ and $\xi_I$ in $\mathfrak{su}(2)$.  We then have that
\[
v_1 \coloneqq \frac{(1+z)\xi_R + i(1-z)\xi_I}{z} \cdot u \text{, } v_2 \coloneqq \frac{i(1-z)\xi_R - (1+z)\xi_I}{z} \cdot u \text{ and } v_3 \coloneqq i \D u \cdot u = iu'
\]
define holomorphic sections of $E$.

By \protect \MakeUppercase {L}emma\nobreakspace \ref {labGrpDerProps}\ref{grpitm1} $[\D \double{u}]$ takes the value $[\xi_R - i \xi_I]$ at $\infty$; we use the notation $\D \double{u}$ when we wish to emphasise that we are thinking about the group derivative of the double $\double{u}$, rather than just the hemisphere $\D u$ coming from $u$ itself.  We know from \protect \MakeUppercase {L}emma\nobreakspace \ref {labGrpDerProps}\ref{grpitm2} that $[\D \double{u}]$ has degree $2$, so assuming it is not a double cover of a line we deduce that it is a smooth conic.  In particular, $\xi_R$ and $\xi_I$ are linearly independent---otherwise $[\D u(0)]$ would be equal to $[\D u(\infty)]$---and $\D \double{u}$ meets the line $\lspan{\xi_R, \xi_I}$ they span in $\P \mathfrak{sl}(2, \C)$ at $0$ and $\infty$ only.  Moreover, these two intersections are transverse.  The latter means that if $\D \double{u}$ has Laurent expansion
\[
\D \double{u} (z) = \frac{\xi_R + i\xi_I}{z} + \eta + \dots
\]
about $0$, with $\eta$ in $\mathfrak{sl}(2, \C)$, then $\eta$ (which generates the tangent direction to $[\D \double{u}]$ at $0$) is linearly independent of $\xi_R$ and $\xi_I$.  Therefore $v_1$, $v_2$ and $v_3$ are linearly independent in the fibre over $0$: if $V$ denotes $\lim_{z \rightarrow 0} (\xi_R + i \xi_I)/z \cdot u(z)$ then
\[
v_1(0) = V + (\xi_R - i \xi_I) \cdot u(0) \text{, } v_2(0) = iV - i(\xi_R - i \xi_I) \cdot u(0) \text{ and } v_3(0) = V + \eta \cdot u(0).
\]

At all points $z \in D$ that are not poles, the $v_j$ are again linearly independent since $\xi_R$, $\xi_I$ and $i \D u(z)$ span the whole of $\mathfrak{sl}(2, \C)$ ($\D u(z)$ is non-zero by the proof of \protect \MakeUppercase {L}emma\nobreakspace \ref {labGrpDerProps}\ref{grpitm2}, so it spans the line $[\D u(z)]$, and we have seen that the latter is not contained in $\lspan{\xi_R, \xi_I}$) and hence their infinitesimal action generates the fibre of $E$.  And at the other pole $a$ of $u$, $\xi_R$ and $\xi_I$ are linearly independent of $[\D u(a)]$, which generates the kernel of the infinitesimal action of $\mathfrak{sl}(2, \C)$, so $v_1(a)$ and $v_2(a)$ span $T_{u(a)}Y_C$.  Since $v_3(a)$ is transverse to this subspace, we see that the $v_j$ are also linearly independent in the fibre over $a$.

We conclude that $v_1$, $v_2$ and $v_3$ form a holomorphic frame for $E$.  It is easy to see that the boundary bundle $F$ is the real span of $z^{1/2}v_1$, $z^{1/2}v_2$ and $zv_3$, and hence the partial indices are $1$, $1$ and $2$.
\end{proof}

We can now deduce the result we want:

\begin{cor} \label{labInd4Trans} If $u \mc (D, \pd D) \rightarrow (X_C, L_C)$ is an index $4$ holomorphic disc which is either axial of type $\xi_f$ or is non-axial but satisfies the hypotheses of \protect \MakeUppercase {L}emma\nobreakspace \ref {labInd4Pars}, then $\ev_2$ and $\ev_1^\mathrm{int}$ are both submersions at $u$ (irrespective of the positions of the marked points).
\end{cor}
\begin{proof}
We remarked earlier that it is enough to show that the partial indices of $u$ are all at least $1$, and this follows from \protect \MakeUppercase {L}emma\nobreakspace \ref {labInd4AxPars} or \protect \MakeUppercase {L}emma\nobreakspace \ref {labInd4Pars}.
\end{proof}

\subsection{Degree control}
\label{sscDegCont}

In Section\nobreakspace \ref {sscAntInv} we saw that every holomorphic disc $u \mc (D, \pd D) \rightarrow (X_C, L_C)$ extends to a sphere $\double{u}$.  It is important for us to have control over the degree of this sphere, in terms of the index of the original disc.

For $C$ equal to $O$ or $I$, the antiholomorphic involution $\tau$ extends to the whole of $X_C$ and for any disc $u$ we have $\mu(\double{u}) = 2 \mu(u)$, since intersections of $\double{u}$ with $Y_C$ inside $D$ pair up with their reflections outside $D$, and each member of the pair has the same intersection multiplicity.

The situation is more complicated for $C$ equal to $\tri$ or $T$ so in these cases we restrict our attention to quasi-axial discs.  We have seen that these are particularly well-behaved and are really all that we need to understand.  So suppose that $u$ is a quasi-axial holomorphic disc with $n$ poles at points $a_1, \dots, a_n$ of orders $k_1, \dots, k_n$.  The poles can be divided into four classes: those of type $\xi_v$, type $\xi_e$, type $\xi_f$, and generic.  Let the poles in the $\xi_v$ class be those at $\{a_i : i \in I_v\}$, and similarly let $I_e$, $I_f$ and $I_g$ index the poles in the $\xi_e$, $\xi_f$ and generic classes.  We then have, using \protect \MakeUppercase {E}xample\nobreakspace \ref {labAxPol}, that
\[
\mu(u)=\sum_i \mu_{a_i}(u)=2\sum_{i \in I_v} k_i + 6\sum_{i \in I_e} k_i + 4\sum_{i \in I_f} k_i + 12\sum_{i \in I_g} k_i.
\]

By \protect \MakeUppercase {L}emma\nobreakspace \ref {labRigQA}\ref{Rigitm3} we have that poles of type $\xi$ and order $k$ reflect to poles of type $-\xi$ and order $k$.  For $C=\tri$ we have $-\xi_v = \xi_e$ and $-\xi_f = \xi_f$ up to conjugation by $\Gamma_C$ (and generic class poles remain in this class under reflection) so
\[
\mu(\double{u})=8\sum_{i \in I_v} k_i + 8\sum_{i \in I_e} k_i + 8\sum_{i \in I_f} k_i + 24\sum_{i \in I_g} k_i.
\]
Similarly, for $C=T$ we have $-\xi_v = \xi_f$ and $-\xi_e = \xi_e$, so
\[
\mu(\double{u})=6\sum_{i \in I_v} k_i + 12\sum_{i \in I_e} k_i + 6\sum_{i \in I_f} k_i + 24\sum_{i \in I_g} k_i.
\]
In particular, if all poles of $u$ are of type $\xi_v$ (or, equivalently, $u$ doesn't hit $N_C$) then we have $\mu(\double{u})=4\mu(u)$ for $\tri$ and $\mu(\double{u})=3\mu(u)$ for $T$.  Recall that by \protect \MakeUppercase {L}emma\nobreakspace \ref {labVertPol} the condition that $u$ doesn't hit $N_C$ in fact automatically forces it to be quasi-axial.

In order to translate this information about Maslov index into control over degree, we use the fact that $\mu$ restricted to $H_2(X_C)$ is just $2c_1(X_C)=2l_C H$, where $H \in H^2(X_C; \Z)$ is the hyperplane class and $l_C$ is $4$, $3$, $2$ and $1$ for $C$ equal to $\tri$, $T$, $O$ and $I$.  Explicitly, we have
\[
\deg (\double{u}) = H \cdot [\double{u}] = \frac{\ip{c_1(X_C)}{[\double{u}]}}{l_C}=\frac{\mu(\double{u})}{2l_C}.
\]
Table\nobreakspace \ref {tabDegCont} gives the resulting degrees.

\begin{center}
\vspace{0.15cm}
\makebox[2.5cm]{}
\begin{tabular}{cccc}
\addlinespace[\abovetopsep]\cmidrule[\heavyrulewidth]{1-3}
$C$ & $\mu(\double{u})/\mu(u)$ & $\deg (\double{u}) / \mu(u)$ & \\ \midrule
$\tri$ & $4$ & $1/2$ & \multirow{2}{*}{\parbox{2.7cm}{\centering Assuming $u$ does not hit $N_C$}} \\
$T$ & $3$ & $1/2$ & \\ \midrule[\cmidrulewidth]
$O$ & $2$ & $1/2$ & \\
$I$ & $2$ & $1$ & \\
\cmidrule[\heavyrulewidth]{1-3}\addlinespace[-\belowrulesep]\addlinespace[\belowbottomsep]
\end{tabular}
\vspace{0.15cm}
\captionof{table}{Degree control for doubles of holomorphic discs $u$ on $L_C$.}
\label{tabDegCont}
\end{center}

\subsection{The closed--open map I}
\label{sscCOI}

Recall that for a closed, connected, monotone, oriented, spin Lagrangian $L$ in a closed symplectic manifold $X$, and a ring $R$, there is a unital $\Z/2$-graded ring homomorphism, the closed--open string map
\[
\CO \mc QH^*(X; R) \rightarrow HF^*(L, L; R).
\]
We work with the pearl model of the codomain, and assume that discs in pearly trajectories are parametrised so that incoming flowlines enter discs at $-1$ and outgoing flowlines exit at $1$.  Then to compute $\CO$ on a given class $\phi \in QH^*(X; R)$, we fix a Poincar\'e dual cycle $\PD(\phi)$ and count rigid pearly trajectories from a chosen Morse cocycle representing the unit $1_L \in HF^*(L, L; R)$ to arbitrary critical points $y$, in which one of the discs maps $0$ to $\PD(\phi)$.  From now on we will assume that our Morse function on $L$ is chosen so that it has a unique minimum, which then represents the unit (we can do this since $L$ is connected).

For each $y$, the moduli space of ordinary pearly trajectories from the minimum to $y$, of index $\mu$ (i.e.~such that the sum of the indices of the discs is $\mu$), has virtual dimension $|y|-1+\mu$, where $|y|$ is the degree (Morse index) of $y$.  In this formula the reparametrisations of the discs fixing $\pm 1$ have been quotiented out, but when we introduce the interior marked point at $0$ in one of the discs we lose the freedom we had to reparametrise it, so the dimension increases by $1$.  Intersecting with $\PD(\phi)$ then cuts down the dimension by $|\phi|$, so the moduli space of `pearly trajectories with a disc mapping $0$ to $\PD(\phi)$' has virtual dimension $|y|+\mu-|\phi|$.  Therefore trajectories contributing to the coefficient $y$ in $\CO(\phi)$ have total index $\mu = |\phi|-|y|$.

For details of this construction and its properties see \cite[Section 2.5]{BCLQH} and \cite[Theorem 2.1.1(ii)]{BCQS}.  These papers describe a multiplication
\[
QH^i(X; R) \otimes HF^j(L, L; R) \rightarrow HF^{i+j}(L, L; R),
\]
making $HF^*(L, L; R)$ into a two-sided algebra over $QH^*(X; R)$ (meaning elements of $QH^*(X; R)$ graded-commute with everything in $HF^*(L, L; R)$), and the map $\CO$ is simply `multiplication by the unit $1_L$'.

There are several superficial differences between the approach taken by Biran--Cornea and the one we use here, which we briefly mention.  Firstly, they work over a Novikov ring, whereas we simply set the Novikov parameter to $1$ and collapse the grading to $\Z/2$.  Secondly, they work with coefficients modulo $2$, whilst we work over an arbitrary ground ring---the necessary orientation arguments are given in \cite[Appendix A]{BCEG}.  Thirdly, they work with homological grading, instead of the \emph{co}homological version we employ.  The latter simplifies gradings for product operations, and means that we flow \emph{up} the Morse function rather than down.  Fourth, they work directly with `quantum (co)homology' of $L$, which we notationally identify throughout with Floer (co)homology.  And, finally, they use the Morse model for $QH^*(X; R)$, whereas we use the singular model.

As a simple example, $\CO(1_X)$ involves only trajectories of index $0$ (so there is a unique disc and it is constant), with outputs of Morse index $0$.  The cycle $\PD(1_X)$ is the whole of $X$, so the disc is unconstrained, and we just need to count (increasing) Morse trajectories from the minimum of the Morse function to itself.  There is clearly a unique such trajectory---the constant one---and we see that $\CO(1_X)=1_L$.

The Auroux--Kontsevich--Seidel criterion \cite[Proposition 6.8]{Au}, which we now briefly review, is obtained by considering $\CO(c_1(X))$:

\begin{prop}\label{labAKS}  Let $X$ be a closed symplectic manifold, and $L \subset X$ a closed, connected, monotone, oriented and spin Lagrangian.  If the self-Floer cohomology $HF^*(L, L; k)$ is non-zero over a field $k$ then the (signed) count $\mathfrak{m}_0$ of index $2$ discs through a generic point of $L$ is an eigenvalue of quantum multiplication by the first Chern class $c_1(X) \qcup \mc QH^*(X; k) \rightarrow QH^*(X; k)$.
\end{prop}
\begin{proof}
Since $L$ is orientable its Maslov class $\mu \in H^2(X, L; \Z)$ is divisible by $2$, and by Poincar\'e duality we can pick a cycle $Y \subset X \setminus L$ representing $\mu/2$.  Note that since $\mu/2$ maps to $c_1(X)$ in $H^2(X; \Z)$ the cycle $Y$ represents $\PD(c_1(X))$.  We now compute $\CO(c_1)$ using this representative.

The contributions are either trajectories of index $0$ with outputs of index $2$ or trajectories of index $2$ with outputs of index $0$.  By construction, each index $2$ disc bounded by $L$ has intersection number $1$ with $Y$, so for the index $2$ trajectories we can just ignore the incidence condition with $Y$ and count things of the form `flow up from the minimum of the Morse function, enter an index $2$ disc, and exit at the minimum'.  Generically this amounts to simply counting index $2$ discs through the minimum, of which there are $\mathfrak{m}_0$, so $\CO(c_1)=\mathfrak{m}_0 \cdot 1_L$.  Each index $0$ disc, meanwhile, has intersection number $0$ with $Y$, so all index $0$ trajectories cancel.

We therefore have $\CO(c_1-\mathfrak{m}_0 \cdot 1_X)=0_L$, so if $HF^*(L, L; k)$ is non-zero then $c_1-\mathfrak{m}_0 \cdot 1_X$ cannot be invertible in $QH^*(X; k)$: $\CO$ maps invertibles in $QH(X; k)$ to invertibles in $HF^*(L, L; k)$, and if the latter is non-zero then $0_L$ is not invertible.  So quantum multiplication $(c_1-\mathfrak{m}_0 \cdot 1_X) \qcup$ is singular, and thus $\mathfrak{m}_0$ is an eigenvalue of $c_1 \qcup$.
\end{proof}

Strictly of course one needs to be careful about the genericity of the auxiliary data: the Morse function, metric and almost complex structure.  The machinery of Biran--Cornea (for example \cite[Proposition 3.1.2]{BCLQH}) allows one to fix the first two, and then choose a generic almost complex structure, so if one is using a particular integrable complex structure $J$ to compute $\mathfrak{m}_0$ it is enough to know that it gives the same answer as a generic (not necessarily integrable) one.  And for this it is enough to know that the index $2$ discs are regular for $J$ (i.e.~all of their partial indices are at least $-1$), and that there are no $J$-holomorphic index $2$ spheres passing through our Morse minimum on $L$.  A standard cobordism argument shows that the values of $\mathfrak{m}_0$ calculated using two such almost complex structures agree.

Now consider $X=X_C$, $Y=Y_C$ and $L=L_C$.  If $v_C$ denotes the number of vertices of the configuration $C$ (we called this $d$ before, and shall still continue to use the notation $d$ outside the context of disc counts) then from \protect \MakeUppercase {P}roposition\nobreakspace \ref {labInd2MS} we have $\mathfrak{m}_0=\pm v_C$, so $\CO(c_1)=\pm v_C \cdot 1_{L_C}$ and if $HF^*(L_C, L_C; k)$ is non-zero then $\pm v_C$ is an eigenvalue of $c_1 \qcup$.  This eigenvalue argument was used by Evans--Lekili \cite[Remark 1.2]{EL1} to show that over a field $k$ we can have $HF(L_\tri, L_\tri; k)\neq 0$ only if $\Char k = 5$ or $7$.

We can also compute the value of $\CO$ on the class dual to the curve $N_C$ when $C \neq I$:
\begin{prop} \label{labCONC} If $f_C$ denotes the number of faces of the configuration $C$ (recalling that the triangle $\tri$ is to be thought of as having two faces) then for $C=\tri$, $T$ or $O$ we have $\CO(\PD (N_C))=\pm f_C \cdot 1_{L_C}$.  In these cases, if $HF^*(L_C, L_C; k)\neq 0$ over a field $k$ then $\pm f_C$ is an eigenvalue of $\PD (N_C)\qcup$.
\end{prop}
\begin{proof}
Fix a choice of $C$ from $\tri$, $T$ and $O$, and let $J'$ be a generic compatible almost complex structure on $X_C$.  Recall that $J$ denotes the standard integrable complex structure, and that $X_C$ has first Chern class $4H$, $3H$ and $2H$ for the three configurations $C$ respectively.  By \protect \MakeUppercase {L}emma\nobreakspace \ref {labIntNC} no $J$-holomorphic disc $u \mc (D, \pd D) \rightarrow (X_C, L_C)$ of index $0$ or $2$ meet $N_C$, so the same is true for $J'$-holomorphic discs if $J'$ is sufficiently close to $J$.  Otherwise there would exist a sequence $(J_n)_{n=1}^\infty$ tending to $J$, and a sequence $(u_n)_{n=1}^\infty$ of $J_n$-holomorphic discs of index at most $2$ which hit $N_C$.  By Gromov compactness some subsequence of the $u_n$ would converge to a $J$-holomorphic stable map of index at most $2$ which meets $N_C$, and no such maps exist (there can be no sphere bubbles since the minimal Chern number is greater than $2$, so any bad stable map would contain a disc component of index at most $2$ which hits $N_C$, of which there aren't any).

From \protect \MakeUppercase {C}orollary\nobreakspace \ref {labInd4Pol} and \protect \MakeUppercase {L}emma\nobreakspace \ref {labVertPol} the only $J$-holomorphic index $4$ discs hitting $N_C$ are the axial ones of type $\xi_f$ and order $1$.  By \protect \MakeUppercase {C}orollary\nobreakspace \ref {labInd4Trans} the interior marked point evaluation map $\ev_1^\mathrm{int}$ is transverse to $N_C$ at such discs, so we can form a cobordism from the space of $J$-holomorphic index $4$ discs mapping an interior marked point to $N_C$ to the corresponding space of $J'$-holomorphic discs, by picking a generic homotopy $J_t$ from $J_0 = J$ to $J_1 = J'$.  If this cobordism is compact then the signed count of such $J$-holomorphic discs which also map a boundary marked point to a fixed generic point of $L$ agrees with the $J'$-holomorphic count.

Assuming the homotopy is sufficiently generic, compactness over $t \in (0, 1]$ is ensured by the usual arguments---any bubbling gives rise to an element of a transversely cut out moduli space of negative virtual dimension.  However, at $t=0$ we need to be careful and rule out bubbled configurations by hand, since $J$ is not itself generic.  If discs or spheres bubble off anywhere other than the interior marked point, we can delete the bubbles and obtain a disc of index less than $4$ which has boundary on $L_C$ and meets $N_C$.  We have already seen that no such discs exist, so we are left to consider bubbling of spheres at the marked point.

For the triangle and tetrahedron this is ruled out immediately by the minimal Chern number: there simply are no non-constant $J$-holomorphic spheres of index at most $4$.  For the octahedron, however, index $4$ spheres \emph{do} exist, and a priori may appear as bubbles connecting $N_O$ to a constant `ghost' disc on $L_O$.  Since $c_1(X_O)$ is twice the hyperplane class, such a sphere would be a projective line inside $X_O$, passing through both $N_O$ and $L_O$.  We claim that no lines have this property.

To see this, consider the point $O_v \in L_O$ given in Appendix\nobreakspace \ref {secExplReps}, namely $[x^5y-xy^5] \in \P S^6V$.  The tangent space to $X_O$ at $O_v$ is given by the infinitesimal action of $\mathrm{SL}(2, \C)$.  Note that $\eta_H = \lb\begin{smallmatrix}1 & 0 \\ 0 & -1\end{smallmatrix}\rb \in \mathfrak{sl}(2, \C)$ acts by $x \mapsto x$, $y \mapsto -y$, whilst $\eta_X = \lb\begin{smallmatrix}0 & 1 \\ 0 & 0\end{smallmatrix}\rb$ acts by $x \mapsto 0$, $y \mapsto x$ and $\eta_Y = \lb\begin{smallmatrix}0 & 0 \\ 1 & 0\end{smallmatrix}\rb$ acts by $x \mapsto y$, $y \mapsto 0$.  Using the Leibniz rule we deduce that $T_{O_v}X_O$ is the linear span of
\begin{multline*}
[x^5y-xy^5] \text{, } [\eta_X \cdot (x^5y-xy^5)] = [x^5y+xy^5] \text{, } [\eta_X \cdot (x^5y-xy^5)] = [x^6-5x^2y^4] \text{,} \\ [\eta_Y \cdot (x^5y-xy^5)] = [5x^4y^2-y^6].
\end{multline*}
Points of $N_O$ are of the form $[(ax+y)^6]$ for $a \in \C\P^1$, and a simple calculation (by hand) shows that no such points are contained in this tangent space.  In other words, no line in $\P S^6 V$ through both $O_v$ and $N_O$ is tangent to $X_O$ at $O_v$, let alone contained in $X_O$.  By $\SU(2)$-homogeneity we deduce that the same holds with any other point of $L_O$ in place of $O_v$, and hence that no line in $X_O$ meets both $N_O$ and $L_O$, proving the claim and completing the proof of compactness of the cobordism.

Now consider the quantity $\CO(\PD(N_C))$ we wish to compute.  We use our generic $J'$ close to $J$.  For degree reasons we only need consider trajectories of index $0$, $2$ and $4$, and since there are no index $0$ or $2$ discs hitting $N_C$ we see that $\CO(\PD(N_C))$ has only index $0$ outputs, arising from trajectories containing a single disc of index $4$.  Therefore
\begin{multline*}
\CO(\PD(N_C)) = \# \{ u \mc (D, \pd D) \rightarrow (X, L) : u \text{ is } J' \text{-holomorphic of index } 4 \text{,} \\ \text{with } u(0) \in N_C \text{ and } u(1) = p \} \cdot 1_L,
\end{multline*}
where $p \in L_C$ is the minimum of our Morse function.

We have just seen that this count can be computed using $J$, for which we know explicitly that all index $4$ discs sending an interior marked point to $N_C$ are translates of $u_f$.  Arguing analogously to the case of unconstrained index $2$ discs in \protect \MakeUppercase {P}roposition\nobreakspace \ref {labInd2MS}, the moduli space of such discs is diffeomorphic to $S^2$ (by $\pi_C \circ (\text{evaluate at pole})$ again, although now the $\pi_C$ has no effect), and evaluation at a boundary marked point has degree $\pm f_C$.  The result now follows by arguing as in \protect \MakeUppercase {P}roposition\nobreakspace \ref {labAKS}, with $\PD(N_C)$ in place of $c_1$.
\end{proof}

In the case of the icosahedron, the compactness argument fails because bad bubbled configurations really do exist.  These consist of index $2$ discs with an index $2$ sphere bubble (a projective line) joining the interior marked point, which evaluates to $Y_I$, to $N_I$.  Note that in fact every point of $Y_I$ is connected by a line in $X_I$ to $N_I$.  The basic problem is that when curves or evaluation maps land in the compactification divisor we cannot use the infinitesimal group action to ensure transversality.

In order to apply the preceding result, we need to calculate $\PD (N_C)$:

\begin{lem} \label{labPDNC} We have
\[
\PD (N_C)=v_C E,
\]
where $E$ is the generator of $H^4(X_C)$ as in Section\nobreakspace \ref {sscBasProps}.
\end{lem}
\begin{proof}
In $H^*(X_C)$ we have that $E\smile H$ is Poincar\'e dual to a point, so $\PD(N_C)=n_C E$ where $n_C$ is the intersection number of $N_C$ with a hyperplane section of $X_C$.  Taking our hyperplane section to be the set of $d$-point configurations in $\P V$ containing the point $[y]$, we see that there is a single intersection at $[y^d]$ with multiplicity $d=v_C$.
\end{proof}

We will see later (in \protect \MakeUppercase {C}orollary\nobreakspace \ref {labIHF}) that $HF^0(L_I, L_I; \Z)$ is isomorphic as a ring to $\Z/(8)$, from which it follows that $\CO(\PD(N_I))$ cannot possibly be $\pm f_I = \pm 20$.  Indeed, if this were the case then we would have $\CO(H)=\CO(12E) = 4 \in \Z/(8)$---so in particular $\CO(E)$ would be odd in $\Z/(8)$---and since $\CO$ is a ring homomorphism (all of our ring homomorphisms are implicitly unital) this is incompatible with the relation $E^2 = 2E + H + 4$ in $QH^*(X_I; \Z)$.

We can now put everything together:
\begin{cor} \label{labCO}  We have:
\begen
\item \label{COitm1} The closed--open map (over any ring) satisfies:
\begin{center}
\begin{tabular}{ccc} \toprule
$C$ & $\CO(c_1)$ & $\CO(\PD (N_C))$ \\ \midrule
$\tri$ & $\CO(4H)=\pm 3 \cdot 1_{L_C}$ & $\CO(3E)=\pm 2 \cdot 1_{L_C}$ \\
$T$ & $\CO(3H)=\pm 4 \cdot 1_{L_C}$ & $\CO(4E)=\pm 4 \cdot 1_{L_C}$ \\
$O$ & $\CO(2H)=\pm 6 \cdot 1_{L_C}$ & $\CO(6E)=\pm 8 \cdot 1_{L_C}$ \\
$I$ & $\phantom{12}\CO(H)=\pm 12 \cdot 1_{L_C}$ & \\
\bottomrule
\end{tabular}
\end{center}
\item \label{COitm2} If $HF^*(L_C, L_C; k) \neq 0$ over a field $k$ of characteristic $p$ then $p$ must be $5$ or $2$ for $C$ equal to $\tri$ or $T$ respectively, and $2$ or $19$ for $C=O$.  For $C=I$, $p$ must be $2$, $43$ or $571$.
\end{enumerate}
\end{cor}
\begin{proof}
\ref{COitm1}  This follows from substituting the values of $c_1$, $v_C$ and $f_C$ into the preceding results.

\ref{COitm2}  The constraints on $p=\Char k$ come from eigenvalue considerations.  Explicitly, we have the following characteristic polynomials:

\begin{center}
\begin{tabular}{ccc} \toprule
$C$ & $\chi(c_1\qcup)$ & $\chi(\PD (N_C)\qcup)$ \\ \midrule
$\tri$ & $(\lambda^2-16)(\lambda^2+16)$ & $(\lambda^2-9)^2$ \\
$T$ & $\lambda(\lambda^3-108)$ & $\lambda(\lambda^3-128)$ \\
$O$ & $\lambda^4-44\lambda^2-16$ & $(\lambda^2-6\lambda-36)^2$ \\
$I$ & $(\lambda+4)(\lambda^3-8\lambda^2-56\lambda-76)$ & \\
\bottomrule
\end{tabular}
\end{center}
These are easily computed in Mathematica by using Table\nobreakspace \ref {tabQH} to express $c_1 \qcup$ and $\PD(N_C) \qcup$ in matrix form with respect to the basis $1$, $H$, $E$, $HE$ of $QH^*(X_C; k)$.  We know that for each configuration $C$, $\pm v_C$ is a root of the first polynomial (over $k$) and $\pm f_C$ is a root of the second (except when $C=I$).  Concretely this means that
\[
p \mid \chi(c_1\qcup)(v_C) \text{ or } \chi(c_1\qcup)(-v_C)
\]
and that
\[
p \mid \chi(\PD(N_C)\qcup)(f_C) \text{ or } \chi(\PD(N_C)\qcup)(-f_C).
\]

For $C=\tri$ the second condition forces $p=5$.  For $C=T$ the second condition implies $p=2$ or $3$, but $p=3$ is ruled out by the first.  For $C$ equal to $O$ or $I$ one just has to plug $\pm 6$ (respectively $\pm 12$) into $\chi(c_1\qcup)$ and see that the only prime factors appearing are $2$ and $19$ (respectively $2$, $43$ and $571$).
\end{proof}

The results for $O$ and $I$ are essentially just the Auroux--Kontsevich--Seidel criterion.  However, we will see later in \protect \MakeUppercase {P}roposition\nobreakspace \ref {RuleOutChars} that we can exploit the antiholomorphic involution to rule out $p=43$ and $571$ for the icosahedron, and in \protect \MakeUppercase {L}emma\nobreakspace \ref {labCOSign} that we actually only need consider the positive sign for $f_C$ in $\chi(\PD(N_C))$, which allows us to exclude $19$ for the octahedron.  In \cite{SmDCS} we study the orientations for the closed--open map in more detail, and show that in fact all of the signs appearing in \protect \MakeUppercase {C}orollary\nobreakspace \ref {labCO}\ref{COitm1} are positive for a `standard' spin structure (and the choice of spin structure on $L_O$ is irrelevant for $\CO(\PD(N_O))$).

\subsection{Bubbled configurations}
\label{sscBubConf}

In order to compute the self-Floer cohomology of $L_C$ we need to study the moduli space $M_4$ of index $4$ discs with $2$ boundary marked points, and the evaluation map $\ev_2 \mc M_4 \rightarrow L_C^2$.  In general this moduli space has boundary components comprising bubbled configurations, and when we compute the local degree of $\ev_2$ at a point $(q, p)\in L_C^2$ we need to ensure that this point does not lie in the image of the boundary.  This is so that the local degree is locally constant on a neighbourhood of $(q, p)$, and hence that we can perturb $p$ and $q$ if necessary to ensure transversality in the pearl complex.

So take distinct points $p, q \in L_C$ and suppose that there exists a bubbled configuration evaluating to $(q, p)$.  Either the two marked points are in a single index $2$ disc component of the bubble tree, or they are in adjacent index $2$ disc components.  In either case, there exists a point $r \in L_C$ and index $2$ discs $u_1$ and $u_2$ such that the boundary of $u_1$ passes through $p$ and $r$ whilst that of $u_2$ passes through $q$ and $r$ (in the first case we just take $u_1=u_2$).

By the classification of index $2$ discs, this means that there exist vertices $v_1$ and $v_2$ of the $d$-point configuration representing $r$ such that $p$ and $q$ are obtained from $r$ by rotating around $v_1$ and $v_2$ respectively.  So $v_1$ lies in the configuration representing $p$ whilst $v_2$ lies in that representing $q$.  In other words, there exist a vertex of $p$ and a vertex of $q$ whose angle (or, equivalently, distance) of separation coincides with the angle between two vertices of $C$, which need not be distinct.  And conversely if there are two such vertices then a bubbled configuration does exist.

Now note that if $w_1$ and $w_2$ are non-zero vectors in the fundamental representation $V$ of $\SU(2)$ then the angle $\theta$ between the points $[w_1]$ and $[w_2]$ on the sphere $\P V$ satisfies
\begin{equation}
\label{eqAngle}
\cos \frac{\theta}{2} = \frac{\lvert \ip{w_1}{w_2} \rvert}{\norm{w_1}\norm{w_2}}.
\end{equation}
This can be verified easily when $w_1 = (0, 1)$, and then the general result follows from the $\SU(2)$-invariance of both sides.  So there exists an index $4$ bubbled configuration through $p$ and $q$, if and only if the sets
\begin{equation}
\label{eqpqSet}
\left\{\frac{\lvert \ip{w_1}{w_2} \rvert^2}{\norm{w_1}^2\norm{w_2}^2} : [w_1] \text{ a vertex of } p \text{ and } [w_2] \text{ a vertex of } q \right\}
\end{equation}
and
\begin{equation}
\label{eqCSet}
\left\{\frac{\lvert \ip{w_1}{w_2} \rvert^2}{\norm{w_1}^2\norm{w_2}^2} : [w_1] \text{ and } [w_2] \text{ vertices of } C \right\}
\end{equation}
intersect.

For example, if $p, q \in L_\tri$ are represented by the triangles with vertices \textcolor{red}{$\blacktriangle$} and \textcolor{blue}{$\bullet$} respectively on the left-hand diagram in Fig.\nobreakspace \ref {figBubConfs} then there \emph{is} a bubbled configuration through $p$ and $q$ because the vertices in the southern hemisphere are distance $2\pi/3$ apart.  The third vertex of the configuration $r$ mentioned above would be at $\infty$.  In contrast, there is \emph{no} bubbled configuration through the points $p$ and $q$ shown in the right-hand diagram since in this case the distances between vertices of $p$ and vertices of $q$ are $\pi/3$ and $\pi$, neither of which appears as a distance between vertices in a single equilateral triangle.
\begin{figure}[ht]
\centering
\includegraphics[scale=1]{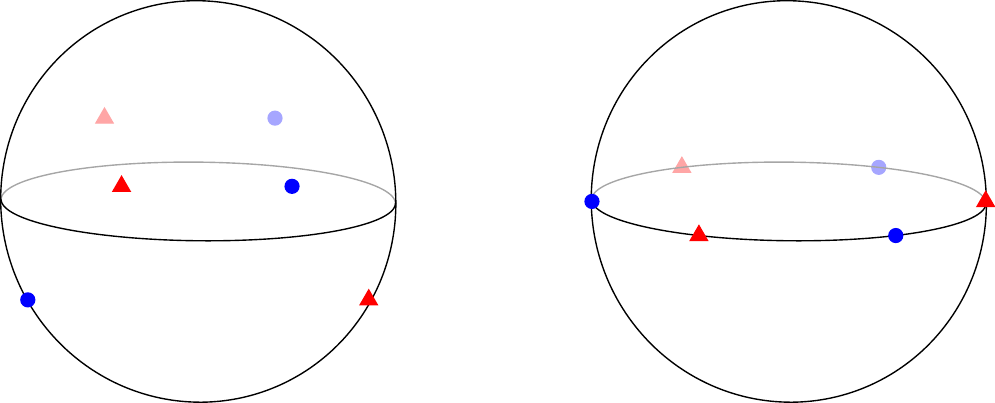}
\caption{Choices of $p, q \in L_\tri$ demonstrating existence and non-existence of bubbled configurations.\label{figBubConfs}}
\end{figure}
In terms of the two sets above, \eqref{eqCSet} is clearly $\{1/4, 1\}$ in both cases, whilst \eqref{eqpqSet} is easily seen to contain $1/4$ for the left-hand diagram but is given by $\{0, 3/4\}$ for the right-hand diagram.

\subsection{The antiholomorphic involution II}
\label{sscModInv}

In this subsection we explore antiholomorphic involutions in a slightly more general setting.  Let $X$ be a complex manifold, $Y \subset X$ an analytic subvariety, $L \subset X \setminus Y$ a totally real submanifold which is closed as a subset of $X$, and $\tau$ an antiholomorphic involution of $X \setminus Y$ which fixes $L$ pointwise.  Suppose moreover that $\tau$ enables us to reflect holomorphic discs with boundary on $L$, in the sense that for any holomorphic disc $u \mc (D, \pd D) \rightarrow (X, L)$ there exists a holomorphic disc $v$ on $L$ such that $v(z) = \tau \circ u (\conj{z})$ for all $z \in D$ with $\conj{z} \notin u^{-1}(Y)$.  Using this we can double any disc $u$ on $L$ to a sphere $\double{u}$.

We now introduce a new definition:

\begin{defn}\label{labStrongSimp}  A holomorphic disc $u \mc (D, \pd D) \rightarrow (X, L)$ is \emph{strongly simple} if its double $\double{u}$ is not multiply-covered.
\end{defn}

Holomorphic discs can always be replaced by strongly simple discs, in the following sense:

\begin{lem}\label{labAllStrongSimp}  Given a non-constant holomorphic disc $u$ with boundary on $L$, there exists a strongly simple disc $v$ on $L$ such that:
\begen
\item \label{simpitm1} $u(\pd D)\subset v(\pd D)$.
\item \label{simpitm2} $\double{u}(\C\P^1)=\double{v}(\C\P^1)$.
\item \label{simpitm3} If $u$ is not itself strongly simple and every non-constant holomorphic disc on $L$ has Maslov index at least $2$ then $\mu(v) \leq \mu(u) -2$.
\end{enumerate}
\end{lem}
\begin{proof}  If $u$ is already strongly simple then we can just take $v=u$, so suppose this is not the case.  Then $\double{u}$ is a non-simple sphere, and hence is given by $w \circ \psi$ for some branched cover $\psi \mc \C\P^1 \rightarrow \C\P^1$ of degree $d > 1$ and some simple holomorphic sphere $w \mc \C\P^1 \rightarrow X$.  Pick three points $a_1$, $a_2$ and $a_3$ in $\pd D$ whose images under $\psi$ are distinct injective points of $w$.  Reparametrising $w$ if necessary, and correspondingly modifying $\psi$, we may assume that $\psi(a_i) \in \pd D$ for each $i$.

As in Section\nobreakspace \ref {sscAntInv}, let $c \mc \C\P^1 \rightarrow \C\P^1$ denote the map $z \mapsto 1/\conj{z}$, with fixed-point set $\pd D$, and let $\conj{w}$ denote the reflection of $w$, i.e.~the holomorphic sphere given by $\tau \circ w \circ c$ whenever this is defined.  Note that $\conj{w}$ is simple (if not then it would be a multiple cover and hence $w=\tau \circ \conj{w} \circ c$ would be a multiple cover) and that
\[
\conj{w}(\C\P^1) = w(\C\P^1)=\double{u}(\C\P^1).
\]
In particular, $w$ and $\conj{w}$ are simple holomorphic spheres with the same image, and therefore differ by reparametrisation.  To see this, let $U$ and $\conj{U}$ be the cofinite subsets of their common image comprising the images of injective points of $w$ and $\conj{w}$ respectively.  Then $w^{-1} \circ \conj{w}$ defines a biholomorphism between the cofinite sets $\conj{w}^{-1}(U \cap \conj{U})$ and $w^{-1}(U \cap \conj{U})$, and considering the effect of this biholomorphism on the ends of these sets (it must pair them up) we deduce that it extends to an automorphism $\phi$ of $\C\P^1$ satisfying $\conj{w} = w \circ \phi$.

Now note that we have
\[
w \circ \psi = \double{u} = \tau \circ \double{u} \circ c = \conj{w} \circ c \circ \psi \circ c
\]
on the cofinite subset $\double{u}^{-1}(Y)$ of $\C\P^1$, and hence
\begin{equation}
\label{eqwpsi}
w \circ \psi = w \circ \phi \circ c \circ \psi \circ c
\end{equation}
on all of $\C\P^1$.  Applying this at our points $a_i$ we deduce that $\phi$ fixes the three points $\psi(a_i)$ and thus is the identity.  Then \eqref{eqwpsi} tells us that $\psi$ coincides with $c \circ \psi \circ c$ at injective points of $w$ and therefore everywhere.  In other words, we have shown that $\conj{w} = w$ and that $\psi$ commutes with $c$ (so $\psi(\pd D) \subset \pd D$).

Using this, we see that $w(\pd D)$ contains $u(\pd D)=w(\psi(\pd D))$ and lies in the fixed locus of $\tau$.  This fixed locus contains $L$ as an isolated component (locally about a fixed point of an antiholomorphic involution one can choose holomorphic coordinates in which the involution is given by complex conjugation), so $w^{-1}(L)$ is open in $\pd D$.  Since $L$ is closed in $X$, $w^{-1}(L)$ is also closed in $\pd D$, and hence $w(\pd D) \subset L$.  This means that $v_1\coloneqq w|_D$ and $v_2\coloneqq w \circ(z \mapsto 1/z)|_D$ are holomorphic discs on $L$ whose boundaries contain $u(\pd D)$.  Their doubles are $w$ and $\conj{w}$ respectively, so they are strongly simple and satisfy $\double{u}(\C\P^1)=\double{v}_i(\C\P^1)$.  If we can show that $\mu(v_i) \leq \mu (u) - 2$ for some $i$ then we can take $v$ to be this $v_i$ and we're done.

Well, in $H_2(\C\P^1, \pd D)$ we have $\psi_* ([D]) = d_1 [D] + d_2 [c(D)]$ for some non-negative integers $d_1$ and $d_2$, which sum to $d$ since $\psi$ commutes with $c$.  Then in $H_2(X, L)$ we have $[u] = d_1 [v_1] + d_2 [v_2]$, and hence
\[
\mu(u) = d_1 \mu(v_1) + d_2 \mu(v_2).
\]
Since each $\mu(v_i)$ is at least $2$ (by our assumption on Maslov indices of discs on $L$), and the sum of the $d_i$ is $d > 1$, we must have $\mu(v_i) \leq \mu(u)-2$ for some $i$, proving the lemma.
\end{proof}

Really our interest in disc analysis is through its application to Floer theory, so suppose now that in fact $X$ is a compact K\"ahler manifold, $L \subset X$ a closed, connected, monotone Lagrangian with minimal Maslov index not equal to $1$, and that every holomorphic disc $u$ bounded by $L$ has all partial indices non-negative; note that the partial indices, as well as the notion of holomorphicity of course, depend on the complex structure $J$ on $X$.  We are still assuming the existence of $Y$ (disjoint from $L$) and $\tau$ as above.  In particular, these conditions are satisfied by the Platonic family, with $Y=Y_C$ for $C=\tri$ or $T$ and $Y=\emptyset$ for $C=O$ or $I$.

If $f$ is a Morse function on $L$, and $g$ is a metric such that $(f, g)$ is Morse--Smale, \protect \MakeUppercase {P}roposition\nobreakspace \ref {labPearlCx} ensures that, possibly after replacing $f$ and $g$ by their pullbacks under a diffeomorphism of $L$, we may use the data $(f, g, J)$ to compute the self-Floer cohomology of $L$ using the pearl complex.  As usual, we assume our coefficient ring has characteristic $2$ unless we have fixed a choice of orientation and spin structure on $L$.  Using \protect \MakeUppercase {L}emma\nobreakspace \ref {labAllStrongSimp} we can restrict our attention to pearly trajectories in which all discs are strongly simple:

\begin{lem}\label{labTrajStrong}
In every pearly trajectory contributing to the differential on the pearl complex given by \protect \MakeUppercase {P}roposition\nobreakspace \ref {labPearlCx}, all holomorphic discs are strongly simple.
\end{lem}
\begin{proof}
Suppose for contradiction that there is a trajectory in which some disc $u$ is not strongly simple.  Let its two marked points be at $\pm 1$---note that $u(-1) \neq u(1)$, otherwise we could delete the disc $u$ from the trajectory and obtain a trajectory in negative virtual dimension, which is impossible (in the proof of \protect \MakeUppercase {P}roposition\nobreakspace \ref {labPearlCx} it is ensured that the relevant moduli space is transversely cut out).  By \protect \MakeUppercase {L}emma\nobreakspace \ref {labAllStrongSimp} there exists a holomorphic disc $v$, of index strictly less than $u$, with $u(\pd D)\subset v(\pd D)$.  In particular, we may reparametrise $v$ such that $v(\pm 1) = u(\pm 1)$, and then replace $u$ by $v$ to obtain again a trajectory in negative virtual dimension, giving the desired contradiction.
\end{proof}

This is particularly useful when $\tau$ extends to a global involution (as in the case of the octahedron and icosahedron), so we assume from now on that $Y$ is empty and that our coefficient ring $R$ has characteristic $2$.  In this situation we have:

\begin{prop}\label{labWide}  The Lagrangian $L$ is wide over $R$.  In other words, after collapsing the grading of $H^*(L; R)$ to $\Z/2$, we have an isomorphism of $\Z/2$-graded $R$-modules
\[
HF^*(L, L; R) \cong H^*(L; R).
\]
\end{prop}
\begin{proof}  We argue analogously to Haug \cite{Ha}, and show that all positive index contributions to the pearl complex differentials (which we also refer to as `quantum corrections') occur in pairs and hence cancel over $R$.  This makes the self-Floer cohomology of $L$, as computed by the pearl complex, agree with the Morse cohomology, which is in turn isomorphic to the singular cohomology.

The way we pair up the positive index contributions is by constructing a fixed-point-free involution $\tau_*$ on the space of such trajectories.  For each sequence $(u_1, \dots, u_l)$ of non-constant holomorphic discs comprising a pearly trajectory (with $l \geq 1$), we define a new trajectory by
\[
\tau_* (u_1, \dots, u_l) = (\overline{u}_1, \dots, \overline{u}_l),
\]
where $\overline{u}_i$ is the disc given by $\overline{u}_i(z) =\tau \circ u_i (\conj{z})$ for all $z$.  If we can show that this `reflect the discs' map has no fixed points then we're done.

Well, if it \emph{did} have a fixed point $(u_1, \dots, u_l)$ then the disc $u_1$ would be equal to its reflection, up to reparametrisation.  In particular, its double $\double{u}_1$ would hit every point in its image at least twice (counting with multiplicity).  This forces $\double{u}_1$ to be a multiple cover (see \cite[Section 2.3]{McS}), contradicting the fact that $u_1$ is strongly simple.  Therefore $\tau_*$ has no fixed points, and thus we have the required cancellation.
\end{proof}

In Appendix\nobreakspace \ref {secTransPearl} we also establish (in \protect \MakeUppercase {P}roposition\nobreakspace \ref {labPearlCxProd}) that any triple $(f_j, g_j)_{j=1}^3$ of Morse--Smale pairs on $L$ can be perturbed in order to be used to define the product on self-Floer cohomology, by counting Y-shaped pearly trajectories with one Morse--Smale pair used on each leg, i.e.~each branch of the Y.  And by an argument analogous to \protect \MakeUppercase {L}emma\nobreakspace \ref {labTrajStrong} the discs appearing in the legs, rather than at the centre of the Y, are all strongly simple.  Using this we get:

\begin{prop}\label{labGrComm}  The product on $HF^*(L, L; R)$ is commutative.  The only positive index trajectories whose contributions do not cancel have a single disc, at the centre of the Y, and no others.
\end{prop}
\begin{proof}  By reflecting the leg discs, we see that trajectories with such discs cancel out.  Reflecting the disc at the centre of the Y reverses the order of the three boundary marked points, and we get a bijection between trajectories contributing to the coefficient of $z$ in $x \qcup y$ and those contributing to the coefficient of $z$ in $y \qcup x$.
\end{proof}

Note however that in general the product on $HF^*(L, L; R)$ is different from that on $H^*(L; R)$.  A simple example is provided by the equator in $\C\P^1$, whose self-Floer cohomology ring over $\Z/(2)$ is isomorphic to $\Z[x]/(2, x^2-1)$, whereas $H^*(S^1; \Z/(2)) \cong \Z[x]/(2, x^2)$ (and $|x|=1$ in both cases); these are not isomorphic as $\Z/2$-graded rings.  Fukaya et al.~proved a very similar result \cite[Corollary 1.6]{FOOOinv}, that the Floer cohomology ring is graded-commutative with rational Novikov coefficients, under hypotheses that ensure the Maslov index is trivial modulo $4$, so that one can control the signs of reflected discs.

We shall not use \protect \MakeUppercase {P}roposition\nobreakspace \ref {labGrComm} in what follows, but from \protect \MakeUppercase {P}roposition\nobreakspace \ref {labWide} we obtain:

\begin{prop} \label{labHFLOIk}  For a field $k$ of characteristic $2$ we have isomorphisms of $k$-vector spaces
\[
HF^0(L_O, L_O; k) \cong HF^1(L_O, L_O; k) \cong k^2
\]
and
\[
HF^0(L_I, L_I; k) \cong HF^1(L_I, L_I; k) \cong k.
\]
\end{prop}
\begin{proof}  Applying \protect \MakeUppercase {P}roposition\nobreakspace \ref {labWide} to these Lagrangians, we reduce the problem to computing their singular cohomology.  In \protect \MakeUppercase {L}emma\nobreakspace \ref {labMonot} we calculated this to be $\Z$, $0$, $\Gamma_C^{\mathrm{ab}}$ and $\Z$ in degrees $0$ to $3$, so it is left to understand the abelianisations of $\Gamma_O$ and $\Gamma_I$.  These are well-known but we sketch a computation for completeness.

First note that these groups are respectively the binary octahedral and binary icosahedral groups, which project $2:1$ to the standard octahedral and icosahedral subgroups of $\mathrm{SO}(3)$, which we denote by $\overline{\Gamma}_O$ and $\overline{\Gamma}_I$.  The image of the commutator subgroup of $\Gamma_C$ in $\overline{\Gamma}_C$ is simply the commutator subgroup of $\overline{\Gamma}_C$.

It is easy to check that the commutator subgroup of $\overline{\Gamma}_O$ is the index $2$ subgroup containing only the rotations through angle $\pi$ (plus the identity of course).  Hence the commutator subgroup of $\Gamma_O$ has either index $2$ or index $4$.  In both cases we see that it is of even order, so contains an element of order $2$.  The only such element in $\SU(2)$ is $-I$, so the commutator subgroup of $\Gamma_O$ is the preimage of the commutator subgroup of $\overline{\Gamma}_O$.  In particular, it has index $2$ so $\Gamma_O^\mathrm{ab} \cong \Z/2$.  Thus $H^*(L_O; k)$ is isomorphic to $k$ degrees $0$ to $3$, by the universal coefficient theorem.

Turning now to the icosahedron, it is well-known that $\overline{\Gamma}_I$ is isomorphic to the alternating group $A_5$, which is simple.  So the commutator subgroup of $\Gamma_I$ is either the whole group---in which case $\Gamma_I^\mathrm{ab}$ is trivial---or is an index $2$ subgroup covering $\overline{\Gamma}_I$.  But the latter is impossible, since an index $2$ subgroup of $\Gamma_I$ would have even order and thus contain $-I$, and there is no proper subgroup of $\Gamma_I$ covering $\overline{\Gamma}_I$ and containing $-I$ ($\Gamma_I$ contains two lifts of each element of $\overline{\Gamma}_I$, and these differ by the action of $-I$, so any subgroup containing one lift and $-I$ contains both lifts).  Hence $\Gamma_I^\mathrm{ab}$ is trivial, and so $H^*(L_I; k)$ is $k$ in degrees $0$ and $3$ but vanishes in degrees $1$ and $2$.  In fact, $L_I$ is the well-known Poincar\'e $3$-sphere: a homology $3$-sphere not homotopy equivalent to $S^3$.
\end{proof}

Even outside characteristic $2$, $\tau$ is useful because it induces an involution on moduli spaces of discs.  The effect of this involution on orientations was computed in \cite[Theorem 1.3]{FOOOinv}: on the space of unparametrised index $2i$ discs with $j$ boundary marked points, the involution changes the orientation by a factor of $(-1)^{i+j}$.  For example, on the space of unmarked index $2$ discs, it is orientation-reversing.  This agrees with our earlier computation in \protect \MakeUppercase {P}roposition\nobreakspace \ref {labInd2MS} that this moduli space is diffeomorphic to the sphere $S^2$, where one can see directly that the involution acts as the antipodal map.

In the context of the pearl complex we are interested in discs with $2$ marked points.  In this case, the involution reverses orientations in index $2$ and preserves orientations in index $4$.  Hence, since we have shown that there are no contributing discs fixed by the involution, all trajectories in the pearl complex which contain an index $2$ disc cancel out, whilst the count of index $4$ discs through two points is always even.

Using this fact we can rule out characteristics $43$ and $571$ for the icosahedron:

\begin{prop}\label{RuleOutChars}
The only possible value of $p$ for $L_I$ in \protect \MakeUppercase {C}orollary\nobreakspace \ref {labCO}\ref{COitm2} is $2$.
\end{prop}
\begin{proof}
Recall from the work of Fukaya--Oh--Ohta--Ono \cite[Chapter 8]{FOOObig} that in order to orient moduli spaces of discs on $L_I$ it actually suffices to choose a (stable conjugacy class of) \emph{relative} spin structure on $L_I$, and that such structures form a torsor for $H^2(X_I, L_I; \Z/2)$ \cite[Proposition 8.1.6]{FOOObig}.  Moreover, the effect of shifting relative spin structure by a class $\eps$ is to change the orientation on the moduli space of discs in class $A$ by $(-1)^{\ip{\eps}{A}}$ \cite[Proposition 8.1.16]{FOOObig} (building on work of de Silva \cite[Theorem Q]{VdS} and Cho \cite[Theorem 6.4]{ChoCl}).

One can easily compute that $H^2(X_I, L_I; \Z/2)$ is just $\Z/2$, and the pairing of the non-zero class $\eps$ with a disc class $A$ simply records the parity of $\mu(A)/2$.  Therefore changing relative spin structure reverses the signs of discs and trajectories of index $2$, and preserves the signs of those of index $4$.  We have just seen that all index $2$ contributions to the differential cancel out, so we conclude that (additively, at least) the self-Floer cohomology of $L_I$ is independent of the choice of relative spin structure.

Now recall the eigenvalue constraint used in \protect \MakeUppercase {C}orollary\nobreakspace \ref {labCO}\ref{COitm2}.  By the preceding discussion, any allowed prime $p$ must work for \emph{both} choices of relative spin structure, for which $\CO(\PD(Y_I))$ takes opposite signs.  This means that $p$ must divide $\chi(\PD(Y_I)\qcup))$ evaluated at both $12$ and $-12$.  The only common factor of the resulting numbers is $2$, eliminating the $p=43$ and $p=571$ cases.
\end{proof}

If one tries to apply the argument of \protect \MakeUppercase {P}roposition\nobreakspace \ref {labWide} to $L_\tri$ and $L_T$, the problem is that the reflection of a disc generally has different index from the original disc, so the map $\tau_*$ does not act on individual moduli spaces of trajectories.  Instead it mixes up moduli spaces of different virtual dimensions, and the argument falls apart.

\section{The Morse and pearl complexes}
\label{secMorsePearl}

\subsection{Stereographic projection}
\label{sscSterProj}

To understand the topology of the Lagrangian $L_C \cong \SU(2)/\Gamma_C$ we seek the fundamental domain for the action of $\SU(2)$ on $C$, in other words the set of points of $\SU(2)$ which are no further from the identity, $I$, than from any other element of $\Gamma_C$.  We think of $\SU(2)$ as the unit sphere in the quaternions $\mathbb{H}$, via
\[
\SU(2) = \left\{ \begin{pmatrix} u & -\conj{v} \\ v & \conj{u} \end{pmatrix} : u, v \in \C \text{ and } |u|^2+|v|^2=1 \right\},
\]
and identify $\mathbb{H}$ with $\R^4$ by $(u, v) \leftrightarrow (\re u, \im v, \re v, \im u)$.  The left- and right-multiplication actions of $\SU(2)$ on $\mathbb{H}$ clearly preserve the standard inner product, and hence the induced round metric on $\SU(2)$ is bi-invariant.  In particular, one-parameter subgroups of $\SU(2)$ are geodesics in the round metric.  We view the fundamental domain as a subset of $\R^3=\{0\} \times \R^3 \subset \R^4$ by stereographic projection from $-I$:
\begin{equation}
\label{labSterProj}
\begin{pmatrix} u & -\conj{v} \\ v & \conj{u} \end{pmatrix} \in \SU(2) \mapsto \frac{1}{1+\re u} (\im v, \re v, \im u).
\end{equation}

Recall from Section\nobreakspace \ref {sscCoords} the identification of $\C\P^1$ with the unit sphere in $\R^3$, which gives the action of $\SU(2)$ as rotations of the sphere that we have been using throughout.  With these conventions, the rotation through angle $\theta \in [0, \pi]$ about a unit vector $\mathbf{n} \in \R^3$ lifts to $\exp (\mathbf{n} \cdot i \theta \conj{\sigma} /2)$ in $\SU(2)$, where $\sigma$ is the vector of Pauli matrices 
\[
\sigma = \bigg( \begin{pmatrix} 0 & 1 \\ 1 & 0 \end{pmatrix}, \begin{pmatrix} 0 & -i \\ i & 0 \end{pmatrix}, \begin{pmatrix} 1 & 0 \\ 0 & -1 \end{pmatrix} \bigg).
\]
It also lifts to $-I$ times this, but we only need consider the representative closest to $I$.  Geodesics on $\SU(2)$ through $I$ correspond to intersections of $2$-planes through $\pm I$ in $\R^4$ with $\SU(2)$, and hence their stereographic projections are straight lines in $\R^3$.  For each fixed $\mathbf{n}$ the one-parameter subgroup $\theta \mapsto \exp(\mathbf{n} \cdot i\theta\conj{\sigma}/2)$ is therefore sent by \eqref{labSterProj} to a straight line in $\R^3$.  It is easy to check that the line is in the direction $\mathbf{n}$, and by restricting the stereographic projection to the $2$-plane through $\pm I$ containing this direction it is easy to compute that in fact $\exp(\mathbf{n} \cdot i \theta \conj{\sigma}/2)$ projects to $\tan(\theta/4)\mathbf{n}$---see Fig.\nobreakspace \ref {figSterProj}.
\begin{figure}[ht]
\centering
\includegraphics[scale=1]{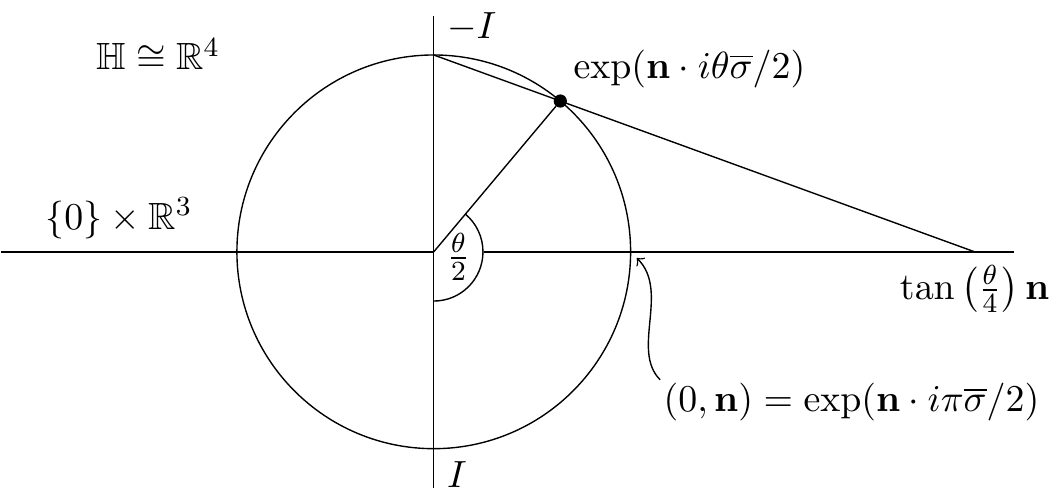}
\caption{Stereographic projection on $\SU(2)$.\label{figSterProj}}
\end{figure}

The set of points in $\SU(2)$ which are equidistant from $\exp(\mathbf{n} \cdot i \theta \conj{\sigma}/2)$ and the identity is an equatorial $2$-sphere on $\SU(2) \cong S^3$ which cuts the geodesic generated by $\mathbf{n} \cdot i \conj{\sigma}/2$ orthogonally at the points
\[
\exp(\mathbf{n} \cdot i (\theta/2) \conj{\sigma}/2) \text{ and } \exp(\mathbf{n} \cdot i (\theta/2+2\pi) \conj{\sigma}/2).
\]
Under stereographic projection this maps to a sphere in $\R^3$ which cuts the line in direction $\mathbf{n}$ orthogonally at $\tan(\theta/8)\mathbf{n}$ and $\tan(\theta/8+\pi/2)\mathbf{n}=-\cot(\theta/8)\mathbf{n}$.  Its centre is thus at
\[
\frac{1}{2}\lb\tan(\theta/8)-\cot(\theta/8)\rb \mathbf{n} = -\cot(\theta/4)\mathbf{n}
\]
and its radius is $(\tan(\theta/8) + \cot(\theta/8))/2 = \cosec(\theta/4)$.

In fact, the fundamental domain is contained in the unit ball in $\R^3$, since this corresponds to the set of points in $\SU(2)$ closer to $I$ than to $-I$, and it is convenient to compose \eqref{labSterProj} with the diffeomorphism $\mathbf{x} \mapsto \mathbf{2x}/(1-\norm{\mathbf{x}}^2)$ from the unit ball to the whole of $\R^3$; this corresponds to replacing the denominator in \eqref{labSterProj} by $\re u$, and results in $\exp(\mathbf{n}\cdot i \theta \conj{\sigma}/2)$ projecting to $\tan(\theta/2)\mathbf{n}$ rather than $\tan(\theta/4)\mathbf{n}$.  If $\mathbf{x}$ denotes the coordinate on the unit ball, and $\mathbf{y}$ the coordinate on the codomain $\R^3$, the ball
\[
\big\{\norm{\mathbf{x}+\cot(\theta/4)\mathbf{n}}^2\leq\cosec^2(\theta/4)\big\} = \big\{2\cot(\theta/4)\mathbf{x}\cdot\mathbf{n}\leq1-\norm{\mathbf{x}}^2\big\}
\]
(or strictly its intersection with the unit ball) is sent to the half-space
\begin{equation}
\label{labHalfSp}
\big\{\mathbf{y}\cdot\mathbf{n} \leq \tan(\theta/4)\big\}.
\end{equation}
For any $\alpha$, the right action of $\exp(\mathbf{n} \cdot i\theta\conj{\sigma}/2)$ on $\SU(2)$ sends $\{\mathbf{y} \cdot \mathbf{n} = \tan\alpha\}$ to $\{\mathbf{y} \cdot \mathbf{n} = \tan(\alpha+\theta/2)\}$, by translation in the direction $\mathbf{n}$ combined with a left-handed rotation about $\mathbf{n}$ through angle $\theta/2$.  This is easy to check when $\mathbf{n}=(0, 0, 1)$, and the general case then follows from the fact that the projection intertwines the action of $\SU(2)$ on itself by conjugation with its action on $\R^3$ by rotation (this in turn is easy to check infinitesimally).  We refer to the map
\[
\begin{pmatrix} u & -\conj{v} \\ v & \conj{u} \end{pmatrix} \in \SU(2) \mapsto \frac{1}{\re u} (\im v, \re v, \im u)
\]
as \emph{modified stereographic projection}, and will usually use coordinates $(x, y, z)$ on $\R^3$ (instead of the $\mathbf{y}$ used above).

This projection also has the advantage that boundaries of axial discs are sent to straight line segments in the fundamental domain.  To see this, note that the boundary of an axial disc is (the image in $L_C$ of) a right translate of one-parameter subgroup of $\SU(2)$.  We have seen that such a path is a geodesic on $\SU(2)$, i.e.~the intersection of $\SU(2) \cong S^3 \subset \R^4$ with a $2$-plane through the origin.  We can therefore write it as the intersection of two equatorial $2$-spheres, which we know project to planes.

\subsection{Triangle}
\label{sscMPTri}

The constructions of the fundamental domain, the Heegaard splitting and the Morse function are based on \cite[Section 5]{EL1}, though are not identical.  In particular, our projection leads to left-handed face identifications rather than right-handed.  For convenience of comparison, we employ matching notation.  We fix our choice of representative configuration $\tri$ as $[x^3+y^3]$, i.e.~$\tri_f$ from Appendix\nobreakspace \ref {secExplReps}.  Then the stabiliser of $\tri$ under the $\SU(2)$-action consists of (the lifts to $\SU(2)$ of) rotations about $(0, 0, 1)$ through angle $\pm 2 \pi/3$, and rotations about $(1, 0, 0)$, $(-1/2, \sqrt{3}/2, 0)$ and $(-1/2, -\sqrt{3}/2, 0)$ through angle $\pm \pi$.

Plugging these into \eqref{labHalfSp}, we see that under modified stereographic projection the image of the fundamental domain is $H \times [-1/\sqrt{3}, 1/\sqrt{3}]$, where $H$ is the regular hexagon with vertices at $2/\sqrt{3}$ times the sixth roots of $-1$.  Each square face of this hexagonal prism is identified with the opposite face via a left-handed rotation through angle $\pi/2$, whilst the hexagonal faces are identified by a left-handed rotation through angle $\pi/3$.

We take a genus $3$ Heegaard splitting of $L_\tri$ whose handlebodies are a thickening of the edges and hexagonal faces of the prism and a thickening of the three lines joining the centres of opposite square faces.  Figure\nobreakspace \ref {figHex1} shows the prism with these two sets marked in the left and right diagrams respectively.  It also shows the critical points of the Morse function built by Evans--Lekili from this splitting: the maximum is at $m$, the index $2$ critical points at $x_1$, $x_2$ and $x_3$, the index $1$ critical points at $x_1'$, $x_2'$ and $x_3'$ and the minimum at $m'$.

\begin{figure}[ht]
\centering
\includegraphics[scale=1]{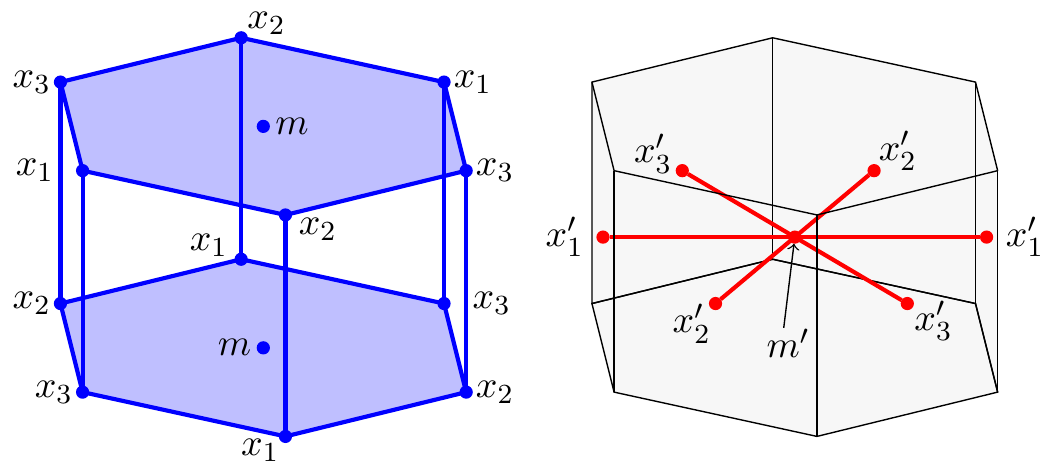}
\caption{The fundamental domain, Heegaard splitting and critical points for $C=\tri$.\label{figHex1}}
\end{figure}

Figure\nobreakspace \ref {figAsc} shows the point $x_3'$, the front face of the prism centred on it, and the shape of trajectories in its ascending manifold close to this face.  The solid trajectories belong to the descending manifolds of $x_1$, $x_2$ (twice) and $x_3$, the dotted trajectories flow into the fundamental domain towards $m$, whilst the dashed trajectories flow out of the domain, and so back into the opposite face with a twist of $\pi/2$, again towards $m$.  Of course we really need to choose a metric on $L_\tri$ in order to talk about trajectories, but we have a natural choice: that induced from the standard round metric on $\SU(2) \cong S^3$.

\begin{figure}[ht]
\centering
\includegraphics[scale=1]{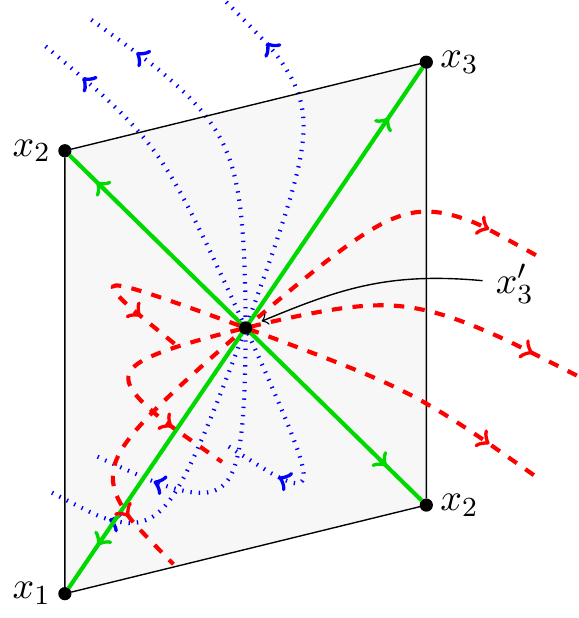}
\caption{Trajectories in the ascending manifold of $x_3'$.\label{figAsc}}
\end{figure}

The point $m'$ at the centre of the fundamental domain represents the identity in $\SU(2)$.  We saw earlier that our modified stereographic projection sends the one-parameter subgroup of $\SU(2)$ comprising (the lifts of) rotations about an axis $l \subset \R^3$ to $l$ itself.  The boundaries of the three index $2$ discs through $m'$ correspond to the one-parameter subgroups of (lifts of) rotations of $\tri$ about vertices, and therefore project to the axes through the vertices.  These are the lines joining the centres of the opposite square faces of the fundamental domain, i.e.~the core circles of the second handlebody.  The point $m$ represents the rotation of $\tri$ through angle $\pi/3$ about a vertical axis, so the lifts to $\SU(2)$ of the boundaries of the index $2$ discs through $m$ are obtained from those through $m'$ by multiplying on the right by (a lift of) this rotation.  We have seen that this right-multiplication action corresponds to translating $m'$ to $m$ and rotating by angle $\pi/6$ about a vertical axis, so the boundaries of the index $2$ discs through $m$ are the diagonals of the hexagonal faces of the fundamental domain.

The rotational symmetry group of the triangle $\tri$ in $\mathrm{SO}(3)$ acts on the fundamental domain respecting the Heegaard splitting.  This corresponds to the action of $\Gamma_C$ on $\SU(2)$ by conjugation, so preserves the round metric, and we may assume that it also preserves the Morse function.  It therefore permutes the descending manifolds, and we may choose orientations on them which are invariant under this action.  To see this note that it trivially preserves the orientations of the descending manifolds of $m'$ and $m$, which are a point and a dense open subset of $L_C$ respectively, and by inspection we can choose invariant orientations on the descending manifolds of the $x_i'$, which we can take to be the boundaries of the index $2$ discs through $m'$ (minus the point $m'$ itself).  Similarly we can choose invariant orientations on the \emph{ascending} manifolds of the $x_i$, which we take to be the boundaries of the index $2$ discs through $m$ (minus the point $m$).  This gives invariant coorientations on their descending manifolds, and since the orientation of $L_C$ is itself invariant, these invariant coorientations can be turned into invariant orientations.

In order to ensure transversality in the pearl complex we perturb the auxiliary data.  \protect \MakeUppercase {P}roposition\nobreakspace \ref {labPearlCx} shows that we can pull back the Morse function and metric by a diffeomorphism $\phi$ arbitrarily $C^\infty$-close to the identity in order to achieve the necessary genericity, so from now on we assume this has been done.  We take $\phi$ sufficiently close to $\id_{L_\tri}$ that the later arguments involving intersections of discs with various ascending and descending manifolds, and the index $4$ count in Section\nobreakspace \ref {sscInd4Tri}, are valid.  Although the Morse function and metric may themselves no longer be invariant under the action of $\Gamma_C$, they are small perturbations of symmetric data, and the perturbations do not affect the symmetry of the orientations.

Choosing the invariant orientations on the descending manifolds appropriately, the Morse differentials $\diff_M$ are
\begin{align*}
\diff_M m'&=0
\\ \diff_M x_i'&= x_i + x_{i+1} + 2 x_{i+2}
\\ \diff_M x_i&=0,
\end{align*}
with subscripts understood modulo $3$.  The coefficient $2$ in $\diff_M x_i'$ corresponds to the two solid flowlines towards $x_2$ in Fig.\nobreakspace \ref {figAsc}.  They both count with the same sign as they differ by rotation about the axis through the centre of the face shown in the figure, whilst the overall differential is invariant under cycling the $i$ because this corresponds to rotations through angle $2\pi/3$ about a vertical axis.  There is no quantum correction to $\diff_M m'$ in the Floer (pearl) differential $\diff m'$ for degree reasons, whilst that to $\diff_M x_i'$ is $X_i m'$, where $X_i$ counts flows upward from $x_i'$ and then along an index $2$ disc to $m'$ (in other words, intersections of index $2$ discs through $m'$ with the ascending manifold of $x_i'$).  There is one such trajectory for each $i$, so each $X_i$ is $\pm 1$, and by the symmetry we can replace all of the $X_i$ by a single $X \in \{\pm 1\}$.

The correction for $m$ can be written (again using approximate cyclic symmetry) in the form $Y(x_1+x_2+x_3)+\hat{Z}m'$, for $Y, \hat{Z} \in \Z$.  These count respectively the index $2$ trajectories $m \leadsto x_i$ and the index $4$ trajectories $m \leadsto m'$.  The former comprise intersections of the index $2$ discs through $m$ with the descending manifolds of the $x_i$; there is one of these for each $i$ so $Y \in \{ \pm 1 \}$.  The latter comprise index $4$ discs through $m$ and $m'$, which we count in Section\nobreakspace \ref {sscInd4Tri}.  There are no index $4$ contributions of the form `index $2$ disc through $m$, flow, index $2$ disc through $m'$', since the boundaries of these index $2$ discs stay within their respective handlebodies and the Morse flow goes from the the $m'$ handlebody to the $m$ handlebody and not vice versa.  Note that Evans--Lekili write $2Z$ for the count we are calling $\hat{Z}$.

Putting everything together, the $\Z/2$-graded Floer (pearl) cochain complex is
\[
CF^0(L_\tri, L_\tri; \Z)=\lspan{m', x_1, x_2, x_3} \text{ and } CF^1(L_\tri, L_\tri; \Z)=\lspan{x_1', x_2', x_3', m}
\]
(where $\lspan{\cdot}$ now indicates the free $\Z$-module generated by $\cdot$), and the Floer differentials $\diff^0 \mc CF^0 \rightarrow CF^1$ and $\diff^1 \mc CF^1 \rightarrow CF^0$ are given in these bases by
\[
\diff^0 = \begin{pmatrix} 0 & & & \\ 0 &  & \smash{\raisebox{-1.5ex}{{\makebox[0pt][c]{\Huge{$A$\ }}}}}  &  \\ 0 &  &  &  \\ 0 & 0 & 0 & 0 \end{pmatrix} \text{ and } \diff^1 = \begin{pmatrix} X & X & X & \hat{Z} \\ 1 & 2 & 1 & Y \\ 1 & 1 & 2 & Y \\ 2 & 1 & 1 & Y \end{pmatrix}.
\]
The matrix $A$ describes the as yet unknown index $2$ corrections to $\diff_M x_i$, but from the fact that $\diff^1 \circ \diff^0 = 0$ it is easy to see that actually $A=0$.  The Smith normal form of $\diff^1$ is thus diagonal with entries $1$, $1$, $1$ and $\det \diff^1 = 3XY-4\hat{Z}$.  This is essentially the argument given by Evans--Lekili.

\subsection{Tetrahedron}
\label{sscMPT}

We follow a similar strategy for the tetrahedron, taking
\[
T=[x^4+2\sqrt{3}x^2y^2-y^4],
\]
with vertices at $\pm \sqrt{2}/(\sqrt{3}-1)$ and $\pm (\sqrt{3}-1)i/\sqrt{2}$---this is exactly $T_e$ from Appendix\nobreakspace \ref {secExplReps}.  Under modified stereographic projection there are $14$ possible planes contributing to the boundary of the fundamental domain: $8$ from rotations about vertices (or equivalently about the centres of faces) through angle $\pm 2\pi/3$, and $6$ from rotations about the mid-points of edges through angle $\pm \pi$.  In fact, only the former are needed.  To see this note that the constraint coming from rotation about the bottom edge is $z \geq -1$ (applying \eqref{labHalfSp}), whilst rotations about the two lower vertices give $z \pm \sqrt{2} x \geq -1$.  Adding the latter two inequalities gives the edge rotation inequality for free.

The fundamental domain is therefore a regular octahedron, with vertices at
\[
(1/\sqrt{2}, \pm 1/\sqrt{2}, 0)\text{, } (-1/\sqrt{2}, \pm 1/\sqrt{2}, 0) \text{ and } (0, 0, \pm 1).
\]
Opposite faces are identified by a left-handed rotation through angle $\pi/3$, and in particular all of the vertices are identified.  We take a genus $4$ Heegaard splitting with handlebodies given by thickening the edges of the octahedron and the line segments joining opposite pairs of faces, shown in Fig.\nobreakspace \ref {figTetFD} in the left and right diagrams respectively.  From this splitting we construct a Morse function with maximum at $m$, index $2$ critical points at $x_1$, $x_2$, $x_3$ and  $x_4$, index $1$ critical points at $x_1'$, $x_2'$, $x_3'$ and $x_4'$ and minimum at $m'$.

\begin{figure}[ht]
\centering
\includegraphics[scale=1]{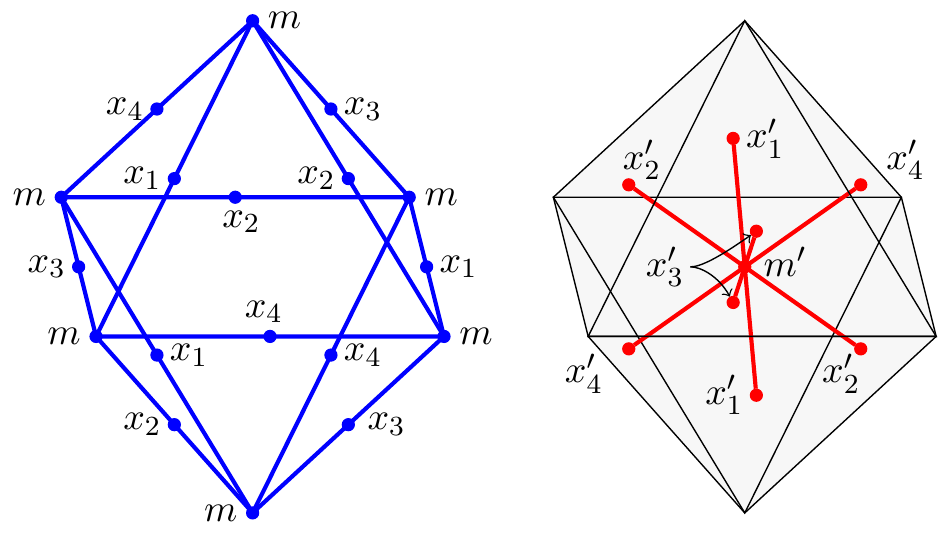}
\caption{The fundamental domain, Heegaard splitting and critical points for $C=T$.\label{figTetFD}}
\end{figure}

The ascending manifolds of the $x_i'$ are the faces of the octahedron, whilst the ascending manifolds of the $x_i$ are the edges.  The descending manifolds of the $x_i$ are, locally, small discs orthogonal to the edges.  Looking down from above onto the top half of the octahedron, we orient the faces and discs as indicated in Fig.\nobreakspace \ref {figOrs} by the dotted and dashed arrows respectively.  As for $\tri$ we perturb the Morse function and metric to ensure transversality, by pulling them back along a diffeomorphism near the identity.

\begin{figure}[ht]
\centering
\includegraphics[scale=1]{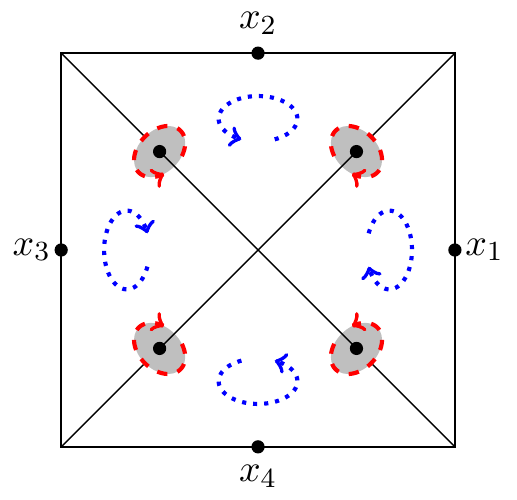}
\caption{Orientations of the ascending manifolds of the $x_i'$ (dotted) and of the descending manifolds of the $x_i$ (dashed).\label{figOrs}}
\end{figure}

Up to an overall sign depending on the chosen orientation of $L_C$, which we can eliminate if necessary by reversing the orientations of the faces defined above, the Morse differentials $\diff_M$ are then
\begin{align*}
\diff_M m'&=0
\\ \diff_M x_i'&= x_{i+1} + x_{i+2} + x_{i+3}
\\ \diff_M x_i&=0,
\end{align*}
with subscripts modulo $4$.

The quantum correction to $\diff_M x_i$ vanishes by the same $\diff^1 \circ \diff^0 = 0$ argument we used earlier, and there is no correction to $\diff_M m'$ for degree reasons.  The index $2$ correction to $\diff_M x_i'$ counts upward flows from $x_i'$ into index $2$ discs through $m'$, i.e.~intersections between index $2$ discs through $m'$ and the ascending manifold of $x_i'$: there is one such intersection for each $i$, given by the thick lines hitting the faces of the octahedron in the right-hand diagram Fig.\nobreakspace \ref {figTetFD}.  So
\[
\diff x_i' = x_{i+1} + x_{i+2} + x_{i+3} + X m'
\]
for some $X \in \{\pm 1\}$.  The reason why the intersections all carry the same sign is that the relative orientations of the disc boundaries match up with the relative orientations of the faces.

Finally, the correction to $\diff_M m$ counts index $4$ discs through $m$ and $m'$---of which there are $Z$, say---and intersections between index $2$ discs through $m$ with the descending manifolds of the $x_i$ ($Z$ is analogous to the count we called $\hat{Z}$ for the triangle, but we drop the $\hat{}$ to reduce clutter).  There are no index $4$ trajectories of the form `index $2$, flow, index $2$' since index $2$ discs through $m$ and $m'$ remain in their respective handlebodies as before.  The count of index $2$ discs through $m$ hitting the descending manifold of $x_i$ is $\pm 1$ for each $i$, so we have
\[
\diff m = Y(x_1+x_2+x_3+x_4) + Z m'
\]
for some $Y \in \{\pm 1\}$.  Again the index $2$ contributions all carry the same sign, because the relative orientations of disc boundaries and descending manifolds match up.

Thus the Floer cochain complex is
\[
CF^0(L_T, L_T; \Z)=\lspan{m', x_i} \text{ and } CF^1(L_T, L_T; \Z)=\lspan{x_i', m},
\]
and with respect to these bases the Floer differentials $\diff^0$ and $\diff^1$ are given by $\diff^0 = 0$ and
\begin{equation*}
\diff^1 = \begin{pmatrix} X & X & X & X & Z \\ 0 & 1 & 1 & 1 & Y \\ 1 & 0 & 1 & 1 & Y \\ 1 & 1 & 0 & 1 & Y \\ 1 & 1 & 1 & 0 & Y \end{pmatrix}.
\end{equation*}
The Smith normal form of $\diff^1$ is diagonal with entries $1$, $1$, $1$, $1$ and $4XY-3Z$.

\subsection{Octahedron and icosahedron}
\label{sscMPOI}

For $C$ equal to $O$ or $I$ one could similarly construct Heegaard splittings and Morse functions under stereographic projection, but we don't need to do this explicitly.  We just need to take a Morse function and metric on $L_C$, with a unique local maximum (at $m$, say) and a unique local minimum (at $m'$), and replace them with appropriate pullbacks as given by \protect \MakeUppercase {P}roposition\nobreakspace \ref {labPearlCx}.  For the icosahedron we ask that after perturbation $(m, m')$ lies close to the pair $(q, p)$ constructed in Section\nobreakspace \ref {sscInd4I}, which we can do by making this so for our original Morse function and then choosing the perturbation to be sufficiently small.

We now focus on the octahedron.  Let the index $2$ critical points be at $x_1, \dots, x_k$, and the index $1$ critical points be at $x_1', \dots, x_k'$; there are equal numbers of each by considering the Euler characteristic of the Morse complex.  Then the Floer cochain complex is
\[
CF^0(L_O, L_O; \Z)=\lspan{m', x_i} \text{ and } CF^1(L_O, L_O; \Z)=\lspan{x_i', m},
\]
and the Floer differentials $\diff^0$ and $\diff^1$ have the form
\begin{equation*}
\diff^0=\begin{pmatrix} 0 & & & \\ \raisebox{0.3ex}[3ex][0ex]{\vdots} &  & \smash{\raisebox{-1ex}{{\makebox[0pt][c]{\ \Huge{$A$\ }}}}}  &  \\ 0 &  &  &  \\ 0 & 0 & \cdots & 0 \end{pmatrix} \text{ and } \diff^1=\begin{pmatrix} b_1 & \cdots & b_k & D \\ & & & c_1 \\ & \smash{\raisebox{0ex}{{\makebox[0pt][c]{\ \ \ \Huge{$M$\ }}}}}  & & \raisebox{0.3ex}[3ex][0ex]{\vdots}  \\ &  &  & c_k \end{pmatrix},
\end{equation*}
where $M$ is the Morse differential $\lspan{x_i'} \rightarrow \lspan{x_i}$, the matrix $A$ and the vectors $B=(b_i)$ and $C=(c_i)$ represent index $2$ corrections, and the number $D$ (not to be confused with the unit disc!) is the index $4$ correction to $\diff_M m$.  By the comments at the end of Section\nobreakspace \ref {sscModInv}, $A$, $B$ and $C$ all vanish, whilst the count $D$ involves only index $4$ discs and is even.  The cokernel of $M$ is exactly the Morse cohomology group $H^2(L_O; \Z) \cong \Z/2$, so its Smith normal form has diagonal entries $1, \dots, 1, 2$.  The Smith normal form of $\diff^1$ therefore has diagonal entries $1, \dots, 1, 2, D$.

The argument for $X_I$ is completely analogous, except now the Smith normal form of $\diff^1$ has diagonal entries $1, \dots, 1, D$, where $D$ is the (even) count of index $4$ discs.

\section{Index $4$ discs and computation of Floer cohomology}
\label{secCompHF}

\subsection{The closed--open map II}
\label{sscCOII}

In this subsection we revisit the closed--open map and say as much as we can about the self-Floer cohomology of $L_C$ with the information we have so far.  Note that we have only had to do computations with axial discs to get this information.

From Section\nobreakspace \ref {secMorsePearl} we know that the ring $HF^*(L_C, L_C; \Z)$ is concentrated in degree $0$ (since the differential $\diff^1$ has non-zero determinant), and is cyclic except for the octahedron, where there are two diagonal entries of the Smith normal form of $\diff^1$ not equal to $1$.  In the latter case, the unit (represented by the critical point $m'$) generates a subring $\Z/(D)$, but $HF^0$ has an extra piece generated by a variable $x$, say, represented by a linear combination of the $x_i$, satisfying $2x=0$ and some quadratic relation $x^2 = \alpha x + \beta$, with $\alpha$ in $\Z/2$ and $\beta$ in $\Z/(D)$.

Since $HF^*(L_\tri, L_\tri; k)$ can only be non-zero over a field $k$ if $\Char k =5$, by \protect \MakeUppercase {C}orollary\nobreakspace \ref {labCO}\ref{COitm2}, the quantity $3XY-4\hat{Z}$ appearing as the determinant of $\diff^1$ for the triangle must be plus or minus a power of $5$ (possibly $5^0$).  Therefore $HF^0(L_\tri, L_\tri; \Z) \cong \Z/(5^n)$ for some non-negative integer $n$.  Similarly $HF^0(L_T, L_T; \Z) \cong \Z/(2^n)$ for some non-negative $n$.  By an analogous argument, now also using \protect \MakeUppercase {P}roposition\nobreakspace \ref {RuleOutChars}, we see that $HF^0(L_I, L_I; \Z) \cong \Z/(2^n)$, but now $n$ must be strictly positive since the index $4$ count $D$ is even (because of the antiholomorphic involution).  For the octahedron $D$ is (plus or minus) a product of powers of $2$ and $19$, and again the exponent of $2$ is positive.

We now show that for $\tri$, $T$ and $I$ there are only a few possibilities for $HF^0(L_C, L_C; \Z)$:
\begin{lem} \label{labCOConstr1}  We have ring isomorphisms:
\begen
\item \label{COCitm1} $HF^0(L_\tri, L_\tri; \Z) \cong 0$ or $\Z/(5)$.
\item \label{COCitm2} $HF^0(L_T, L_T; \Z) \cong 0$ or $\Z/(2)$ or $\Z/(4)$.
\item \label{COCitm3} $HF^0(L_I, L_I; \Z) \cong \Z/(2)$ or $\Z/(4)$ or $\Z/(8)$.
\end{enumerate}
\end{lem}
\begin{proof} \ref{COCitm1}  We have seen that
\[
HF^0(L_\tri, L_\tri; \Z) \cong \Z/(5^n)
\]
for some non-negative integer $n$, so suppose for contradiction that $n \geq 2$.  Composing $\CO$ with the quotient map $\Z/(5^n) \rightarrow \Z/(25)$, we get a unital ring homomorphism $\theta \mc QH^*(X_\tri; \Z) \rightarrow \Z/(25)$, and from \protect \MakeUppercase {C}orollary\nobreakspace \ref {labCO}\ref{COitm1} we know that $\theta$ satisfies $\theta(3E)=\pm 2$.  Squaring and applying $E^2=1$ we deduce that $9 = 4$ in $\Z/(25)$, which is the desired contradiction.  This proves the result.

\ref{COCitm2}  The argument is analogous except now we suppose for contradiction that we have a unital ring homomorphism $\theta \mc QH^*(X_T; \Z) \rightarrow \Z/(8)$ which satisfies $\theta(3H) = 4$ and $\theta(4E)=4$.  These two equalities force $\theta(H)$ to be even and $\theta(E)$ to be odd.  But this is impossible because $E^2=H$.

\ref{COCitm3}  This time suppose for contradiction that we have unital ring homomorphisms
\[
\theta_\pm \mc QH^*(X_I; \Z/(16)) \rightarrow \Z/(16)
\]
satisfying $\theta_\pm(H) = \pm 12$.  These correspond to the two choices of relative spin structure from the proof of \protect \MakeUppercase {P}roposition\nobreakspace \ref {RuleOutChars}.  Applying $\theta_\pm$ to $H^2=22E+2H+24$ we see that $6\theta_\pm(E) = 0$, and hence that $\theta_\pm(E)$ is divisible by $8$.  In the case of $\theta_-$ this is inconsistent with $E^2=2E+H+4$.
\end{proof}

These results for the triangle, tetrahedron and icosahedron are not used in the direct computations of Floer cohomology in the following subsections, although it is interesting to note that the rings turn out to be as large as is allowed by the above restrictions.  For the octahedron we can go further and actually pin down the Floer cohomology.  First we need to use the explicit calculation for the triangle from Section\nobreakspace \ref {sscInd4Tri} in order to determine a sign:

\begin{lem}\label{labCOSign}  Assuming \protect \MakeUppercase {C}orollary\nobreakspace \ref {labTriHF}, the value $\pm f_C \cdot 1_L$ of $\CO(\PD(N_C))$ calculated in \protect \MakeUppercase {P}roposition\nobreakspace \ref {labCONC} (for $\tri$, $T$ and $O$) takes the positive sign.
\end{lem}
\begin{proof}
Let $k$ be a field of characteristic $5$.  In $QH^*(X_\tri; k)$ we have $E = H^2 = (4H)^2$, so
\[
\CO(\PD(N_\tri)) = \CO(3E) = 3 \CO(4H)^2.
\]
Using $\CO(4H) = \pm 3 \cdot 1_L$ we therefore deduce that $\CO(\PD(N_\tri)) = 27 \cdot 1_L$.  By \protect \MakeUppercase {C}orollary\nobreakspace \ref {labTriHF} we know that $HF^0(L_\tri, L_\tri; k) \cong k$, and in $k$ we have $27 = 2 \neq -2$, so we have proved the lemma in the case of the triangle.

We now directly compare the orientations on the relevant moduli spaces to show that the discs contributing to $\CO(\PD(N_C))$ for different choices of $C$ all count with the same sign.  The key result of \cite[Chapter 8]{FOOObig} we shall use is the following: given a Riemann--Hilbert pair $(E, F)$ and a homotopy class of trivialisation of $F$, there is an induced orientation on $\ker \conj{\pd}$; applying this to $(E, F) = (u^*TX_C, u|_{\pd D}^*TL_C)$ as $u$ ranges over holomorphic discs of given index, with the homotopy class of trivialisation taken to be that arising from a choice of orientation and spin structure on $L_C$, we obtain a coherent orientation on the moduli space of parametrised holomorphic discs of this index.  We fix the orientation of $\mathfrak{su}(2)$ in which the generators of right-handed rotations about a triple of right-handed axes form a positively oriented basis, and combine this with the canonical identification $TL_C \cong L_C \times \mathfrak{su}(2)$ given by the infinitesimal group action to obtain a homotopy class of trivialisation of $TL_C$, and hence an orientation and spin structure on $L_C$.

In fact this suffices to deal with all other choices of orientation and spin structure, since changing the orientation on $L_C$ cancels out in the definition of the closed--open map, whilst modifying the spin structure by a class $\eps \in H^1(L_C; \Z/2)$ changes the sign attached to a disc $u$ according to the parity of the pairing of $\eps$ with the boundary $\pd u$ (this is a special case of the corresponding result for relative spin structures which we used in \protect \MakeUppercase {P}roposition\nobreakspace \ref {RuleOutChars}, via the connecting homomorphism $H^1(L_C) \rightarrow H^2(X_C, L_C)$).  The discs we are interested in have boundaries which sweep out the rotation of $C$ about the centre of a face, which can be realised as the composition of the rotations about two adjacent vertices.  These two vertex rotations are conjugate in $\Gamma_C=\pi_1(L_C)$, so they become equal when we pass to the abelianisation $H_1(L_C; \Z)$, and we therefore see that the boundaries of interest are multiples of $2$ in $H_1(L_C; \Z)$.  This means that they pair to $0 \in \Z/2$ with any class in $H^1(L_C; \Z/2)$, and hence that the signs of the discs are unaffected by changes of spin structure.

So consider a trajectory contributing to $\CO(\PD(N_C))$ for $C$ equal to $\tri$, $T$ or $O$.  We know from \protect \MakeUppercase {P}roposition\nobreakspace \ref {labCONC} that the trajectory comprises a single disc $u \mc (D, \pd D) \rightarrow (X_C, L_C)$, which is axial of type $\xi_f$ and order $1$, mapping $0$ to $N_C$ and the outgoing marked point to the Morse minimum $m'$.  Let $(E, F)$ be the corresponding Riemann--Hilbert pair, and $V$ the kernel of its Cauchy--Riemann operator, carrying the orientation defined by the orientation and spin structure on $L_C$.  Note that $V$ has (real) dimension $\dim L_C + \mu(u) = 7$.  From \cite[Section A.2.3]{BCEG} we see that the sign with which this trajectory counts is given---up to some overall sign independent of $C$---by the orientation sign of the map
\[
\ev \mc V \rightarrow T_{u(0)}X_C/T_{u(0)}N_C \oplus T_{m'}L_C,
\]
which takes a holomorphic section $v$ to its evaluation at the normal space to $N_C$ at the pole and to the tangent space to $L_C$ at $m'$.

In \protect \MakeUppercase {L}emma\nobreakspace \ref {labInd4AxPars} we saw a splitting of $(E, F)$ by a holomorphic frame $v_1$, $v_2$, $v_3$.  A basis for $V$ is then given by
\begin{multline}
\label{eqSecBasis}
i(1-z)v_1 \text{, } i(1-z)v_2 \text{, } (1-z)^2v_3 \text{, } i(1-z^2)v_3 \text{, } (1+z)v_1+i(1-z)v_2 \text{,} \\ i(1-z)v_1-(1+z)v_2 \text{ and } zv_3.
\end{multline}
The first four elements evaluate to $0$ in $T_{m'}L_C$ and to a positively oriented basis in $T_{u(0)}X_C / T_{u(0)}N_C$ (equipped with the complex orientation).  The final three elements, meanwhile, evaluate to $0$ in the latter space and to $4\alpha \cdot m'$, $4\beta \cdot m'$ and $\gamma \cdot m'$ respectively in the former.  Recall from the proof of \protect \MakeUppercase {L}emma\nobreakspace \ref {labInd4AxPars} that $\alpha$, $\beta$ and $\gamma$ represent infinitesimal right-handed rotations about a right-handed set of axes, so form a positively oriented basis of $\mathfrak{su}(2)$, and hence \eqref{eqSecBasis} is sent by $\ev$ to a positively oriented basis of the codomain.  In other words, $\ev$ is orientation-preserving if and only if \eqref{eqSecBasis} is positively oriented as a basis of $V$.

Now fix a reference disc $u_\tri$ contributing to $\CO(\PD(N_\tri))$, with associated Riemann--Hilbert pair $(E_\tri, F_\tri)$ and $\conj{\pd}$ kernel $V_\tri$.  There is an obvious isomorphism $h$ between $(E, F)$ and $(E_\tri, F_\tri)$, given by sending the frame $v_1$, $v_2$, $v_3$ for $E$ to the corresponding frame for $E_\tri$, and this induces an isomorphism $H \mc V \rightarrow V_\tri$ which sends the basis \eqref{eqSecBasis} to the obvious corresponding basis for $V_\tri$.  We want to show that $u$ contributes to $\CO(\PD(N_C))$ with the same sign as $u_\tri$ contributes to $\CO(\PD(N_C))$, and from the above argument this happens if and only if $H$ is orientation-preserving.  This will follow if we can show that $h$ maps the homotopy class of trivialisation of $F$ induced by the orientation and spin structure on $L_C$ to that of $F_\tri$ induced by the orientation and spin structure on $L_\tri$.

To prove this, note that the frame $\alpha \cdot u|_{\pd D}$, $\beta \cdot u|_{\pd D}$, $\gamma \cdot u|_{\pd D}$ for $F$ is tautologically in the homotopy class induced by the orientation and spin structure on $L_C$.  Similarly for the corresponding frame $\alpha \cdot u_\tri|_{\pd D}$, $\beta \cdot u_\tri|_{\pd D}$, $\gamma \cdot u_\tri|_{\pd D}$ for $F_\tri$.  We can express these frames in terms of the $v_j$ (restricted to $\pd D$, but we drop this from the notation to reduce clutter) as
\[
\frac{(1+z)v_1 + i(1-z)v_2}{4} \text{, } \frac{i(1-z)v_1 - (1+z)v_2}{4} \text{ and } zv_3,
\]
and hence they are carried to each other by $h$, proving that $H$ is orientation-preserving and thus completing the proof of the lemma.
\end{proof}

No such result could hold for the index $2$ discs contributing to $\CO(c_1\qcup)$ without restricting the spin structure, since changing it would reverse the sign.

From this we deduce:

\begin{cor}\label{labCOLO}  If $HF^*(L_O, L_O; k) \neq 0$ over a field $k$ of characteristic $p$ then $p$ must be $2$.  The index $4$ count $D$ for the octahedron is a power of $2$.
\end{cor}
\begin{proof}
We argue as in \protect \MakeUppercase {C}orollary\nobreakspace \ref {labCO}\ref{COitm2}, but now we know that $p$ divides $\chi(\PD(N_C)\qcup) (8)$---the $-8$ alternative is ruled out by \protect \MakeUppercase {L}emma\nobreakspace \ref {labCOSign}.  This forces $p$ to be $2$ or $5$, and we have already seen that $5$ is not permitted by considering $\chi(c_1\qcup)$.  This in turn means that $D$ cannot be divisible by $19$, so it is just a power of $2$.
\end{proof}

Now we can compute the self-Floer cohomology of $L_O$:

\begin{prop} \label{labCOConstr2}  We have an isomorphism of unital rings
\[
HF^0(L_O, L_O; \Z) \cong \Z [x]/(2, x^2+x+1).
\]
\end{prop}
\begin{proof}  We have shown already that $HF^0(L_O, L_O; \Z) \cong \Z [x] /(D, 2x, x^2-\alpha x-\beta)$ for some $D$  in $\{\pm2, \pm4, \pm8, \dots\}$, some $\alpha$ in $\Z/2$, and some $\beta$ in $\Z/(D)$, and that $\CO(2H)=\pm6$ and $\CO(6E)=8$.  Let $\CO(H)=ax+b$ and $\CO(E)=cx+d$, for $a, c \in \Z/2$ and $b, d \in \Z/(D)$.

Consider the reduction modulo $2$ of $\CO$, restricted to the subring of $QH^*(X_O; \Z/2)$ generated by $E$.  This ring is $\Z[E]/(2, E^2-E-1)$, isomorphic to the field $\F_4$ of four elements, so the restriction of $\CO$ to it must be injective.  In particular, since the codomain
\[
HF^0(L_O, L_O; \Z/2) \cong \Z[x]/(2, x^2-\alpha x-\beta)
\]
also has four elements we see that it too must be isomorphic to $\F_4$, and hence that $\alpha$ and $\beta$ are both odd.  Moreover, the coefficient $c$ of $x$ in $\CO(E)$ must be odd.

If we can show that $D=\pm2$ then the above argument gives the claimed form of $HF^0(L_O, L_O; \Z)$, so suppose for contradiction that $D$ is a multiple of $4$.  Since $\CO(6E)=8$, we deduce that the coefficient $d$ of the unit in $\CO(E)$ must be even.  Similarly, since $\CO(2H)=\pm6$ the coefficient $b$ must be odd.  Now applying $\CO$ to the relation $H^2 = 5E+3$ we obtain
\[
(a^2\alpha-5c)x+(a^2\beta+b^2-5d-3)=0.
\]
From the coefficient of $x$ we see that $a$ must be odd, whilst from the coefficient of the unit we see that $a$ must be even, giving the desired contradiction.  Therefore $D=\pm2$ and $HF^0(L_O, L_O; \Z)$ is as claimed.
\end{proof}

The fact that the signed count $D$, of index $4$ discs through two generic points $p$ and $q$ of $L_O$, is $\pm2$ can be understood as follows.  From Section\nobreakspace \ref {sscDegCont} we know that any such disc completes to a rational curve of degree $2$, which must therefore be contained in some $2$-plane in $\P S^6V$.  This curve passes through $p$ and $q$ tangent to $X_O$ (since the whole curve is contained in $X_O$), and generically these two tangent $3$-planes $T_pX_O, T_qX_O \subset \P S^6V$ meet in a single point, $r$.  The plane of the curve is then spanned by $p$, $q$ and $r$.  The fact that the count is $\pm 2$ tells us that the intersection of the plane $\lspan{p, q, r}$ with $X_O$ is indeed a degree $2$ curve with equator on $L_O$.  We already remarked in Section\nobreakspace \ref {sscModInv} that the two hemispheres should count with the same sign.

For the triangle, tetrahedron and icosahedron, some work is required to compute the index $4$ contribution to the Floer differential $\diff^1$ and hence evaluate the Floer cohomology.  This is the subject of the remainder of the paper.

\subsection{Triangle}
\label{sscInd4Tri}

Let $p=[x^3+y^3]$ and $q=[x^3-y^3]$, representing equilateral triangles on the equator of $\P V$ which differ by a rotation through angle $\pi$ about a vertical axis.  These are the (unperturbed) points $m'$ and $m$ respectively from Section\nobreakspace \ref {sscMPTri}.

\begin{prop} \label{labInd4Tripq}  $(q, p)$ is a regular value of the two-point index $4$ evaluation map $\ev_2 \mc M_4 \rightarrow L_\tri^2$, with exactly two preimages.
\end{prop}
\begin{proof}  There are two axial discs of type $\xi_f$ and order $1$ passing through $p$ and $q$, and by \protect \MakeUppercase {C}orollary\nobreakspace \ref {labInd4Trans} $\ev_2$ is a submersion at these discs.  They are reflections of each other and their boundaries sweep the rotation from $p$ to $q$ to $p$ about a vertical axis in either direction, through a total angle of $2\pi/3$.  We now check that there are no other index $4$ discs through $p$ and $q$.  It is easy to see that there can be no axial discs of type $\xi_v$ and order $2$ (since the configurations $p$ and $q$ do not have a common vertex), so by \protect \MakeUppercase {C}orollary\nobreakspace \ref {labInd4Pol} we are left to rule out discs with two poles of type $\xi_v$ and order $1$.

Suppose for contradiction then that $u \mc (D, \pd D) \rightarrow (X_\tri, L_\tri)$ is such a disc passing through $p$ and $q$.  Its double $\double{u}$ is a rational curve in $X_\tri$ of degree $2$ (by the results of Section\nobreakspace \ref {sscDegCont}), so is either a double cover of a line or is a smooth conic.  If the former, the image of $\double{u}$ would be the line through $p$ and $q$, but this line does not intersect $Y_\tri \setminus N_\tri$.  We know, however, that $\double{u}$ \emph{does} meet this set, at the poles of $u$, so this case is impossible.  Therefore $\double{u}$ must be a smooth conic, and hence, in particular, an embedding.

The poles of $u$ reflect to poles of $\double{u}$ of type $\xi_e$ (and order $1$), and evaluate to points
\[
P \coloneqq [(ax+y)^3] \text{ and } Q \coloneqq [(bx+y)^3] \in N_\tri
\]
for some $a, b \in \C\P^1$, which must be distinct as $\double{u}$ is injective.  Since $\double{u}$ has degree $2$, its image is contained in a $2$-plane, and hence the points $p$, $q$, $P$ and $Q$ must be coplanar.  If $a$ and $b$ are finite and non-zero then applying this condition we see that $a=b$, contradicting the fact that they must be distinct.  We may therefore assume without loss of generality that $a=\infty$ and that $b$ is finite.  The case where one of $a$ and $b$ is zero is analogous.  Then the poles of type $\xi_v$ evaluate to
\[
R \coloneqq [(x+cy)y^2] \text{ and } S \coloneqq [((bx+y)+d(x-\conj{b}y))(x-\conj{b}y)^2] \in Y_C \setminus N_C,
\]
for some $c, d \in \C$, using that $R$ and $S$ lie in $Y_\tri \setminus N_C$ and satisfy $\tau(R)=P$ and $\tau(S)=Q$.

From coplanarity of $p$, $q$, $P$, $Q$ and $R$ (all lie in the image of $\double{u}$) we deduce that $b=0$.  But then $P$, $Q$, $R$ and $S$ have standard coordinates $[1:0:0:0]$, $[0:0:0:1]$, $[0:0:1:c]$ and $[d:1:0:0]$, so cannot be coplanar, giving a contradiction.  Hence no such two-pole index $4$ disc $u$ can exist, and we're done.
\end{proof}

Note also that there is no index $4$ bubbled configuration through $p$ and $q$; in fact this is precisely the right-hand example given in Fig.\nobreakspace \ref {figBubConfs} in Section\nobreakspace \ref {sscBubConf}.  Hence there is an open neighbourhood $U$ of $(q, p)$ in $L_\tri^2$ such that each point in $U$ is a regular value of $\ev_2$, and the local degree (i.e.~signed count of preimages) remains constant on $U$.

We can now compute:

\begin{cor}\label{labTriHF}  The index $4$ count $\hat{Z}$ appearing in Section\nobreakspace \ref {sscMPTri} is $\pm 2$, and the determinant $3XY-4\hat{Z}$ of the Floer differential $\diff^1$ is $\pm 5$.  The self-Floer cohomology ring of $L_\tri$ over $\Z$ satisfies
\[
HF^0(L_\tri, L_\tri; \Z) \cong \Z/(5) \text{ and } HF^1(L_\tri, L_\tri; \Z) = 0.
\]
If $k$ is a field of characteristic $5$ then we have additive isomorphisms
\[
HF^0(L_\tri, L_\tri; k) \cong HF^1(L_\tri, L_\tri; k) \cong k.
\]
\end{cor}
\begin{proof}  By \protect \MakeUppercase {P}roposition\nobreakspace \ref {labInd4Tripq}, and the absence of bubbled configurations, we have $\hat{Z}=\pm 2$ (if the two discs count with the same sign) or $0$ (if they count with opposite signs).  And we saw in Section\nobreakspace \ref {sscCOII} that $3XY-4\hat{Z}$ must be a power of $5$.  Recalling that $X, Y \in \{\pm 1\}$, the only possibility is that $\hat{Z}\in\{\pm2\}$ and $3XY-4Z\in \{\pm5\}$.  Plugging the latter into the Smith normal form of the Floer differential calculated in Section\nobreakspace \ref {sscMPTri} gives the claimed cohomology.
\end{proof}

\subsection{Tetrahedron}
\label{sscInd4T}

Now let $p=[x^4+2\sqrt{3}x^2y^2-y^4]$ and $q=[x^4-2\sqrt{3}x^2y^2-y^4]$, representing regular tetrahedra with an opposite pair of horizontal edges, differing by rotation through angle $\pi/2$ about a vertical axis.  These are the points $m'$ and $m$ respectively from Section\nobreakspace \ref {sscMPT}.

\begin{prop} \label{labInd4Tpq}  $(q, p)$ is not in the image of of the two-point index $4$ evaluation map $\ev_2 \mc M_4 \rightarrow L_T^2$ (and so is vacuously a regular value).
\end{prop}
\begin{proof}  Since the configurations $p$ and $q$ have no vertex in common there are no axial discs of type $\xi_v$ passing through them both.  Similarly since there is no face of $p$ which differs from a face of $q$ by rotation about its centre there are no axial discs of type $\xi_f$ passing through $p$ and $q$.  So now suppose for contradiction that $u \mc (D, \pd D) \rightarrow (X_T, L_T)$ is a two-pole index $4$ disc through $p$ and $q$ with double $\double{u}$.  This time the two poles of type $\xi_v$ from $u$ reflect to poles of type $\xi_f$.  Again $\deg \double{u} = 2$ but now we can rule out the double cover of a line for more trivial reasons: the line in $\P S^4 V$ through $p$ and $q$ does not lie in $X_T$ (it contains the point $[x^2y^2]$ for example).  Hence $\double{u}$ is an embedding.

Considering the points
\[
P \coloneqq [(ax+y)^4] \text{ and } Q \coloneqq [(bx+y)^4] \in N_T,
\]
to which the poles of type $\xi_f$ evaluate, and the fact that they must be coplanar with $p$ and $q$, we get either $a=0$ and $b=\infty$ (or vice versa) or that $a$ is equal to $-b$ and is a fourth root of $-1$.  Each of these cases leads to a contradiction by looking at the possible reflections of $P$ and $Q$, as with $C=\tri$.
\end{proof}

Again there is no bubbled configuration through $p$ and $q$.  To see this recall from Section\nobreakspace \ref {sscMPT} that the vertices of $p$ are at $\pm \sqrt{2}/(\sqrt{3}-1)$ and $\pm (\sqrt{3}-1)i/\sqrt{2}$.  Those of $q$ differ by multiplication by $i$ (rotation by $\pi/2$ about a vertical axis), so we can explicitly compute the sets \eqref{eqpqSet} and \eqref{eqCSet} from Section\nobreakspace \ref {sscBubConf}.  The former is $\{0, 2/3\}$, whilst the latter is $\{1/3, 1\}$, and these are clearly disjoint.  We can therefore perturb $p$ and $q$ slightly without introducing any preimages.  We get:

\begin{cor}\label{labTHF}  The index $4$ count $Z$ in Section\nobreakspace \ref {sscMPT} is $0$, and the determinant $4XY-3Z$ of $\diff^1$ is $\pm 4$.  The self-Floer cohomology ring of $L_T$ over $\Z$ satisfies
\[
HF^0(L_T, L_T; \Z) \cong \Z/(4) \text{ and } HF^1(L_T, L_T; \Z) = 0.
\]
If $k$ is a field of characteristic $2$ then we have additive isomorphisms
\[
HF^0(L_T, L_T; k) \cong HF^1(L_T, L_T; k) \cong k.
\]
\end{cor}
\begin{proof}  The value of $Z$ follows immediately from \protect \MakeUppercase {P}roposition\nobreakspace \ref {labInd4Tpq} (plus the absence of bubbled configurations), and then, recalling that $X, Y \in \{\pm 1\}$, the determinant must be $\pm 4$.  Substituting into the Floer differential in Section\nobreakspace \ref {sscMPT} gives the cohomology.
\end{proof}

\subsection{Icosahedron}
\label{sscInd4I}

Let $p$ and $q$ be given by $[x^{12}\mp 11\sqrt{5} x^9y^3 - 33x^6y^6\pm 11\sqrt{5} x^3y^9 + y^{12}]$, representing regular icosahedra with an opposite pair of horizontal faces, differing by rotation through angle $\pi$ about a vertical axis.  Note that $p$ is the configuration $I_f$ from Appendix\nobreakspace \ref {secExplReps} whilst $q$ is obtained from $p$ by the rotation $x \mapsto -x$, $y \mapsto y$.

\begin{prop} \label{labInd4Ipq}  $(q, p)$ is a regular value of the two-point index $4$ evaluation map $\ev_2 \mc M_4 \rightarrow L_I^2$, with signed count of preimages $\pm 8$.
\end{prop}

\begin{proof}  The argument is rather technical so is relegated to Appendix\nobreakspace \ref {secInd4Icos}.  The reason this computation is difficult is that in addition to the two obvious axial discs of type $\xi_f$, whose boundaries rotate $p$ about a vertical axis to $q$ and then on to $p$ again (just as for the triangle), there are $6$ other discs that occur in the preimage.  Constructing these discs, proving that there can be no others, and showing that all $8$ discs count with the same sign, is not easy.
\end{proof}

Once again there is no bubbled configuration through $p$ and $q$---see \protect \MakeUppercase {L}emma\nobreakspace \ref {INoBub}---so the local degree is constant near $(q, p)$.  We have thus proved:

\begin{cor}\label{labIHF}  The index $4$ count $D$ for the icosahedron in Section\nobreakspace \ref {sscMPOI} is $\pm 8$, and hence the self-Floer cohomology ring of $L_I$ over $\Z$ satisfies
\[
HF^0(L_I, L_I; \Z) \cong \Z/(8) \text{ and } HF^1(L_I, L_I; \Z) = 0.
\]
\end{cor}

\appendix
\section{Transversality for the pearl complex}
\label{secTransPearl}

\subsection{Preliminaries}  In this appendix we discuss the transversality results required to set up the pearl complex.  This is mainly to show how to work with a \emph{fixed} complex structure, but also includes a brief review of Biran--Cornea's foundational work in \cite{BCQS} based on \emph{generic} almost complex structures, in order to compare and contrast the two approaches and show that they give the same (co)homology.

We begin by recalling some basic notions in differential topology.  A standard reference is \cite[Chapter 2]{Hir}, on which we base our terminology.  For topological spaces $P$ and $Q$ let $C^0(P, Q)$ denote the space of continuous maps from $P$ to $Q$, equipped with the compact-open topology.  If $P$ and $Q$ are actually smooth manifolds then for each $r \in \{1, 2, \dots, \infty\}$ we can form the $r$-jet space $J^r(P, Q)$, and the space $C^r(P, Q)$ of $r$-times continuously differentiable maps from $P$ to $Q$ (or smooth maps in the case $r = \infty$) embeds in $C^0(P, J^r(P, Q))$ via the prolongation map $j^r$.  This gives a natural topology on $C^r(P, Q)$, which we call the (weak) $C^r$-topology, which is induced by a (non-canonical) complete metric $d_r$.  When $P$ is compact this topology coincides with the strong $C^r$-topology, sometimes called the Whitney topology (although the reader is warned that terminology varies between different authors), and the set of (smooth) diffeomorphisms $\Diff(P)$ is open in $C^\infty(P, P)$.

Given manifolds $P$ and $Q$, a submanifold $R \subset Q$, subsets $A \subset P$, $B \subset R$, and a smooth map $f \mc P \rightarrow Q$, let $f \trans{A}{B} R$ denote that $f$ is transverse to $R$ along $A \cap f^{-1}(B)$.  In other words, for all points $p$ in $A$ such that $f(p)$ is in $B$ we have
\[
\im D_p f + T_{f(p)}R = T_{f(p)}Q.
\]
Note this differs from the notation used in \cite{Hir}, where $f \pitchfork_K R$ denotes what we are calling $f \trans{K}{R} R$.

Fix now a closed $n$-manifold $L$.  Later this will be our Lagrangian, but for our present purposes this is irrelevant.  For a positive integer $s$ let $\Delta_{L, s}$ denote the big diagonal
\[
\{ (p_1, \dots, p_s) \in L^s : p_j = p_k \text{ for some } j\neq k\} \subset L^s.
\]
The key result in differential topology we shall use is the following:

\begin{lem}\label{labPullTrans2}  For any countable collections of manifolds $(M_j)$, positive integers $(s_j)$, submanifolds $(N_j\subset L^{s_j} \setminus \Delta_{L, s_j})$, and smooth maps $(f_j \mc M_j \rightarrow L^{s_j})$, there exists a diffeomorphism $\phi$ of $L$, arbitrarily $C^\infty$-close to $\id_L$, such that for all $j$ the map $f_j$ is transverse to $\phi^{\times s_j}(N_j)$.
\end{lem}

This will follow from:

\begin{lem}\label{labDenseOpen}  In the setup of \protect \MakeUppercase {L}emma\nobreakspace \ref {labPullTrans2}, for all $j$ and all $p \in M_j$ and $q \in N_j$ there exist neighbourhoods $U_{j, p, q}$ of $p$ in $M_j$ and $V_{j, p, q}$ of $q$ in $N_j$, such that the set
\[
W_{j, p, q} \coloneqq \{ \phi \in \Diff(L) : (\phi^{-1})^{\times s_j} \circ f \trans{U_{j, p, q}}{V_{j, p, q}} N_j\}
\]
is open and dense in $\Diff(L)$ in the $C^\infty$-topology.
\end{lem}

To deduce \protect \MakeUppercase {L}emma\nobreakspace \ref {labPullTrans2} from \protect \MakeUppercase {L}emma\nobreakspace \ref {labDenseOpen}, for each $j$ we simply take a countable subcover $\{U_{j, p_k, q_k} \times V_{j, p_k, q_k}\}_k$ of the cover $\{U_{j, p, q} \times V_{j, p, q}\}_{(p, q)}$ of $M_j \times N_j$ and then consider the intersection
\[
\bigcap_{j, k} W_{j, p_k, q_k} \subset \Diff(L).
\]
Since $\Diff(L)$ is an open subset of the complete metric space $C^\infty(L, L)$, this intersection is dense in $\Diff(L)$ by the Baire category theorem, so in particular it contains elements arbitrarily $C^\infty$-close to $\id_L$.  Such elements provide the $\phi$ of \protect \MakeUppercase {L}emma\nobreakspace \ref {labPullTrans2}.

\begin{proof}[Proof of \protect \MakeUppercase {L}emma\nobreakspace \ref {labDenseOpen}]
Fix arbitrary $j$, $p$ and $q$, and metrics on $L$ and $M_j$.  From now on we shall drop all $j$'s from the notation, and just refer to $M_j$, $s_j$, $N_j$ and $f_j$ as $M$, $s$, $N$ and $f$ respectively.  Let $\pi_1, \dots, \pi_s \mc L^s \rightarrow L$ denote the projections onto the factors, and choose vectors $v_1, \dots, v_a \in T_q L^s$ which form a basis for a complement to $T_qN$.  For each $i$ use cutoff functions to construct a smooth vector field $V_i$ on $L$ whose value at $\pi_k(q)$ coincides with the projection $D_q\pi_k(v_i)$ for all $k$.

Now consider the map
\begin{align*}
\psi \mc \R^a &\rightarrow C^\infty(L, L)
\\ \mathbf{t} &\mapsto \exp\lb \sum_i t_i V_i \rb,
\end{align*}
which sends a vector $\mathbf{t}$ to the time $1$ flow of $\sum t_i V_i$.  It is easy to check that given three topological spaces $P$, $Q$ and $R$, and a continuous map $h \mc P \times Q \rightarrow R$, the map $h_{\ev} \mc P \rightarrow C^0(Q, R)$ given by $x \mapsto h(x, \cdot)$ is continuous.  Since the map
\begin{align*}
\R^a \times L &\rightarrow L
\\(\mathbf{t}, x) &\mapsto \psi(\mathbf{t})(x)
\end{align*}
is smooth, and hence defines a continuous map $j^\infty_L \psi \mc \R^a \times L \rightarrow J^\infty(L, L)$ by prolongation along the $L$ factor, we deduce that $\psi = (j^\infty_L \psi)_{\ev}$ is continuous.

By construction of the $v_i$ and $V_i$, the map
\begin{align*}
\Psi \mc \R^a \times N &\rightarrow L^s
\\ (\mathbf{t}, x) &\mapsto \psi(\mathbf{t})^{\times s}(x)
\end{align*}
is a submersion at the point $(0, q)$, and along $\{0\} \times N$ it is simply the inclusion of $N$.  There therefore exist an open ball $B$ in $N$ about $q$ and a positive $\eps$ such that $\Psi$ gives a diffeomorphism from $B^a_0(\eps) \times B$ onto an open tubular neighbourhood $T$ of $B$ in $L^s$, where $B^a_0(\eps)$ is the open ball of radius $\eps$ about $0$ in $\R^a$.  Let $\pi \mc T \rightarrow B^a_0(\eps)$ be the composition of the inverse diffeomorphism $\Psi^{-1}$ with projection onto the first factor.  Now pick a smooth cutoff function $\rho$ on $L^s$ which has compact support contained in $T$ and takes the value $1$ on a compact neighbourhood $T'$ of $q$ in $L^s$.  Let $V$ be a compact neighbourhood of $q$ in $N$ such that $T'$ contains a neighbourhood of $V$ in $L^s$.  Figure\nobreakspace \ref {figTransSetup} shows this setup.
\begin{figure}[ht]
\centering
\includegraphics[scale=1]{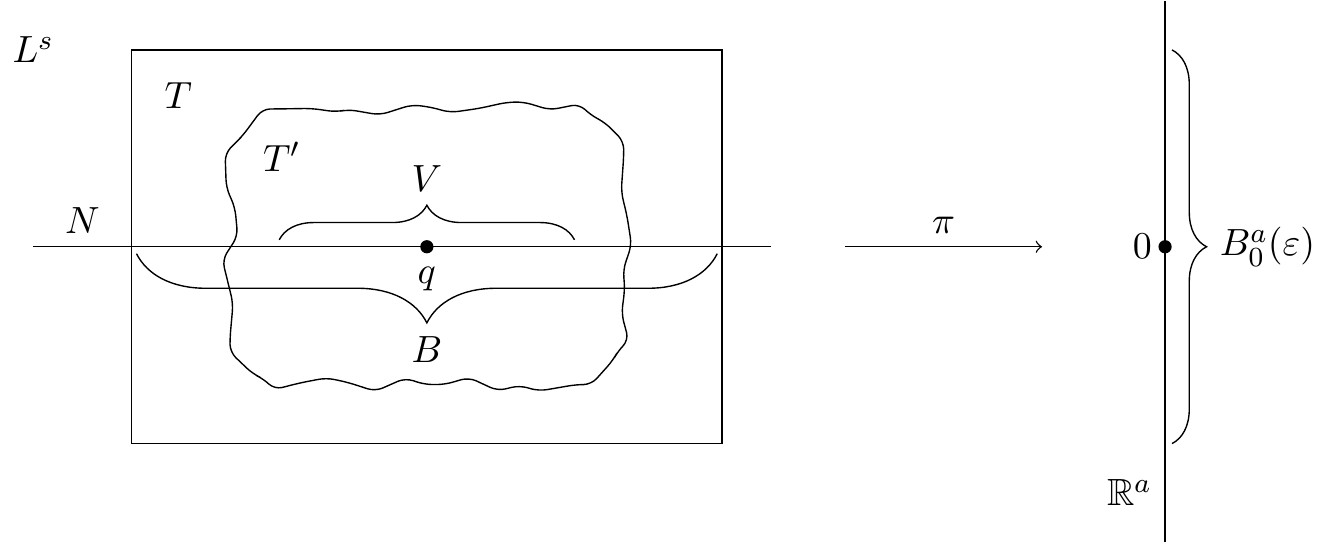}
\caption{The tubular neighbourhood $T$ and projection $\pi$.\label{figTransSetup}}
\end{figure}

Let $U$ be an arbitrary compact neighbourhood of $p$ in $M$.  We claim that $U_{j, p, q}=U$ and $V_{j, p, q}=V$ have the desired properties, so let
\[
W = \{ \phi \in \Diff(L) : (\phi^{-1})^{\times s} \circ f \trans{U}{V} N\}.
\]
We need to show that this set is open and dense in $\Diff(L)$.

First we prove it is open, so take a diffeomorphism $\phi \in W$ and let $F=(\phi^{-1})^{\times s} \circ f$.  We wish to show that if $\widehat{\phi}$ is sufficiently $C^\infty$-close to $\phi$ then $\widehat{\phi}$ is also in $W$.  In fact we will prove the stronger statement that if a smooth map $G \mc M \rightarrow L^s$ is sufficiently $C^1$-close to $F$ then $G \trans{U}{V} N$.  Let $U' = U \cap F^{-1} (T')$ be the preimage of $T'$ in $U$.  Note that $U'$ is a closed subset of the compact set $U$, so is itself compact.  Since $L^s \setminus T'$ is bounded away from $V$, if $G$ is sufficiently $C^0$-close to $F$ then $G(U \setminus U')$---which is contained in $G (F^{-1}(L^s \setminus T'))$---is disjoint from $V$.  To show that $G \trans{U}{V} N$ in this case it therefore suffices to check that $G \trans{U'}{V} N$.

Given such a map $G$ ($C^0$-close to $F$) and a point $x$ in $U'$, consider the derivative $D_x(\pi \circ G) \mc T_xM \rightarrow \R^a$ of $\pi \circ G$ at $x$.  This sends the unit ball in $T_xM$ (with respect to our metric on $M$) to a subset $S_x$ of $\R^a$ containing $0$.  Let
\[
r_G(x) = \sup \{ r \in \R_{\geq 0} : B^a_0(r) \subset S_x\}
\]
be the supremum of the radii of the balls about $0$ in $\R^a$ which are contained in this subset.  Note that the map $r_G \mc U' \rightarrow \R_{\geq 0}$ is continuous, and satisfies $r_G(x) > 0$ if and only if $D_x (\pi \circ G)$ is surjective.  Now consider the map
\begin{align*}
R_G \mc U' &\rightarrow \R_{\geq 0}^2
\\ x &\mapsto (r_G(x), d(G(x), V)),
\end{align*}
where $d(G(x), V)$ denotes the distance (with respect to the metric on $L^s$ coming from our metric on $L$) between the point $G(x)$ and the set $V$.  This map is continuous and vanishes precisely at points of $U'$ where $G \trans{U'}{V} N$ fails.  Crucially $R_G$ is also continuous in $G$ in the $C^1$-topology.  Since $F \trans{U'}{V} N$ by hypothesis, $R_F$ is nowhere zero and thus by compactness of $U'$ its image is bounded away from zero.  Therefore the same is also true of $R_G$ for $G$ sufficiently $C^1$-close to $F$.  In other words, such $G$ satisfy $G \trans{U'}{V} N$, proving our openness claim.

We now show that $W$ is dense in $\Diff(L)$, so take any diffeomorphism $\phi$ of $L$ and again let $F=(\phi^{-1})^{\times s} \circ f$.  We need to construct a diffeomorphism $\widehat{\phi}$, arbitrarily $C^\infty$-close to $\phi$, which is contained in $W$.  Equivalently, we need a $\widehat{\phi}$ arbitrarily $C^\infty$-close to $\id_L$ such that $F \trans{U}{\widehat{\phi}^{\times s}(V)} \widehat{\phi}^{\times s}(N)$.  Since the map $\psi \mc \R^a \rightarrow C^\infty(L, L)$ is continuous, it is enough to show that $F \trans{U}{V_{\mathbf{t}}} N_\mathbf{t}$ for arbitrarily small choices of $\mathbf{t}$, where $N_\mathbf{t} = \psi(\mathbf{t})^{\times s}(N)$ and $V_\mathbf{t} = \psi(\mathbf{t})^{\times s}(V)$.

Define a smooth map $\widetilde{\pi} \mc L^s \rightarrow B^a_0(\eps)$ by
\[
\widetilde{\pi} (x) = \begin{cases} \rho(x) \pi(x) & \text{if } x \in T \\ 0 & \text{otherwise,} \end{cases}
\]
where $\rho$ is our cutoff function on $L^s$.  This coincides with $\pi$ on a neighbourhood of $V$, so for $\mathbf{t}$ sufficiently small and $x \in V_\mathbf{t}$ we have that $\widetilde{\pi}(x) = \mathbf{t}$ and $T_x N_\mathbf{t} = \ker D_x \widetilde{\pi}$ (by definition of $\Psi$, the map $\psi(\mathbf{t})^{\times s}$ on $B \subset N$ corresponds under $\Psi^{-1}$ to translation by $\mathbf{t}$).  In particular, for $\mathbf{t}$ small $F$ is transverse to $N_\mathbf{t}$ along $V_\mathbf{t}$ if and only if $\mathbf{t}$ is a regular value of $\widetilde{\pi} \circ F$.  By Sard's theorem, such regular values exist arbitrarily close to $0$, completing the proof of density and thus of \protect \MakeUppercase {L}emma\nobreakspace \ref {labDenseOpen}.
\end{proof}

\subsection{Constructing the complex}  Suppose $X$ is a closed symplectic manifold and $L \subset X$ is a closed, connected, monotone Lagrangian with minimal Maslov number $N_L \neq 1$.  Fix a coefficient ring $R$; if the characteristic of $R$ is not $2$ then we also assume that $L$ admits an orientation and spin structure, and fix a choice of these.

Let $(f, g)$ be a Morse--Smale pair on $L$, and $J$ an $\omega$-compatible almost complex structure on $X$.  For a tuple $\mathbf{A} = (A_1, \dots, A_r) \in H_2(X, L; \Z)^r$, with $r > 0$, let $z(\mathbf{A})$ be the number of $j$ for which $A_j = 0$.  Let $W^a_x$ and $W^d_y$ be respectively the ascending and descending manifolds of critical points $x$ and $y$ of $f$, let $\Phi_t$ denote the time $t$ flow of $\nabla f$, and let $Q$ be the submanifold of $L \times L$ given by
\[
Q = \{ (p, q) \in L \times L : p \notin \crit(f) \text{, } q = \Phi_t(p) \text{ for some } t > 0\} \cong \lb L \setminus \crit(f) \rb \times \R_{>0}.
\]
For a class $A \in H_2(X, L; \Z)$, recall that $\mtw(A)$ denotes the moduli space of unparametrised $J$-holomorphic discs $u \colon (D, \partial D) \rightarrow (X, L)$ representing $A$, with two (distinct) boundary marked points.  Note, however, that in contrast with most of the rest of the paper the $J$ here need not be integrable.  Let the evaluation maps at the two marked points be $\ev_{\pm}(A) \colon \mtw(A) \rightarrow L$.

Now define, for $x$, $y$ and $\mathbf{A}$ as above, the pearly trajectory (or \emph{string of pearls}) moduli space
\begin{multline}
\label{eqPearly}
\mathcal{P}(x, y, \mathbf{A}) = \Big(\big(\ev_-(A_1) \times \ev_+(A_1) \times \ev_-(A_2) \times \dots \\ \times \ev_+(A_r)\big)^{-1}\lb W^a_x \times Q^{r-1} \times W^d_y\rb\Big)\Big/\R^{z(\mathbf{A})},
\end{multline}
where the $\R^{z(\mathbf{A})}$ acts by translation of constant discs (corresponding to classes $A_j=0$) along flowlines.  None of the flowlines is actually doubly infinite in time so strictly each constant disc can only be translated by a subinterval of $\R$, but we overlook this slight notational imprecision.  Such configurations are illustrated in Fig.\nobreakspace \ref {figStringPearls}; the arrows depict the flowlines of $\nabla f$.
\begin{figure}[ht]
\centering
\includegraphics[scale=1]{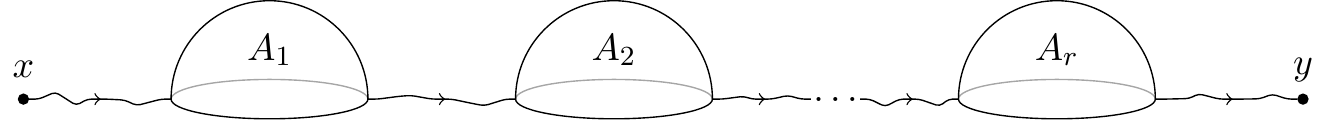}
\caption{A pearly trajectory, or string of pearls.\label{figStringPearls}}
\end{figure}
Really we may restrict our attention to \emph{reduced} strings, for which $A_j \neq 0$ for all $j$ unless $r=1$, and transversality for these spaces automatically gives transversality for the same spaces with extra constant discs inserted (the case $r=1$, $A_1=0$ gives rise to standard Morse trajectories).  However, it is notationally convenient to allow any number of constant discs.

We shall also need moduli spaces of strings of pearls with \emph{loose ends}, of the form
\begin{equation}
\label{eqLE1} W^a_x(\mathbf{A}) \coloneqq \Big(\big(\ev_-(A_1) \times \ev_+(A_1) \times \ev_-(A_2) \times \dots \times \ev_-(A_r)\big)^{-1}\lb W^a_x \times Q^{r-1}\rb \Big)\Big/\mathbb{R}^{z^a(\mathbf{A})}
\end{equation}
\begin{equation}
\\ \label{eqLE2} W^d_y(\mathbf{A}) \coloneqq \Big(\big(\ev_+(A_1) \times \ev_-(A_2) \times \ev_+(A_2) \times \dots \times \ev_+(A_r)\big)^{-1}\lb Q^{r-1} \times W^d_y\rb \Big)\Big/\R^{z^d(\mathbf{A})},
\end{equation}
where $z^a(\mathbf{A})$ is the number of $j \leq r-1$ with $A_j=0$ and $z^d(\mathbf{A})$ is the number of $j \geq 2$ with $A_j=0$---we quotient out by translation of constant discs except those which are at the ends of trajectories.  These configurations are illustrated in Fig.\nobreakspace \ref {figLooseEnds}.
\begin{figure}[ht]
\centering
\includegraphics[scale=1]{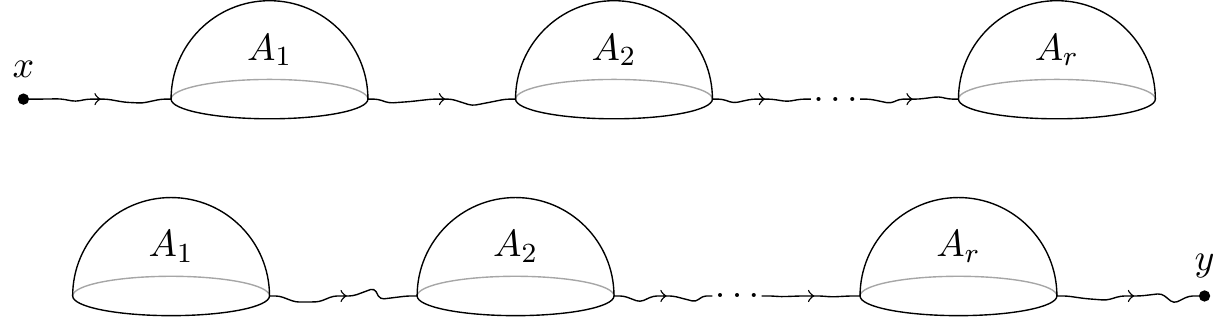}
\caption{Strings of pearls with loose ends.\label{figLooseEnds}}
\end{figure}
Note that these spaces carry evaluation maps $\ev(x, \mathbf{A})$ and $\ev(\mathbf{A}, y)$ at the loose marked point of the end disc.  Again we could restrict to reduced trajectories, where all discs but the end one are non-constant, but we shall not do so for now.

In order to define the pearl complex for $L$ using the data $(f, g, J)$, we need the following:
\begen
\item\label{JisReg} $J$ is regular, meaning that the moduli spaces $\mtw(A)$ are cut out transversely.  This ensures that the $\mtw(A)$ are all smooth manifolds of the correct dimension.
\item\label{StringTrans0} The moduli spaces \eqref{eqPearly} of virtual dimension $0$ are cut out transversely, so that the spaces used to define the differential $\diff$ (as described below) are smooth manifolds of the correct dimension.
\item\label{StringTrans1} The same requirement as \ref{StringTrans0} but in virtual dimension $1$, so that we can construct the moduli spaces used to prove $\diff^2=0$ (again, see below).  We'll unimaginatively call these `$\diff^2=0$ moduli spaces'.
\item\label{BoundPts} The moduli spaces \eqref{eqPearly}, with some of the $Q$ factors (possibly none or all of them) replaced by copies of the diagonal $\Delta_L \in L \times L$, are cut out transversely whenever their virtual dimension is at most $0$.  This means that the moduli spaces in \ref{StringTrans0} are compact, and that those in \ref{StringTrans1} can be compactified by introducing strings of pearls which are degenerate in exactly one of the following ways: a single Morse flowline has broken, a single disc has bubbled into two (with one marked point in each component), or a single flowline has shrunk to zero.
\item\label{PearlBrokGlue} Given a broken string of pearls $\gamma$ in virtual dimension $0$ (in which a single flowline is broken, but which is otherwise non-degenerate), as illustrated in Fig.\nobreakspace \ref {figPearlyMorseBr}, we need the loose end spaces
\[
W^a_x(\mathbf{A'}=(A_1, \dots, A_k)) \text{ and } W^d_y(\mathbf{A''}=(A_{k+1}, \dots, A_r))
\]
to be transversely cut out, so that they are smooth manifolds of the correct dimension (note we may always assume that $k$ is not $0$ or $r$, by introducing extra classes equal to zero at the start and end of $\mathbf{A}$).
\begin{figure}[ht]
\centering
\includegraphics[scale=1]{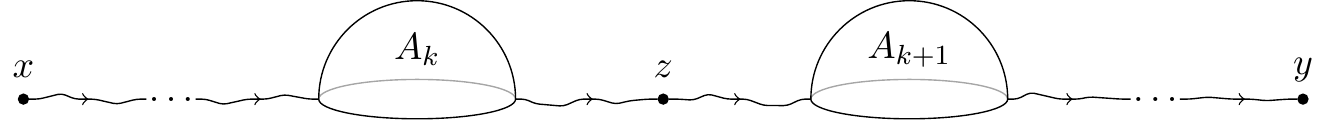}
\caption{A broken string of pearls.\label{figPearlyMorseBr}}
\end{figure}
Given this, we automatically have by \ref{StringTrans0} and by \ref{StringTrans1} that $\ev(x, \mathbf{A'}) \times \ev(\mathbf{A''}, y)$ is transverse to $W^d_z \times W^a_z$ and to $Q$ respectively.  Standard Morse-theoretic gluing arguments, as given in \cite[Proposition 3.2.8]{AuDam} for example, then show that every such broken string $\gamma$ occurs as a unique boundary point in the compactification of the $\diff^2=0$ moduli spaces.  In fact, we only need the loose end spaces to be cut out transversely in neighbourhoods of the points appearing in $\gamma$, viewed as an element of
\[
\big(\ev(x, \mathbf{A'}) \times \ev(\mathbf{A''}, y)\big)^{-1}\lb W^d_z \times W^a_z\rb \subset W^a_x(\mathbf{A'}) \times W^d_y(\mathbf{A''}).
\]
\item\label{PearlBubGlue} Given a bubbled string of pearls $\gamma$ in virtual dimension $0$, with a single disc---the $k$th---bubbled into two (of classes $A_k'$ and $A_k''=A_k-A_k'$) but otherwise non-degenerate, as illustrated in Fig.\nobreakspace \ref {figPearlyDiscBub}, we need the loose end spaces
\[
W^a_x(\mathbf{A'}=(A_1, \dots, A_{k-1}, 0)) \text{ and } W^d_y(\mathbf{A''}=(0, A_{k+1}, \dots, A_r))
\]
to be cut out transversely.  By \ref{BoundPts}, the maps
\[
\ev(x, \mathbf{A'}) \times (\text{inclusion of }\Delta_L \subset L \times L) \times \ev(\mathbf{A''}, y)
\]
and
\[
\ev_-(A_k') \times \ev_+(A_k') \times \ev_-(A_k'') \times \ev_+(A_k'')
\]
are transverse, so by the gluing theorem for $J$-holomorphic discs \cite[Theorem 4.1.2]{BCQS} each such bubbled string occurs as a unique boundary point in the compactification of the $\diff^2=0$ moduli spaces.  Again, we only need this transversality in neighbourhoods of the points of $W^a_x(\mathbf{A'})$ and $W^d_y(\mathbf{A''})$ appearing in $\gamma$.
\begin{figure}[ht]
\centering
\includegraphics[scale=1]{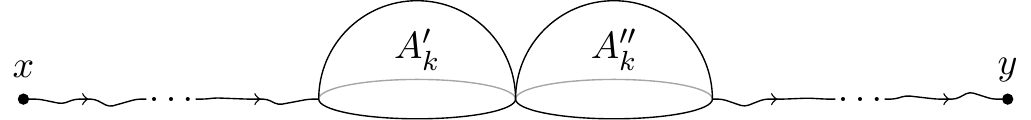}
\caption{A once-bubbled string of pearls.\label{figPearlyDiscBub}}
\end{figure}
\end{enumerate}

If all of these conditions are satisfied then the differential can be defined by counting rigid reduced strings of pearls, which form compact zero-dimensional manifolds (compactness for fixed $\mathbf{A}$ follows from \ref{BoundPts}, whilst the fact that only finitely many reduced choices of $\mathbf{A}$ give non-empty moduli spaces of virtual dimension $0$ follows from Gromov compactness).  The relation $\diff^2=0$ is proved by considering one-dimensional moduli spaces of pearly trajectories, which can be compactified to compact one-manifolds by adding in boundary points as described in \ref{BoundPts}.  The first two types of boundary point appear exactly once each, by \ref{PearlBrokGlue} and \ref{PearlBubGlue}, whilst the third type appear once by the transversality already provided by \ref{BoundPts}.  Explicitly, collapsing a flowline to zero corresponds to replacing a copy of $Q$ in \eqref{eqPearly} by $\Delta_L$, for which we have achieved transversality, and the end of the $\diff^2=0$ moduli space which exhibits this collapsing can be seen by instead replacing $Q$ by the manifold-with-boundary $Q_0$, defined by
\[
Q_0 = \{ (p, q) \in L \times L : p \notin \crit(f) \text{, } q = \Phi_t(p) \text{ for some } t \geq 0\} = (Q \cup \Delta_L) \setminus \Delta_{\crit(f)}.
\]

Before discussing how to achieve the necessary transversality, we first introduce some further terminology and notation.  Recall that a pseudoholomorphic disc $u$ is \emph{simple} if its set of injective points
\[
\{z \in D : u^{-1}(u(z))=\{z\} \text{ and } u'(z) \neq 0\}
\]
contains a dense open subset of $D$, and that a sequence $(u_1, \dots, u_r)$ of discs is \emph{absolutely distinct} if for all $j$ we have
\[
u_j(D) \not\subset \bigcup_{k \neq j} u_k(D).
\]
For an $r$-tuple $\mathbf{B}$ of homology classes we define $\mtw(\mathbf{B})$ to be the moduli space
\begin{equation}
\label{eqMtwBDef}
\big(\ev_+(B_1) \times \ev_-(B_2) \times \ev_+(B_2) \times \dots \times \ev_-(B_r)\big)^{-1} \lb\Delta_L^{r-1}\rb
\end{equation}
of bubbled chains of discs, illustrated in Fig.\nobreakspace \ref {figChainDisc}.
\begin{figure}[ht]
\centering
\includegraphics[scale=1]{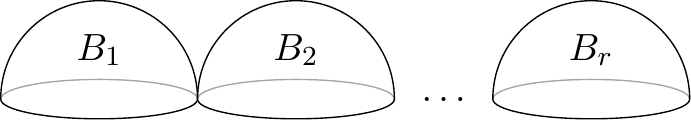}
\caption{A bubbled chain of discs.\label{figChainDisc}}
\end{figure}
These spaces carry evaluation maps $\ev_\pm(\mathbf{B})$ at the two end marked points.  We shall be interested in moduli spaces defined by \eqref{eqPearly} but with each $A_j$ now an $r_j$-tuple $\mathbf{B}^j=(B^j_1, \dots, B^j_{r_j})$.  We denote this modification by \eqref{eqPearly}', and refer to these configurations as \emph{generalised strings of pearls} or \emph{generalised pearly trajectories}.  A generalised string is reduced if each disc is non-constant, or there is only one disc.  Note that transversality for \eqref{eqPearly} with copies of $Q$ replaced by $\Delta_L$ can be expressed in terms of transversality for the $\mtw(\mathbf{B})$ (i.e.~for \eqref{eqMtwBDef}) and for \eqref{eqPearly}'.

We'll say a generalised string of pearls $\gamma$ is \emph{diagonal-avoiding} if the evaluation maps in \eqref{eqPearly}' at $\gamma$ miss the big diagonal $\Delta_{L, 2r}$.
\begin{figure}[ht]
\centering
\includegraphics[scale=1]{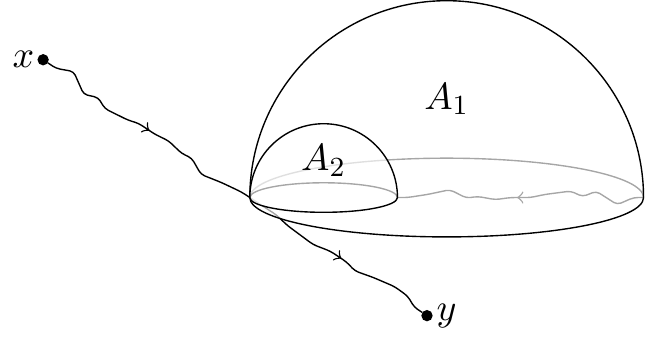}
\caption{A non-diagonal-avoiding pearly trajectory.\label{figNonDiag}}
\end{figure}
A trajectory $\gamma$ which is \emph{not} diagonal-avoiding is shown in Fig.\nobreakspace \ref {figNonDiag}: the entry point of the first disc is equal to the exit point of the second disc, or in other words $\ev_-(A_1)$ and $\ev_+(A_2)$ coincide at $\gamma$.  The notion of diagonal avoidance clearly also applies to the loose end moduli spaces \eqref{eqLE1} and \eqref{eqLE2} in an obvious way.  To be clear, in this case the diagonal avoidance condition applies only to the $2r-1$ evaluation maps appearing in each of \eqref{eqLE1} and \eqref{eqLE2}, not to the evaluation map at the loose end marked point.

We are now ready to describe the attainment of transversality in various settings, so suppose that $X$ is in fact K\"ahler with integrable complex structure $J$, and that all $J$-holomorphic discs in $X$ with boundary on $L$ have all partial indices non-negative (recall that Evans--Lekili showed that these hypotheses are satisfied when $(X, L)$ is $K$-homogeneous, in the proof of \cite[Lemma 3.2]{EL1}).  In particular \ref{JisReg} is automatically satisfied and for all classes $A \in H_2(X, L)$ the evaluation maps $\ev_\pm(A) \mc \mtw(A) \rightarrow L$ are submersions.  This in turn means that for all tuples of classes $\mathbf{B} = (B_1, \dots, B_r)$ the evaluation map
\[
(\ev_+(B_1), \ev_-(B_2), \dots, \ev_+(B_{r-1}), \ev_-(B_r))
\]
is transverse to $\Delta_L^{r-1}$ (i.e.~\eqref{eqMtwBDef} is transverse), so the moduli spaces $\mtw(\mathbf{B})$ of bubbled chains are transversely cut out and thus form smooth moduli spaces of the correct dimension.  Given any Morse--Smale pair $(f, g)$, we shall show that it can be pulled back by a diffeomorphism $\phi$ of $L$, which is $C^\infty$-close to $\id_L$, so that the moduli spaces of diagonal-avoiding generalised strings of pearls and of diagonal-avoiding strings with loose ends are transversely cut out.

It is then enough to show that all reduced generalised strings of pearls in virtual dimension at most $1$ are diagonal-avoiding.  This immediately gives \ref{StringTrans0}--\ref{BoundPts}, even for non-reduced strings, although the latter aren't actually needed.  For the loose end spaces in \ref{PearlBrokGlue} and \ref{PearlBubGlue} recall that we only need transversality near to the points which actually appear in the degenerate trajectories.  In particular, once we have shown that the only degenerate trajectories which occur are diagonal-avoiding, the only trajectories with loose ends which we need consider are those which are also diagonal-avoiding.  (Our assumption that all partial indices of holomorphic discs are non-negative actually implies that the loose end moduli spaces here are automatically transversely cut out, even if they are not diagonal-avoiding.  We do not make use of this fact in our argument though, as it does not extend to all of the Y-shaped loose end trajectories used later for the Floer product.)

Strictly there \emph{are} non-diagonal-avoiding reduced generalised strings of pearls in virtual dimension at most $1$, namely those trajectories with a single disc, which is constant, but we shall show that these are the only exceptions.  This issue does not affect the argument for \ref{PearlBrokGlue} and \ref{PearlBubGlue}, and the only potential issue it causes for \ref{StringTrans0}--\ref{BoundPts} is with transversality for standard Morse trajectories.  However, we assumed that the pair $(f, g)$ we started with was already Morse--Smale, and pulling back by a diffeomorphism does not affect this property, so this potential issue does not actually arise.

We have therefore reduced the problem of constructing a pearl complex using the integrable $J$ to the following two results:

\begin{lem}\label{labPullTrans}  For any Morse--Smale pair $(f, g)$ there exists a diffeomorphism $\phi$ of $L$, arbitrarily $C^\infty$-close to $\id_L$, such that the moduli spaces of diagonal-avoiding generalised strings of pearls and of diagonal-avoiding strings with loose ends for auxiliary data $(\phi^*f, \phi^*g, J)$ are transversely cut out.
\end{lem}

\begin{proof}
Apply \protect \MakeUppercase {L}emma\nobreakspace \ref {labPullTrans2}, taking the $M_j$ to be products of moduli spaces $\mtw(\mathbf{B})$ of bubbled chains of discs, the $f_j$ to be products of the corresponding evaluation maps, and the $N_j$ to be products of copies of $Q$ with ascending and descending manifolds of critical points of $f$.
\end{proof}

\begin{lem}\label{labDiagAvoid}  If $(f, g)$ is a Morse--Smale pair for which all moduli spaces of diagonal-avoiding generalised strings of pearls are transversely cut out, then all reduced generalised strings of pearls in virtual dimension at most $1$ are diagonal-avoiding unless they have a single disc, which is constant.
\end{lem}

\begin{proof}
Suppose that we are given a generalised string of pearls
\[
\gamma = \big(( u^1_1, \dots, u^1_{s_1}), \dots, (u^r_1, \dots, u^r_{s_r})\big) \in \mathcal{P}(x, y, \mathbf{A}=(\mathbf{B}^1, \dots, \mathbf{B}^r)) \subset \mtw(\mathbf{B}^1) \times \dots \times \mtw(\mathbf{B}^r)
\]
in virtual dimension $d \leq 1$ which has no constant discs.  Under $\ev_-(\mathbf{B}^1) \times \dots \times \ev_+(\mathbf{B}^r)$ this evaluates to some point $(p_1, \dots, p_{2r}) \in L^{2r}$; let $N(\gamma)$ be the number of pairs $(j, k)$ with $1 \leq j < k \leq 2r$ and $p_j = p_k$.

If the count $N(\gamma)$ is zero then by definition $\gamma$ is diagonal-avoiding, so we're done.  Otherwise, pick a pair $(j, k)$ contributing to this count.  Our aim is to show that we can delete some of the discs to form a new trajectory satisfying the same hypotheses as $\gamma$, but now in negative virtual dimension (since the deleted discs are non-constant, and hence of index at least $2$) and with a strictly smaller $N$-value.  Repeating this until $N$ reaches $0$ we obtain a diagonal-avoiding trajectory---which is therefore cut out transversely---in negative virtual dimension, which is impossible.  We thus conclude that $\gamma$ was diagonal-avoiding to begin with.

There are several cases to consider.  First suppose that $j$ and $k$ are both odd---say $j=2a-1$ and $k=2b-1$, with $a < b$.  In this case we delete the discs $u^l_m$ for $a \leq l < b$, and obtain a generalised string of pearls in virtual dimension
\[
d - \sum_{a \leq l < b} \mu(\mathbf{B}^l) \leq d - N_L < 0.
\]
Similarly if $j=2a$ and $k=2b$ (with $a < b$) or $j=2a-1$ and $k=2b$ (with $a \leq b$) then we delete $u^l_m$ for $a < l \leq b$ or $a \leq l \leq b$ respectively.  Finally, if $j=2a$ and $k=2b-1$ then we must have $b > a+1$ (otherwise $(p_j, p_k)$ lies in $Q$, which does not meet $\Delta_L$), and we delete $u^l_m$ for $a < l < b$.
\end{proof}

As an aside, we remark that it is possible that one could also achieve transversality by an argument similar to Haug's in \cite[Section 7.2]{Ha}, which is itself adapted from the proof of genericity of Morse--Smale metrics for a given Morse function using the Sard--Smale theorem.  Rather than restricting to diagonal-avoiding trajectories---which are trajectories whose flowlines don't intersect where they meet discs---one would instead work with trajectories whose flowlines don't intersect at all, and would then have to check that Haug's analysis can all be made to work in this setting.  The approach to transversality given above seems preferable however, as it is more concrete and gives better control on the resulting ascending and descending manifolds.

Now return to the case of general $(X, L)$, where there is no longer a preferred choice of almost complex structure.  The approach of Biran--Cornea is to \emph{fix} $(f, g)$ and instead choose $J$ to achieve transversality.  Standard arguments with universal moduli spaces show that \ref{JisReg}--\ref{PearlBubGlue} are satisfied for a generic (second category) choice of $J$, as long as we restrict to trajectories in which the discs are simple and absolutely distinct.  This is to ensure that for any trajectory we can perturb $J$ independently on a neighbourhood of (the image of) an injective point of each disc, and is analogous to the restriction to diagonal-avoiding trajectories in our argument.

One then has to show that, generically, the only trajectories which occur in virtual dimension at most $1$ automatically satisfy the simple and absolutely distinct conditions.  This is to ensure that imposing these conditions does not destroy compactness: a priori a limit of simple discs need not be simple, for example.  Here the argument, described in \cite[Section 3.2]{BCQS}, splits into two cases ($\dim L \geq 3$ and $\dim L \leq 2$), but both amount to showing that for a second category set of almost complex structures certain additional evaluation map transversality conditions are satisfied (namely \cite[Equation (9)]{BCQS} and the two bullet points in the proof of \cite[Proposition 3.4.1]{BCQS} in the two cases respectively).

\subsection{Invariance of the cohomology}\label{sscInvCoh}

Having constructed the pearl complex, we would like to show that the resulting cohomology is independent of the choice of auxiliary data, up to canonical isomorphism.  There are at least three different variants of this independence which one may be interested in.  Firstly, if we have a choice of integrable $J$ for which all holomorphic discs have all partial indices non-negative then one may want to prove that the cohomology for this fixed $J$ is independent of the choice of perturbed $(f, g)$ constructed above.  Secondly, in the general setting considered by Biran--Cornea one may want to prove that the whole triple $(f, g, J)$ can be varied.  And thirdly, one may want to show that the complex we constructed for special $J$ gives the same cohomology as that constructed by Biran--Cornea with generic $J$.  Of course, the latter really supersedes the first variant, but there are situations where one may want to work at all times with the special $J$ and then the first method genuinely is needed.

The first two of these arguments use the method of Morse cobordisms, introduced in \cite{CoRa} and used by Biran--Cornea to prove the second variant of independence in \cite[Section 5.1.2]{BCQS}.  The idea is as follows.  Given two choices $(f_0, g_0, J_0)$ and $(f_1, g_1, J_1)$ of auxiliary data, we choose a Morse cobordism $(F, G)$ comprising a Morse--Smale pair on $L \times [0, 1]$ which coincides with $(f_0, g_0)$ on $L \times \{0\}$ and $(f_1+C, g_1)$ on $L \times \{1\}$, where $C \gg 0$ is a positive constant, and such that the flow of $\nabla F$ is tangent to the boundary $L \times \{0, 1\}$ and points in the direction of strictly increasing $t$ (which denotes the coordinate on the $[0, 1]$ factor) in the interior.  We also choose an appropriate homotopy $J_t$ of almost complex structures from $J_0$ to $J_1$.

We then consider moduli spaces of discs in $X \times [0, 1]$ with boundary on $L \times [0, 1]$, which are constant---$T$, say---on the $[0, 1]$ factor, and $J_T$-holomorphic in the $X$ factor, and carry evaluation maps at two boundary marked points.  Using these moduli spaces and the data $(F, G)$ we can build moduli spaces of strings of pearls on $L \times [0, 1]$.  Counting rigid strings from $x \in \crit(f_0) \times \{0\} \subset \crit(F)$ to $y \in \crit(f_1) \times \{1\} \subset \crit(F)$ we obtain a map between the corresponding pearl complexes.  The fact that this is a chain map follows from considering one-dimensional moduli spaces of such trajectories, and the boundaries of their compactifications.  To prove that this chain map is a quasi-isomorphism, and is independent of the choice of $(F, G, J_t)$ up to chain homotopy, we use Morse cobordisms on $L \times [0, 1]^2$ defined in a similar way.

The key property of all of these constructions, where we work over a parameter space $P$ which is a manifold with corners (equal to $[0, 1]$ or $[0, 1]^2$), is that trajectories (excluding their end points) live entirely in one stratum of $P$.  For example, if $\gamma$ is a pearly trajectory on $L \times [0, 1]^2$ which contains a disc lying over a point $p \in \{0\} \times (0, 1) \subset [0, 1]^2$ then the whole trajectory, apart from its end points, lies over $\{0\} \times (0, 1)$.  The reason for this is that the flow of the Morse cobordism is tangent to each boundary stratum.

Now, to prove the first variant of independence, where $J_0$ and $J_1$ are both equal to our special integrable $J$ for which all partial indices are non-negative, we choose an arbitrary Morse cobordism $(F, G)$ and take $J_t=J$ for all $t$.  By construction of $(f_0, g_0)$ and $(f_1, g_1)$ we already have transversality for strings of pearls and loose end spaces which lie over $0$ or $1$, so we only need worry about transversality for the moduli spaces comprising discs over the interior of $[0, 1]$.  And for these we can use the same arguments as for \protect \MakeUppercase {L}emma\nobreakspace \ref {labPullTrans} and \protect \MakeUppercase {L}emma\nobreakspace \ref {labDiagAvoid} to perturb $(F, G)$ on $L \times (0, 1)$ to achieve transversality.  The same approach works for the cobordisms over $[0, 1]^2$, where we start out with transversality for trajectories contained in the boundary strata and perturb the cobordism over the interior.

The second variant is proved by taking an arbitrary cobordism $(F, G)$ and choosing a generic path $J_t$ of almost complex structures from $J_0$ to $J_1$ in order to achieve transversality for trajectories of discs which are simple and absolutely distinct (with the images of discs now viewed as subsets of $X \times [0, 1]$ rather than just $X$), and the additional conditions needed to ensure all trajectories in virtual dimension at most $1$ are of this form.  Again we can do this since we only need to consider trajectories living over the interior of $[0, 1]$, where we have the freedom to perturb $J_t$.  The $[0, 1]^2$ cobordisms are dealt with similarly.

For the third variant we can actually take a slightly simpler approach.  We need not vary the Morse data, and instead can just consider one-parameter families of moduli spaces of strings of pearls in which the almost complex structure varies along a generic path starting at our special $J$ and ending at some generic $J_1$.  The one-dimensional such moduli spaces can be compactified and all boundary points occurring in the interior of the path $J_t$ cancel out.  We are left with boundary points occurring at the end of the path---which are pearly trajectories for $J_1$---and those at the beginning----which are trajectories for $J$ but counted with a minus sign.  The $J$- and $J_1$-complexes are thus isomorphic.

Combining the results of the appendix so far with Biran--Cornea's proof that pearl complex (co)homology is self-Floer (co)homology \cite[Section 5.6]{BCQS}, we have proved the following:

\begin{prop}\label{labPearlCx}  Suppose $X$ is a compact K\"ahler manifold with complex structure $J$, and $L \subset X$ is a closed, connected, monotone Lagrangian with minimal Maslov number $N_L \neq 1$, equipped with a Morse--Smale pair $(f, g)$ and---if the coefficient ring $R$ has characteristic not equal to $2$---a choice of orientation and spin structure.  If every $J$-holomorphic disc in $X$ with boundary on $L$ has all partial indices non-negative then there exists a diffeomorphism $\phi$ of $L$, arbitrarily $C^\infty$-close to the identity, such that the pearl complex can be defined using the auxiliary data $(\phi^*f, \phi^*g, J)$, and computes the self-Floer cohomology of $L$.
\end{prop}

\subsection{The Floer product using special $J$}

The pearl complex also carries extra algebraic structures which one may want to compute using a special integrable $J$ (for which all partial indices are non-negative, as above), and we now consider the case of the Floer product.  This is defined by taking three Morse--Smale pairs $(f_1, g_1)$, $(f_2, g_2)$ and $(f_3, g_3)$, which are generic in a sense to be made precise later, and defining a map
\[
 * \mc C_1^p \otimes C_2^q \rightarrow C_3^{p+q}
\]
(where $C_j$ is the pearl complex constructed using $(f_j, g_j, J)$) which satisfies the Leibniz rule and hence induces a product on cohomology.  This product then has to be shown to be associative and independent of the various choices made.

The arguments involved are fundamentally the same as those used in the preceding subsections, so we focus on the features which require modification.  The moduli spaces are more numerous than before and it would be rather cumbersome and unenlightening to express them all individually as fibre products analogous to \eqref{eqPearly}, \eqref{eqLE1} and \eqref{eqLE2}, so we instead describe them in words and illustrate them with diagrams of examples, which are hopefully easier to digest.  It is easy to translate back and forth between these diagrams and fibre product expressions as needed.

The product itself is defined by counting Y-shaped configurations, as shown in Fig.\nobreakspace \ref {figYTraj}.
\begin{figure}[ht]
\centering
\includegraphics[scale=1]{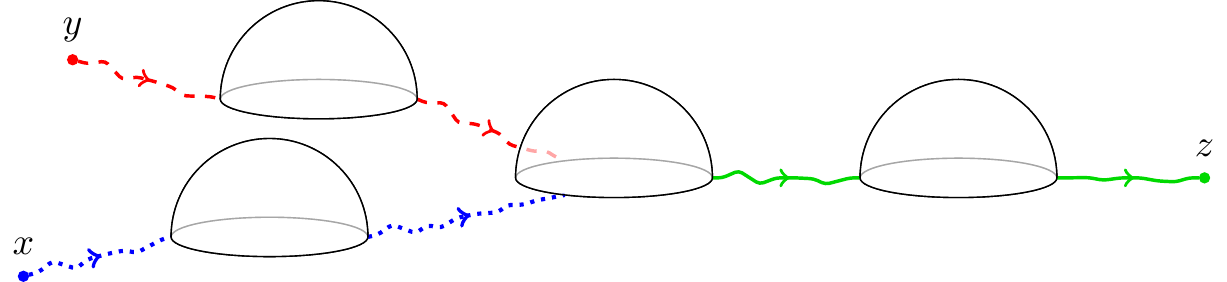}
\caption{A Y-shaped pearly trajectory.\label{figYTraj}}
\end{figure}
In the diagram the dotted lines denote flowlines of $\nabla f_1$, whilst dashed is used for $f_2$ and solid for $f_3$.  Blobs on flowlines denote critical points of the corresponding Morse function.  The number of discs shown is purely illustrative: each branch of the Y may have any number of discs, which we may assume to be non-constant, including zero.  The central disc \emph{is} allowed to be constant, and terms defining the standard cup product on the Morse cohomology of $L$ come from trajectories in which this is the case and there are no other discs.  These moduli spaces have an obvious description analogous to \eqref{eqPearly}, in which the central disc carries three marked points, which must be in the order indicated in the diagram (going round clockwise we must have the incoming dotted flowline, then the incoming dashed flowline, and finally the outgoing solid flowline).  This restriction on the order of the marked points leads to the failure of the product to be graded-commutative in general.

Just as for the ordinary strings of pearls, it is helpful to consider more general moduli spaces in which we allow discs to be replaced by bubbled chains of discs, or by bubbled Y-shaped configurations of discs at the centre.  Some of these bubbled Y-shaped configurations are illustrated in Fig.\nobreakspace \ref {figBubY}: the top left diagram shows a non-constant disc with a single bubble at each marked point; the top right shows a constant central disc, bubbled at each marked point; the bottom diagram shows a non-constant central disc which carries a single bubble at one marked point and a chain of two bubbles at another.  Note that since the discs are assumed to have all partial indices non-negative these bubbled Y-shaped configurations are also cut out transversely so form smooth moduli spaces of the correct dimension.  We shall call trajectories in which discs may be bubbled \emph{generalised Y-shaped strings of pearls} or \emph{generalised Y-shaped pearly trajectories}.
\begin{figure}[ht]
\centering
\includegraphics[scale=1]{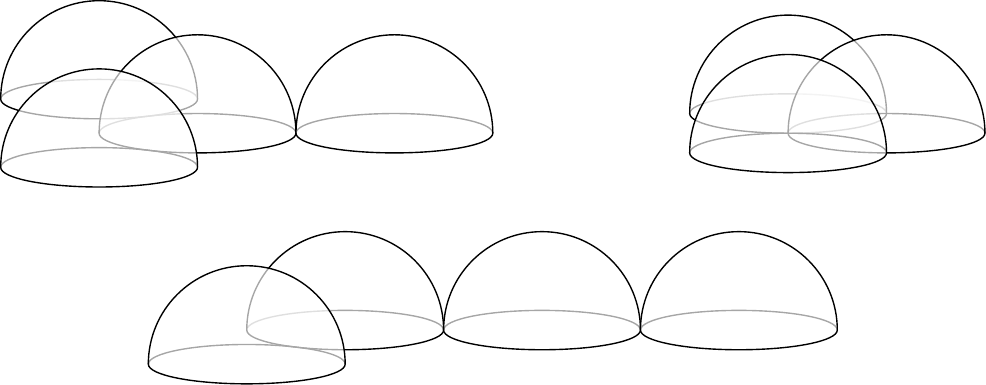}
\caption{Bubbled Y-shaped configurations of discs.\label{figBubY}}
\end{figure}

We shall also need the corresponding loose end moduli spaces, where there may now be one or two loose ends, as illustrated in Fig.\nobreakspace \ref {figYLoose}.
\begin{figure}[ht]
\centering
\includegraphics[scale=1]{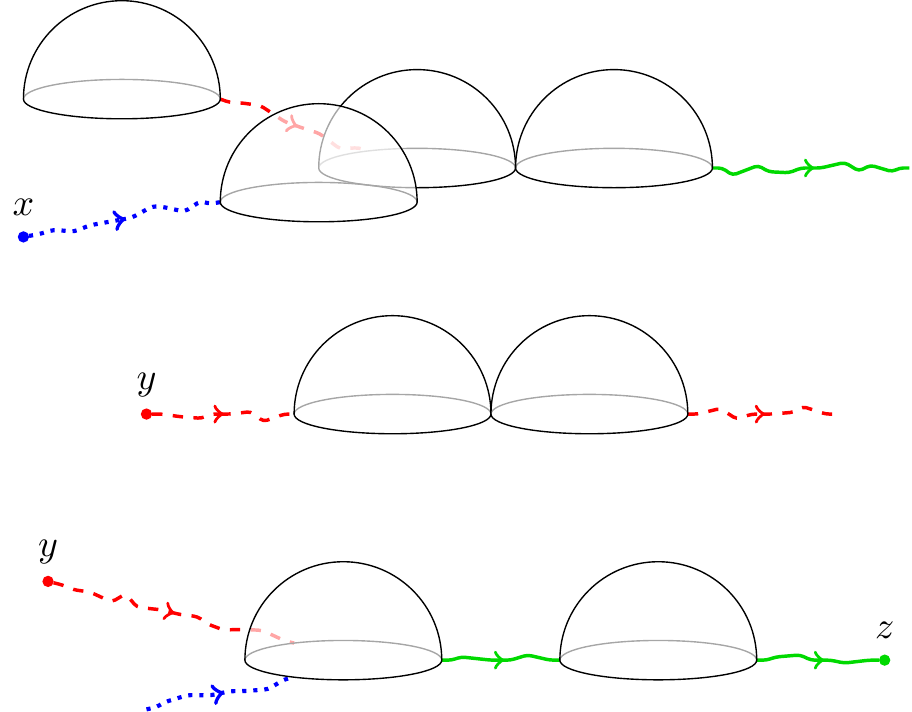}
\caption{Y-shaped strings of pearls with loose ends.\label{figYLoose}}
\end{figure}
As before, we insert a constant disc at any loose end point which is bare (i.e.~without a disc) to keep track of its position.  We say a Y-shaped pearly trajectory, possibly with loose ends, is \emph{reduced} if the only constant discs are at the centre of a Y or at loose ends.

The notion of diagonal-avoidance has to be slightly modified for these spaces of trajectories involving multiple sets of Morse data.  Each evaluation map into $L$ from a moduli space of discs, or more generally of bubbled configurations of discs, comes labelled with a $1$, $2$ or $3$ depending on which function's gradient flow joins up with that evaluation map.  For example, in the bottom trajectory in Fig.\nobreakspace \ref {figYLoose} the three evaluation maps on the central disc are labelled $1$, $2$ and $3$ clockwise from bottom left (and this will always be the case), the other non-constant disc has both marked points labelled $3$ (they join solid flowlines, indicating $\nabla f_3$), whilst the constant disc at the loose end has its evaluation map labelled $1$ (as it joins a dotted flowline, meaning $\nabla f_1$).  The modified diagonal-avoidance condition is then that the evaluation maps carrying the same label avoid the big diagonal in their corresponding $L$ factors.

With these definitions in place, the transversality we require is that: moduli spaces of diagonal-avoiding generalised strings of pearls and the associated diagonal-avoiding loose end spaces are transversely cut out for each $(f_j, g_j)$; moduli spaces of diagonal-avoiding generalised Y-shaped strings of pearls and the loose end versions are transversely cut out.  We'll call these conditions `product transversality'.  Using an obvious modification of \protect \MakeUppercase {L}emma\nobreakspace \ref {labPullTrans2} these can be achieved by pulling back the $(f_i, g_i)$ by diffeomorphisms $\phi_j$ of $L$ which are $C^\infty$-close to the identity.  The first condition ensures that each $(f_j, g_j)$ defines a valid pearl complex, and the second condition lets us define the product by counting rigid reduced Y-shaped strings of pearls.  By considering compactifications of moduli spaces of reduced Y-shaped pearly trajectories of virtual dimension $1$ we obtain the Leibniz property
\[
\diff(x * y) = (\diff x) * y + (-1)^{|x|} x * (\diff y),
\]
which means that the product descends to cohomology.  The reduced moduli spaces in virtual dimension at most $1$ are all automatically diagonal-avoiding, by applying the argument of \protect \MakeUppercase {L}emma\nobreakspace \ref {labDiagAvoid} to each leg of the Y.  Note that the exceptional non-diagonal-avoiding case that occurs for the basic (i.e.~non-Y-shaped) trajectories, namely that of standard Morse trajectories, does not occur in the Y-shaped case, since the Morse product trajectories actually \emph{are} diagonal-avoiding in our modified sense (the three flowlines which meet correspond to distinct Morse--Smale pairs which can be perturbed independently).

The only new phenomenon that occurs is bubbling of the thrice-marked central disc at the boundary of one-dimensional moduli spaces, which is taken care of by gluing results analogous to those for twice-marked discs, as in \cite[Section 5.2]{BCQS}.  Note that convergence of two of the marked points (which can be viewed as bubbling off of a constant `ghost' disc), is cancelled out by the shrinking of a Morse flowline from a constant central disc to a non-constant disc, as shown from above in Fig.\nobreakspace \ref {figGhostBub}.
\begin{figure}[ht]
\centering
\includegraphics[scale=1]{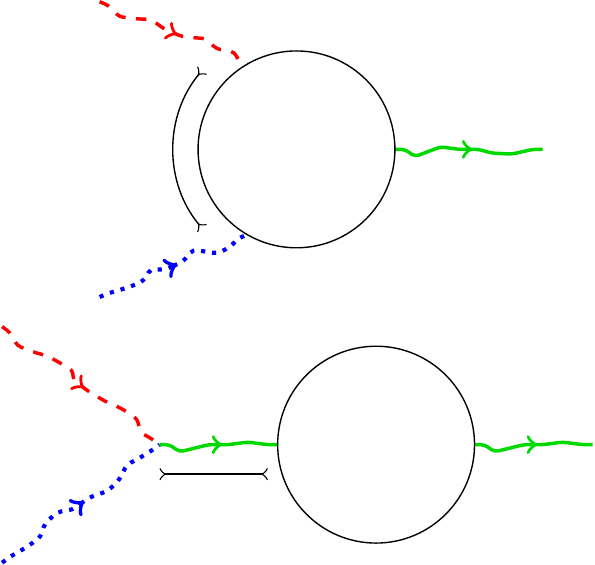}
\caption{Convergence of marked points cancels shrinking of a flowline.\label{figGhostBub}}
\end{figure}
In particular, if we allow the marked points to appear in both orders around the boundary circle then the degenerate configuration occurs at \emph{three} ends of the corresponding one-dimensional moduli space (twice from the convergence of marked points---once in either order---and once from the shrinking of a flowline) and hence does not cancel out.  This is why the order of the marked points has to be fixed.

Suppose we replace $(f_3, g_3)$, say, with another Morse--Smale pair $(f_3', g_3')$.  From now on we'll drop explicit mention of the metrics.  The key idea for proving invariance of the product is:
\begin{lem}
We can perturb $f_3'$ by a diffeomorphism $C^\infty$-close to the identity in order to achieve product transversality for $f_1$, $f_2$ and $f_3'$.
\end{lem}
\begin{proof}
Fix a moduli space $\mathcal{M}$ of trajectories for which we need to achieve transversality.  For example, $\mathcal{M}$ could be the space of diagonal-avoiding trajectories of the shape shown in the third diagram in Fig.\nobreakspace \ref {figYLoose} (with $f_3'$ in place of $f_3$), with specified homology classes for the discs.  Deleting all flowlines of $\nabla f_3'$, the trajectory breaks into pieces which are either moduli spaces of discs (possibly bubbled chains or bubbled Y-shaped configurations), or loose end trajectories involving only $f_1$ and $f_2$.  Since we have already attained product transversality for $f_1$, $f_2$ and $f_3$, these loose end trajectories for $f_1$ and $f_2$ are transversely cut out.

We can therefore describe $\mathcal{M}$ as a fibre product analogous to \eqref{eqPearly}.  Instead of taking a product of moduli spaces of discs over the flow spaces (meaning the ascending/descending manifolds, or the space $Q$) for $\nabla f$, we take a product of moduli spaces of discs (possibly bubbled) or of transversely cut out loose end trajectories for $f_1$ and $f_2$, over the flow spaces for $\nabla f_3'$.  This is illustrated for our example in Fig.\nobreakspace \ref {figPertThird}, where the downward arrows represent the evaluation maps at the marked points.
\begin{figure}[ht]
\centering
\includegraphics[scale=1]{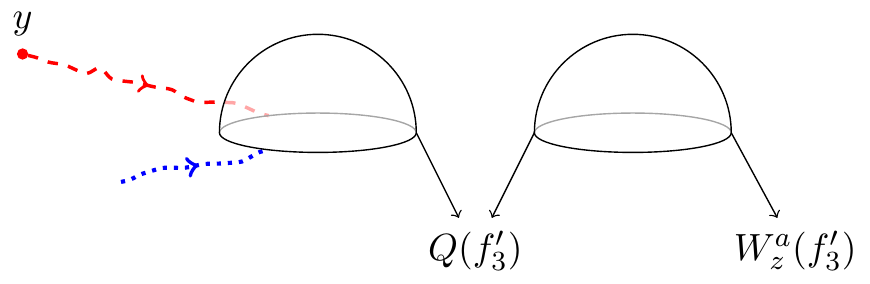}
\caption{Expressing $\mathcal{M}$ as a fibre product.\label{figPertThird}}
\end{figure}
Once all of the moduli spaces $\mathcal{M}$ are described in this way, we can use \protect \MakeUppercase {L}emma\nobreakspace \ref {labPullTrans} as before to see that transversality can be achieved by perturbing $f_3'$.
\end{proof}

Similarly, if we are given transverse triples of data $f_1$, $f_2$, $f_3$ and $f_1$, $f_2$, $f_3'$ then we can perturb any Morse cobordism from $f_3$ to $f_3'$ as in Appendix\nobreakspace \ref {sscInvCoh} to ensure the transversality required to get a comparison map from the $f_3$ complex to the $f_3'$ complex and for this comparison map to respect the product on cohomology.  Clearly the same is true if we change either of the other Morse--Smale pairs $(f_1, g_1)$ or $(f_2, g_2)$ instead.

Of course, to prove that the product is independent of the choices of Morse data in general we need a way to compare the products induced by two arbitrary triples $f_1$, $f_2$, $f_3$ and $f_1'$, $f_2'$, $f_3'$.  To do this we introduce an auxiliary triple $f_1''$, $f_2''$, $f_3''$, and Morse cobordisms from each $f_j$ to $f_j''$ and from $f_j''$ to $f_j'$, and perturb them all so that we get transverse triples and comparison maps as follows
\[
(f_1, f_2, f_3) \leadsto (f_1, f_2, f_3'') \leadsto (f_1, f_2'', f_3'') \leadsto (f_1'', f_2'', f_3'') \leadsto (f_1'', f_2'', f_3') \leadsto (f_1'', f_2', f_3') \leadsto (f_1', f_2', f_3').
\]

In order to show the product we have just defined using the special integrable $J$ coincides with the product defined using a generic almost complex structure (which obviously gives an indirect proof of the invariance of the former) we proceed as in the third variant of Appendix\nobreakspace \ref {sscInvCoh}: we introduce one-parameter moduli spaces of Y-shaped strings of pearls in which the almost complex structure is allowed to vary along a generic path, and consider the boundaries of the moduli spaces of virtual dimension $1$.  In Biran--Cornea's work, the same metric is used on each leg of the Y-shaped trajectories, but this is not necessary.

The upshot of this discussion is:

\begin{prop}\label{labPearlCxProd}  In the setting of \protect \MakeUppercase {P}roposition\nobreakspace \ref {labPearlCx}, but now given three Morse--Smale pairs $(f_j, g_j)_{j=1}^3$ on $L$, there exist diffeomorphisms $(\phi_j)_{j=1}^3$ of $L$, arbitrarily $C^\infty$-close to the identity, such that the Floer product can be computed using the pearl model with auxiliary data $(\phi_j^*f_j, \phi_j^*g_j, J)$.
\end{prop}

\section{Index $4$ count for the icosahedron}
\label{secInd4Icos}

The purpose of this appendix is to perform the necessary analysis of index $4$ discs with boundary on $L_I$ passing through the points $p$ and $q$ considered in Section\nobreakspace \ref {sscInd4I}.  In particular, we show that $(q, p)$ is a regular value of the two-point index $4$ evaluation map $\ev_2 \mc M_4 \rightarrow L_I^2$, the signed count of preimages is $\pm 8$, and there are no bubbled configurations.

There are two axial discs of type $\xi_f$ through $p$ and $q$.  By \protect \MakeUppercase {C}orollary\nobreakspace \ref {labInd4Trans} they are regular points of $\ev_2$, and by the comments at the end of Section\nobreakspace \ref {sscModInv} they count with the same sign.  There are clearly no axial discs of type $\xi_v$ and order $2$ through $p$ and $q$ since they have no vertices in common, so we are left to deal with the two-pole discs.  Unfortunately this is rather involved.

\begin{alem}\label{ind4NonAxCount} There are exactly six two-pole index $4$ discs passing through $p$ and $q$.  They are regular points of $\ev_2$ and all count towards $\deg \ev_2$ with the same sign.
\end{alem}
\begin{proof}
Suppose $u \mc (D, \pd D) \rightarrow (X_I, L_I)$ is such a disc, with double $\double{u}$.  All of the poles of $\double{u}$ are of type $\xi_v$, and $\deg \double{u}=4$.  Let one of the poles of $u$ evaluate to $P \coloneqq [(kx+y)^{11}(lx+y)]$, with $k, l \in \C\P^1$ distinct.

The argument is essentially completely elementary: we use the fact that $\double{u}$ has degree $4$ to force linear dependencies between various points and directions on it, and exploit these to deduce the possible positions of the poles.  The difficulty comes from the significant amounts of algebraic manipulation involved, for which we use Mathematica.  Our main strategy is to reduce the constraints to polynomials in $k$, which are easily handled by a computer (one can rule out solutions by showing that various polynomial constraints have no common factor for example).  Unfortunately, the argument also requires us to consider the reflections of points under the antiholomorphic involution, which introduces complex conjugate terms into our polynomials.  The key observation is that we can get away with only considering values of $k$ which are real or of unit modulus, for which $\conj{k}$ can be re-expressed as $k$ or $1/k$.

Let us now begin the proof proper.  The image of $\double{u}$ lies in a $4$-plane $\Pi \subset \P S^{12}V$, and we have three points in $\Pi$: $p$, $q$ and $P$.  We also know, thinking of $T_pX_I$ as a subspace of $\P S^{12}V$, that $\dim (\Pi \cap T_pX_I) \geq 1$ since this space contains the tangent line to $\double{u}$ at $p$ (if the derivative of $\double{u}$ vanishes where $\double{u}$ passes through $p$ then we take the lowest order derivative which doesn't vanish; this is a well-defined direction, which we note can also be expressed as $[\D u]|_p \cdot p$, where $[\D u]|_p$ is shorthand for the value of $[\D u]$ at $u^{-1}(p)$).  This in turn means that $\dim(\Pi' \cap T_pX_I) \leq 1$, where $\Pi'$ is an arbitrary complement to $\Pi$.  Similarly $\dim (\Pi' \cap T_qX_I)$ and $\dim (\Pi' \cap T_PX_I)$ are at most $1$.  We deduce that
\begin{equation}
\label{labLinDep}
\dim \lspan{T_pX_I, T_qX_I, T_PX_I} \leq \dim \lspan{\Pi, \Pi' \cap T_pX_I, \Pi' \cap T_qX_I, \Pi' \cap T_PX_I} \leq 10,
\end{equation}
whereas generically three $3$-planes in $\P S^{12}V$ would have $11$-dimensional span.

We can explicitly compute $T_pX_I$ and $T_qX_I$ using the infinitesimal action of $\mathfrak{sl}(2, \C)$, and write $T_PX_I$ as
\begin{align*}
\begin{pmatrix} k & l \\ 1 & 1 \end{pmatrix} \cdot T_{[x^{11}y]}X_I &= \begin{pmatrix} k & l \\ 1 & 1 \end{pmatrix} \cdot \lspan{ x^{12}, x^{11}y, x^{10}y^2, x^6y^6}
\\ &= \lspan{(kx+y)^{12}, (kx+y)^{11}(lx+y), (kx+y)^{10}(lx+y)^2, (kx+y)^6(lx+y)^6}.
\end{align*}
The condition \eqref{labLinDep} can then be expressed by finding a basis for the $4$-plane $\lspan{T_pX_I, T_qX_I}^\perp \subset \P S^{12}V^*$, applying these five functionals to the above collection of four polynomials spanning $T_PX_I$, and asking that the resulting $5 \times 4$ matrix has rank $3$.  The particular basis we use---for no particular reason other than that the expressions involved are fairly short---is
\[
[33\eta^{12}+2\eta^6\theta^6+33\theta^{12}]\text{, } [2\eta^7\theta^5+33\eta\theta^{11}]\text{, } [33\eta^{11}\theta+2\eta^5\theta^7] \text{, } [\eta^8\theta^4+3\eta^2\theta^{10}] \text{ and } [3\eta^{10}\theta^2+\eta^4\theta^8],
\]
where $\eta$, $\theta$ is the basis of $V^*$ dual to $x$, $y$.

It is straightforward to check that if $k$ and $l$ are both finite and one of them is zero then so is the other, contradicting our assumption that they are distinct.  One can also check that if $k=\infty$ then $l=0$, so $\double{u}$ passes through $[x^{11}y]$ and hence also $[xy^{11}]$ by reflection.  Similarly if $l=\infty$ then $k=0$ so again $\double{u}$ passes through both $[x^{11}y]$ and $[xy^{11}]$.  In this case we see that
\begin{equation}
\label{eqReflSpan}
\dim \lspan{T_pX_I, T_qX_I, [x^{11}y], [xy^{11}]} \leq \dim \lspan{\Pi, \Pi' \cap T_pX_I, \Pi' \cap T_qX_I} \leq 8,
\end{equation}
but this dimension can be explicitly calculated to be $9$.  We conclude that both $k$ and $l$ are finite and non-zero.

The $4 \times 4$ minors of the matrix we know to have rank $3$ give five polynomials in $k$ and $l$ which must all vanish.  Each is divisible by $9801(k-l)^6$, which we can cancel off since $k$ and $l$ are distinct, to leave polynomials of degree $3$ in $l$, which can then be written as $f_{i0}+f_{i1}l+f_{i2}l^2+f_{i3}l^3$ for polynomials $f_{ij}$ in $k$ (for $i=1, \dots, 5$ and $j=0, \dots, 3$).  We thus have the constraint
\[
\begin{pmatrix} f_{10} & \dots & f_{13} \\ \vdots & \ddots & \vdots \\ f_{50} & \dots & f_{53} \end{pmatrix} \begin{pmatrix} 1 \\ l \\ l^2 \\ l^3 \end{pmatrix} = 0
\]
on $k$ and $l$.

The matrix $F=(f_{ij})$ has rank $2$: the $3 \times 3$ minors are all identically zero, whilst there is no common root to all of the $2 \times 2$ minors.  Therefore the (right) kernel of $F$ is $2$-dimensional.  Letting $g_1 = 194 k^{30}+9031 k^{24}-59344 k^{18}-59692 k^{12}-2426 k^6+20$, and assuming for now that this is non-zero, an explicit basis for this kernel is given by the vectors
\def\scale{0.9}
\begin{equation}
\label{eqExplBas1}
\scalebox{\scale}{
$\begin{pmatrix}
k^3g_1 \\
0 \\
-3 k^5 \lb149 k^{24}-7423 k^{18}-22434 k^{12}-7423 k^6+149\rb \\
217 k^{30}-18361 k^{24}-144086 k^{18}-65987 k^{12}+4868 k^6+22
\end{pmatrix}$
}
\end{equation}
and
\begin{equation}
\label{eqExplBas2}
\scalebox{\scale}{
$\begin{pmatrix}
0 \\
k^4g_1 \\
-k^5\lb22 k^{30}+4868 k^{24}-65987 k^{18}-144086 k^{12}-18361 k^6+217\rb \\
2 k^{36}+2208 k^{30}-51207 k^{24}-238046 k^{18}-51207 k^{12}+2208 k^6+2
\end{pmatrix}$
}.
\end{equation}
Dividing by $k^3g_1$ and $k^4g_1$ respectively, we write this basis as $(1, 0, \alpha, \beta)$ and $(0, 1, \gamma, \delta)$.

We know that $(1, l, l^2, l^3)$ is a linear combination of these two vectors, from which it is easy to see that
\begin{equation}
\label{labABCD}
l^2 = \alpha + l \gamma \text{ and } l^3 = \beta + l \delta.
\end{equation}
We thus have
\[
\beta + l \delta = l(\alpha + l \gamma) = l \alpha + (\alpha+l \gamma)\gamma,
\]
and hence
\[
l(\alpha+\gamma^2-\delta) = \beta - \alpha \gamma.
\]
The coefficient of $l$ in the latter cannot vanish, otherwise the right-hand side must also vanish, and these two rational functions in $k$ have no common root.  We thus have
\begin{equation}
\label{lval}
l=\frac{\beta - \alpha \gamma}{\alpha+\gamma^2-\delta},
\end{equation}
and we can substitute this back into \eqref{labABCD}, clear denominators, and take the greatest common divisor of the resulting polynomials in $k$ to get
\begin{multline}
\label{defg2}
320 k^{96}-224128 k^{90}-1467885 k^{84}+5601117772 k^{78}+42700276243 k^{72}
\\-623885336112 k^{66}+2513717360270 k^{60}-5265619809592 k^{54}+6655153864734 k^{48}
\\ -5265619809592 k^{42}+2513717360270 k^{36}-623885336112 k^{30}+42700276243 k^{24}
\\ +5601117772 k^{18}-1467885 k^{12} -224128 k^6+320 = 0.
\end{multline}

Let $g_2$ denote the left-hand side of this equation.  Note that both $g_1$ and $g_2$ are polynomials in $k^6$.  This reflects the fact that the problem of finding discs through $p$ and $q$ is invariant under rotations through $\pi/3$ about a vertical axis, corresponding to multiplication by a sixth root of unity---although each of $p$ and $q$ is only invariant under rotations through $2\pi/3$, rotating through half this angle swaps $p$ and $q$ over.  In order to prove the existence or non-existence of discs with a given value of $k$, we only need to deal with a single representative of each orbit of this symmetry.  Let $h_1$ and $h_2$ be the polynomials defined by $h_j(k^6)=g_j(k)$, of degrees $5$ and $16$ respectively.  (One may also notice that $g_2$---or equivalently $h_2$---is palindromic, in the sense that reading its coefficients from highest power of $k$ to lowest gives the same list as reading from lowest to highest.  This reflects the fact that our problem is invariant under the automorphism of $\C\P^1$ which sends $z$ to $1/z$, which swaps the configurations $p$ and $q$.  The polynomial $g_1$ does not have this property as we broke the $z \mapsto 1/z$ symmetry by choosing a basis for $F$ in the form \eqref{eqExplBas1} and \eqref{eqExplBas2}.)

Now suppose for contradiction that $g_1$ vanishes; we asserted earlier that this is not the case.  It is easy to check---by plotting a graph and counting changes of sign, for example---that all roots of $h_1$ are real, and hence by the above comment on the rotational symmetry we may assume that our root $k$ of $g_1$ is real.  We claim that $l$ is also real.  If not, the vector $(1, l, l^2, l^3)$ and its conjugate are linearly independent and thus span $\ker F$.  In particular we deduce that no non-zero element of $\ker f$ vanishes in both of the first two components.  However, the basis vectors in \eqref{eqExplBas1} and \eqref{eqExplBas2} both provide counterexamples (a computer calculation shows that $g_1$ has no common roots with the third and fourth components of each vector, so if $g_1$ vanishes then neither vector is zero), proving that $l$ is indeed real.

The reflection of the pole at $P=[(kx+y)^{11}(lx+y)]$ is therefore at $\tau(P)=[(x-ky)^{11}(x-ly)]$, and we obtain
\[
\dim \lspan{T_pX_I, T_qX_I, P, \tau(P)} \leq 8
\]
analogously to \eqref{eqReflSpan}.  To exploit this we proceed as for \eqref{labLinDep}, by applying our basis of $\lspan{T_pX_I, T_qX_I}^\perp$ to $P$ and $\tau(P)$, setting the $2 \times 2$ minors equal to zero, writing this as the vanishing of a $10 \times 3$ matrix times $(1, l, l^2)$ and then computing the greatest common divisor $g_3$ of the $3 \times 3$ minors of this last matrix.  Our root $k$ of $g_1$ must also be a root of $g_3$, but one can check that the two polynomials have no common factor, giving the desired contradiction and completing our argument that $g_1$ is non-zero.

Now we can return to the main thread of the argument, recalling that the polynomial $g_2$, given by the left-hand side of \eqref{defg2}, must vanish.  Its reduction $h_2$, obtained by replacing the variable $k^6$ by $z$ say, factorises as
\begin{multline*}
h_2=(8 - 11 z + 8 z^2)(1 - 671 z + 2301 z^2 - 671 z^3 + z^4)(1 -    185 z + 357 z^2 - 185 z^3 + z^4) 
\\ (40 + 6279 z + 81132 z^2 -    178264 z^3 + 81132 z^4 + 6279 z^5 + 40 z^6).
\end{multline*}
We claim all roots are real or of unit modulus.  To see this note that the second and fourth factors have only real roots by counting sign changes.  The first factor, meanwhile, has two complex conjugate roots and is invariant under $z \leftrightarrow 1/z$, so the roots have unit modulus.  Finally, the third factor has at least two real roots (it is negative at $z=0$ and positive for large real $z$) and is again invariant under $z \leftrightarrow 1/z$, so the other two roots are either real or conjugate complex numbers of unit modulus.

Using the $\pi/3$ rotational symmetry, we may therefore assume that our solution value of $k$ is either real or of unit modulus.  In the former case it is clear, this time from \eqref{lval}, that $l$ is also real.  We can repeat the argument we used above to show that $g_1\neq0$ in order to see that $k$ is a common root of $g_2$ and $g_3$.  Their greatest common divisor is
\[
(k^4-3 k^3-k^2+3 k+1)(k^4+3 k^3-k^2-3 k+1)(2 k^4+k^2+2),
\]
and we can immediately rule out the third factor as it has no real roots.  We claim that the other factors do not give solution curves either, so suppose for contradiction that $k$ is a root of one of them.

Using \eqref{lval} one can calculate that $l$ is antipodal to $k$ (i.e.~$l=-1/k$ since $k$ is real).  In other words, the point $P=[(kx+y)^{11}(x-ky)]$ is the pole of an axial disc of type $\xi_v$.  Moreover one can show that $k$ is a vertex of the icosahedron representing $p$ or that representing $q$, and hence this disc $v$ passes through one of these points---say it passes through $p$.  We shall show that $[\D\double{u}]$ has degree $2$ but agrees with the constant map $[\D\double{v}]$ at $p$, $P$ and $\tau(P)$, giving the contradiction we seek.

To see that $\deg[\D\double{u}]=2$ we simply compute that
\begin{equation}
\label{pqPReq}
\dim\lspan{T_pX_I, T_qX_I, P, \tau(P)} = 8,
\end{equation}
so the image of $\double{u}$ is not contained in a $3$-plane, and then apply the result of \protect \MakeUppercase {L}emma\nobreakspace \ref {labGrpDerProps}\ref{grpitm2}.  It is also easy to see that $[\D\double{u}]$ and $[\D\double{v}]$ agree at $P$, since both must represent the kernel of the infinitesimal action there, and similarly at $\tau(P)$.  We are left to show that they coincide at $p$.

Well, from \eqref{pqPReq} we see that there is a unique linear dependence between $P$, $\tau(P)$ and the elements of $T_pX_I$ and $T_qX_I$, and we know that this dependence is actually between $p$, $q$, $P$, $\tau(P)$, $[\D\double{u}]|_p \cdot p$ and $[\D\double{u}]|_q \cdot q$.  So if $p'$ and $q'$ are points in $T_pX_I$ and $T_qX_I$ respectively, such that $P$, $\tau(P)$, $p'$ and $q'$ are linearly dependent (and the coefficient of $p'$ in this dependence is non-zero), then $p'$ must lie in the tangent line $\lspan{p, [\D\double{u}]|_p \cdot p}$ to $\double{u}$ at $p$.  Since $P$, $\tau(P)$ and the tangent line $\lspan{p, [\D\double{v}]|_p \cdot p}$ to $\double{v}$ at $p$ lie in the $2$-plane containing $\double{v}$, we deduce that there is some point $p''$ in this tangent line such that $P$, $\tau(P)$ and $p''$ are linearly dependent.  By the preceding comment, this forces $p''$ to be in $\lspan{p, [\D\double{u}]|_p \cdot p}$.  It is easy to see that $p''$ is not equal to $p$ (since $p$, $P$ and $\tau(P)$ are linearly independent) so
\[
\lspan{p, [\D\double{u}]|_p \cdot p} = \lspan{p, p''} = \lspan{p, [\D\double{v}]|_p \cdot p}.
\]
We know, moreover, that $[\D\double{u}]|_p \cdot p$ and $[\D\double{v}]|_p \cdot p$ are orthogonal to $p$ since $p$ is in the zero set of the moment map \eqref{eqMomMap}.  Therefore $[\D\double{u}]$ must coincide with $[\D\double{v}]$ at $p$.  This completes the argument ruling out real values of $k$.

We are left to consider the case that $k$ is a root of $g_2$ of unit modulus.  In this situation one can check from \eqref{lval} that $l$ also has unit modulus, by computing $l \conj{l}-1$, substituting $1/k$ for $\conj{k}$, and showing that the numerator of the resulting rational expression in $k$ is divisible by $g_2$.  The reflection $\tau(P)$ of $P$ is thus at $[(-kx+y)^{11}(-lx+y)]$, and we can argue analogously to the construction of $g_3$ in the real case to see that $k$ must be a root of the greatest common divisor $g_4$ of the $3\times 3$ minors of a certain matrix.  Since $k$ is also a root of $g_2$, it must in fact be a root of
\[
g_5 \coloneqq \gcd(g_2, g_4) = 8 - 11k^6 + 8k^{12}.
\]
We shall see that the $12$ solutions are precisely the $k$-values of the twelve poles appearing across the six claimed non-axial discs.

Let us now rename the $k$-value of our pole $P$ to $k_P$ in order to distinguish it from the corresponding quantity $k_Q$ for the second pole $Q$ of our disc (which we have been ignoring so far).  Let the corresponding $l$-values be $l_P$ and $l_Q$, and let the reflections of $P$ and $Q$ be at $R$ and $S$ respectively.  Using $g_5=0$ to simplify \eqref{lval}, we have
\[
l_\bullet = \frac{k(64k_\bullet^6-221)}{183}.
\]
A calculation then shows that
\begin{equation}
\label{eqspan8}
\dim\lspan{T_pX_I, T_qX_I, P, R} = 8,
\end{equation}
so $\double{u}$ spans a $4$-plane and is thus an embedding (it's the rational normal curve in this $4$-plane).  In particular, the points $P$, $Q$, $R$ and $S$ are distinct.  Moreover, we can compute the unique linear dependence coming from \eqref{eqspan8} and deduce that
\begin{equation}
\label{eqDup}
[\D\double{u}]|_p = \left[\begin{pmatrix}
 2 \sqrt{5} k_P^2 (19 k_P^6-17) & 16 (k_P^6+28) \\
 -16 k_P^4 (20 k_P^6+11) & -2 \sqrt{5} k_P^2 (19 k_P^6-17) \\
\end{pmatrix}\right].
\end{equation}
The same formula clearly also holds with $k_Q$ in place of $k_P$.  We therefore have
\begin{equation}
\label{kPkQeqn}
\frac{k_P^2(19k_P^6-17)}{k_Q^2(19k_Q^6-17)} = \frac{k_P^6+28}{k_Q^6+28} = \frac{k_P^4(20k_P^6+11)}{k_Q^4(20k_Q^6+11)},
\end{equation}
noting that all denominators are non-zero since $k_Q^6$ is a root of $8-11z+8z^2$ and thus is irrational.

From this string of equalities we immediately see that if $k_P^6=k_Q^6$ then $k_P^2=k_Q^2$ and thus $k_P = \pm k_Q$.  This forces $P$ to coincide with $Q$ or $S$, which is impossible, so we conclude that $k_P^6 \neq k_Q^6$.  Without loss of generality (swapping $P$ and $Q$ if necessary) we may thus assume that
\[
k_P^6 = \frac{11+3\sqrt{15}i}{16} \text{ and } k_Q^6 = \frac{11-3\sqrt{15}i}{16}.
\]
Plugging these values into \eqref{kPkQeqn}, we get
\[
k_Q = \pm \frac{1-\sqrt{15}i}{4}k_P.
\]
We therefore see that for each choice of a sixth root of $(11+3\sqrt{15}i)/16$ for $k_P$ there are at most two possible values for $k_Q$.  Given one particular choice, all others are obtained by applying the $\pi/3$ rotational symmetry.

For questions of existence and uniqueness of curves $\double{u}$ realising each possible $(k_P, k_Q)$, and transversality of the relevant evaluation map, we need only consider one representative of each orbit under this symmetry.  From now on we may therefore fix a choice
\[
k_P = \frac{i\sqrt{1+\sqrt{15}i}}{2},
\]
with the corresponding choice
\[
k_Q = \pm\frac{i\sqrt{1-\sqrt{15}i}}{2}.
\]
We may parametrise $u$ so that the poles evaluating to $P$ and $Q$ occur at points $a \in (0, 1)$ and $-a$ in the domain respectively.  Note that the boundary marked points evaluating to $p$ and $q$ may not be at $\pm 1$ in this parametrisation.  The residues of $\D\double{u}$ at each pole can be computed using \protect \MakeUppercase {L}emma\nobreakspace \ref {labGrpDerRes}, and we can then assemble these to give an expression for $\D\double{u}$:
\[
\D\double{u} = \frac{\Res_a \D\double{u}}{z-a} + \frac{\Res_{1/a} \D\double{u}}{z-1/a} + \frac{\Res_{-a} \D\double{u}}{z+a} + \frac{\Res_{-1/a} \D\double{u}}{z+1/a}.
\]
To see that this expression is correct, note that the difference between the two sides is holomorphic on $\C$ and decays at $\infty$ (by performing a change of variables $z \leftrightarrow 1/z$) so is identically zero.

Substituting the value of $k_P$ into \eqref{eqDup} we get
\begin{equation}
\label{Dupeq}
[\D\double{u}]|_p = \left[\begin{pmatrix}
\sqrt{5}i & 16i \\
16i & -\sqrt{5}i
\end{pmatrix}\right].
\end{equation}
Similarly we have
\begin{equation}
\label{Duqeq}
[\D\double{u}]|_q = \left[\begin{pmatrix}
\sqrt{5}i & -16i \\
-16i & -\sqrt{5}i
\end{pmatrix}\right].
\end{equation}
We therefore see that the off-diagonal entries of $\D\double{u}$ coincide at two distinct points on $\pd D$, namely the marked points mapping to $p$ and $q$.  This immediately rules out the plus version of $k_Q$---for which the off-diagonal entries only agree at $0$---and shows that in the minus case (to which we now restrict our attention) the marked points are at $\pm i$ in some order.

Substituting $z=i$ into our expression for $\D\double{u}$ we obtain
\[
[\D\double{u}(i)] = \left[ \begin{pmatrix} 4\sqrt{2}ai & 5\sqrt{5}(a^2-1)i \\ 5\sqrt{5}(a^2-1)i & -4\sqrt{2}ai \end{pmatrix} \right].
\]
The top left-hand entry is positive imaginary, whilst the top right-hand entry is negative imaginary.  Comparing with \eqref{Dupeq} and \eqref{Duqeq}, this means that $i$ must evaluate to $q$ rather than $p$, and that
\[
\frac{5\sqrt{5}(1-a^2)}{4\sqrt{2}a} = \frac{16}{\sqrt{5}}.
\]
Since $a \in (0, 1)$, this gives $a=(9\sqrt{33}-32\sqrt{2})/25$.  Plugging this back into $\D\double{u}$, we obtain
\begin{equation}
\label{explDuval}
\D\double{u} = \frac{a(1+a^2)i}{450\sqrt{3}(z^2-a^2)(a^2z^2-1)} \begin{pmatrix} 9\sqrt{5}(1-z^2) & 16(\sqrt{55} + 18iz + \sqrt{55}z^2) \\ -16(\sqrt{55} - 18iz + \sqrt{55}z^2) & -9\sqrt{5}(1-z^2) \end{pmatrix}.
\end{equation}

Now that $a$ is determined, $\double{u}$ is a degree four rational curve whose value is known at six points (it maps $-i$, $i$, $a$, $-a$, $1/a$ and $-1/a$ to $p$, $q$, $P$, $Q$, $R$, and $S$ respectively), so it is uniquely determined if it exists.  We know that the six target points span a $4$-plane, and it straightforward to check that all proper subsets are linearly independent by lifting them to vectors $v_1, \dots, v_6$ in $\C^{13}$, computing the linear dependence $\sum \lambda_i v_i = 0$, and noting that each $\lambda_i$ is non-zero.  We can therefore explicitly write down the unique degree four rational curve $[U]$ in $\C\P^{12}$ with the required incidence properties, as the projectivisation of a holomorphic map $U \mc \C \rightarrow \C^{13} \setminus \{0\}$ given (up to scaling) by
\[
U(z) = \sum_i \lambda_iv_i \prod_{j\neq i} (z-a_j),
\]
where $a_1, \dots, a_6$ represent the six points in the domain $\C\P^1$ mapping to the known target points.  In our case we get
\[
U(z) = 
\begin{pmatrix} 29 ( 29 \sqrt{5} z^4-20 i \sqrt{11} z^3-338 \sqrt{5} z^2-20 i \sqrt{11} z+29 \sqrt{5})  \\ 928 \sqrt{55} ( z^4-1)  \\ 2552 ( z^2+1)  ( 3 \sqrt{5} z^2+10 i \sqrt{11} z+3 \sqrt{5})  \\ -1595 \sqrt{11} ( z^2-1)  ( \sqrt{5} z^2-18 i \sqrt{11} z+\sqrt{5})  \\ -38280 ( z^2+1)  ( \sqrt{5} z^2-2 i \sqrt{11} z+\sqrt{5})  \\ -15312 \sqrt{55} ( z^4-1)  \\ -319 \sqrt{5} ( 169 z^4-850 z^2+169)  \\ -15312 \sqrt{55} ( z^4-1)  \\ -38280 ( z^2+1)  ( \sqrt{5} z^2+2 i \sqrt{11} z+\sqrt{5})  \\ -1595 \sqrt{11} ( z^2-1)  ( \sqrt{5} z^2+18 i \sqrt{11} z+\sqrt{5})  \\ 2552 ( z^2+1)  ( 3 \sqrt{5} z^2-10 i \sqrt{11} z+3 \sqrt{5})  \\ 928 \sqrt{55} ( z^4-1)  \\ 29 ( 29 \sqrt{5} z^4+20 i \sqrt{11} z^3-338 \sqrt{5} z^2+20 i \sqrt{11} z+29 \sqrt{5}) \end{pmatrix}.
\]

To see that this really gives a solution for $\double{u}$, we just need to check that $[U]$ maps $\C$ into $X_I$ and the unit circle $\pd D$ into $L_I$.  Since $[U]$ sends $-i$ to $L_I$, it is enough to check that the derivative of $[U]$ is given at each point (or at least at each point of a dense subset) by the action of an element of $\mathfrak{sl}(2, \C)$ on $[U]$, and that this element can be taken to lie in $\mathfrak{su}(2)$ along $\pd D$.  But we have already computed exactly what this element must be---namely $\D\double{u}$---and it is clear from \eqref{explDuval} that this is proportional to an element of $\mathfrak{su}(2)$ along $\pd D$.  One can calculate by computer that
\[
U' = \D\double{u} \cdot U + \frac{4 z (625 z^2-4721)}{625 z^4-9442 z^2+625} U,
\]
and thus that $[U]'$ is indeed given by $\D\double{u} \cdot [U]$ on the dense set where the denominator in the coefficient of $U$ is non-zero.  This completes the proof that $[U]$ is a valid solution, and is the only one for this choice of $k_P$.  \protect \MakeUppercase {C}orollary\nobreakspace \ref {labInd4Trans} guarantees that this solution is a regular point of $\ev_2$.

The five other choices of $k_P$ give five other solutions, obtained from $[U]$ by rotations about a vertical axis through multiples of $\pi/3$.  This gives obvious orientation-preserving isomorphisms between their corresponding Riemann--Hilbert pairs and the kernels of the associated Cauchy--Riemann operators.  These isomorphisms commute with the evaluation map $\ev_2$ for rotations through even multiples of $\pi/3$, and intertwine it with the map swapping the two marked points for odd multiples.  In particular we see that all six discs are regular points of $\ev_2$, and that those differing by even multiples of $\pi/3$ count with the same sign.  Discs differing by odd multiples of $\pi/3$ can be obtained from each other by reflection, followed by rotation through an even multiple of $\pi/3$, both of which are orientation-preserving, and hence all six discs in fact carry the same sign, proving the lemma.
\end{proof}

Combining these with the axial discs mentioned at the start, we get an overall local degree of $\pm 8$ if the two families of discs---one comprising the two axial discs, the other comprising the six non-axials---count with the same sign, and $\pm 4$ otherwise.

\begin{alem}  The two families count with the same sign, and hence the local degree of $\ev_2$ is $\pm 8$.
\end{alem}
\begin{proof}
To show this we have to look at the relative orientations on the moduli spaces of discs.  The argument is fairly similar to the proof of \protect \MakeUppercase {L}emma\nobreakspace \ref {labCOSign}, and we fix the analogous orientation and spin structure on $L_I$ to the one used there on $L_\tri$, $L_T$ and $L_O$.  Again changing the orientation on the Lagrangian doesn't affect the relative sign we are interested in (only an overall sign on both families of discs, which cancels out in the final definition of the pearl complex differential), and there are no other spin structures to worry about as $H^1(L_I; \Z/2)=0$.

Suppose now that $(E, F)$ is a rank $3$ Riemann--Hilbert pair corresponding to a holomorphic disc $u$ with boundary on $L_I$, and that $(E, F)$ admits a splitting given by a holomorphic frame $v_1$, $v_2$, $v_3 = iu'$, with respect to which the partial indices are $1$, $1$ and $2$.  Suppose that we are also given boundary marked points $e^{i\theta}$ and $e^{i\phi}$ with $\theta < \phi < \theta + 2 \pi$; we think of these as outgoing and incoming respectively, playing the roles of $1$ and $-1$ in the usual picture of pearly trajectories.  Note that since we are working with upward Morse flows, rather than downward, our notion of incoming and outgoing should really be opposite to Biran--Cornea's for orientation purposes.  However, since we are only interested in \emph{relative} signs we ignore this issue.

Defining $f(z) = i(e^{i\theta/2}-e^{-i\theta/2}z)$ and $g(z) = i(e^{i\phi/2}-e^{-i\phi/2}z)$, a basis of the kernel of the Cauchy--Riemann operator is given by
\begin{equation}
\label{eqKerBas}
fv_1 \text{, } fv_2 \text{, } f^2v_3 \text{, } gv_1 \text{, } gv_2 \text{, } g^2v_3 \text{ and } fgv_3.
\end{equation}
Let $\delta \in \{\pm 1\}$ be the orientation sign of this basis with respect to our choice of orientation and spin structure on $L_I$, and let $R$ denote the one-dimensional space of infinitesimal reparametrisations of the disc fixing the two marked points.  The section $fgv_3$ spans $R$ and generates an automorphism of the disc which moves in the direction from $e^{i \theta}$ towards $e^{i \phi}$, so (still viewing $e^{i \theta}$ and $e^{i\phi}$ as outgoing and incoming) the conventions of \cite[Appendix A.1]{BCEG} mean that the basis
\begin{equation}
\label{eqQuotBas}
fv_1 \text{, } fv_2 \text{, } f^2v_3 \text{, } gv_1 \text{, } gv_2 \text{ and } g^2v_3
\end{equation}
of $(\ker \conj{\pd}) / R$ also carries orientation sign $\delta$.

For each $\psi$ we have a basis of $T_{e^{i\psi}}L_I$ given by
\[
e^{i\psi/2}v_1(e^{i\psi}) \text{, } e^{i\psi/2}v_2(e^{i\psi}) \text{ and } e^{i\psi}v_3(e^{i\psi}),
\]
and we can ask whether it is positively oriented.  This is unchanged under continuous variations of $\psi$ so is independent of the value of $e^{i\psi}$.  Let the orientation of this basis be $\eps \in \{\pm1\}$.  The infinitesimal evaluation map
\[
D \ev_2 \mc (\ker \conj{\pd}) / R \rightarrow T_{e^{i\phi}}L_I \oplus T_{e^{i\theta}}L_I
\]
sends the basis \eqref{eqQuotBas} to
\begin{multline*}
(f(e^{i \phi})v_1(e^{i\phi}), 0) \text{, } (f(e^{i \phi})v_2(e^{i\phi}), 0) \text{, } (f(e^{i \phi})^2v_3(e^{i\phi}), 0) \text{,}
\\ (0, g(e^{i\theta})v_1(e^{i\theta})) \text{, } (0, g(e^{i\theta})v_2(e^{i\theta})) \text{ and } (0, g(e^{i\theta})^2v_3(e^{i\theta})),
\end{multline*}
which is homotopic to
\begin{multline*}
(e^{i \phi/2}v_1(e^{i\phi}), 0) \text{, } (e^{i \phi/2}v_2(e^{i\phi}), 0) \text{, } (e^{i \phi}v_3(e^{i\phi}), 0) \text{,}
\\ (0, -e^{i\theta/2}v_1(e^{i\theta})) \text{, } (0, -e^{i\theta/2}v_2(e^{i\theta})) \text{ and } (0, e^{i\theta}v_3(e^{i\theta})).
\end{multline*}
The map thus carries orientation sign $\delta \eps^2 = \delta$.

Now let $u_\mathrm{a}$ be one of the two axial index $4$ discs through $p$ and $q$, parametrised so the pole is at $0$, with corresponding Riemann--Hilbert pair $(E_\mathrm{a}, F_\mathrm{a})$.  Similarly let $u_\mathrm{n}$ be the non-axial disc constructed in \protect \MakeUppercase {L}emma\nobreakspace \ref {ind4NonAxCount}, reparametrised so that the pole evaluating to $P$ is at $0$, with Riemann--Hilbert pair $(E_\mathrm{n}, F_\mathrm{n})$.  From the proofs of \protect \MakeUppercase {L}emma\nobreakspace \ref {labInd4AxPars} and \protect \MakeUppercase {L}emma\nobreakspace \ref {labInd4Pars} (whose notation we follow) we see that both of these Riemann--Hilbert pairs are of the form just considered.  Let $\delta_\mathrm{a}$ and $\delta_\mathrm{n}$ be the respective orientations signs.  Take the marked points $e^{i\theta_\bullet}$ and $e^{i \phi_\bullet}$ so that $u_\bullet(e^{i\theta_\bullet}) = p$ and $u_\bullet(e^{i\phi_\bullet}) = q$ for $\bullet$ equal to $\mathrm{a}$ and $\mathrm{n}$.  We are interested in the relative signs of the evaluation maps $D\ev_2$, i.e.~$\delta_\mathrm{a}$ versus $\delta_{\mathrm{n}}$.

There is an obvious isomorphism $h \mc (E_\mathrm{a}, F_\mathrm{a}) \rightarrow (E_\mathrm{n}, F_\mathrm{n})$, defined simply to preserve the holomorphic frames $v_j$ of the $E_\bullet$, which induces an isomorphism $H$ between the kernels of their Cauchy--Riemann operators.  Taking the basis for $\ker \conj{\pd}_\mathrm{a}$ corresponding to \eqref{eqKerBas}, applying $H$, and homotoping $\theta_\mathrm{a}$ to $\theta_\mathrm{n}$ and $\phi_\mathrm{a}$ to $\phi_\mathrm{n}$, we obtain the respective basis for $\ker \conj{\pd}_\mathrm{n}$.  We thus see that $H$ is orientation-preserving if and only if $\delta_\mathrm{a} = \delta_\mathrm{n}$.  In other words, the two families of discs count with the same sign if and only if $H$ is orientation-preserving.

Note that the basis $\alpha$, $\beta$, $\gamma$ of $\mathfrak{su}(2)$ appearing in the proof of \protect \MakeUppercase {L}emma\nobreakspace \ref {labInd4AxPars} is defined by the property that the kernel of the infinitesimal $\mathfrak{sl}(2, \C)$-action at the pole of $u_\mathrm{a}$ is spanned by $\alpha+i\beta$ and $\gamma$.  Assuming that $\alpha$, $\beta$, $\gamma$ is positively oriented as a basis of $\mathfrak{su}(2)$, the homotopy class of trivialisation of $F_\mathrm{a}$ induced by our orientation and spin structure on $L_I$ is tautologically represented by the frame $\alpha \cdot u_\mathrm{a}$, $\beta \cdot u_\mathrm{a}$, $\gamma \cdot u_\mathrm{a}$.  Under $h$ this frame is carried to $\xi_R \cdot u_\mathrm{n}$, $\xi_I \cdot u_\mathrm{n}$, $iz \D u_\mathrm{n} \cdot u_\mathrm{n}$, and since $iw\D u_\mathrm{n}(w)$ is linearly independent of $\xi_R$ and $\xi_I$ at each point $w \in \pd D$ we see that this frame is homotopic to $\xi_R \cdot u_\mathrm{n}$, $\xi_I \cdot u_\mathrm{n}$, $iw\D u_\mathrm{n}(w) \cdot u_\mathrm{n}$ for any given $w$.  We claim that the basis $\xi_I$, $\xi_R$, $iw \D u_\mathrm{n}(w)$ of $\mathfrak{su}(2)$ is positively oriented (meaning that it carries the same orientation as $\alpha$, $\beta$, $\gamma$), and hence the latter frame represents the homotopy class of trivialisation of $F_\mathrm{n}$ induced by our orientation and spin structure.  This in turn implies that $H$ is orientation-preserving, so the two families of discs count with the same sign.

To compute the relative orientations of the bases $\alpha$, $\beta$, $\gamma$ and $\xi_R$, $\xi_I$, $iw\D u_\mathrm{n}(w)$, first recall from the proof of \protect \MakeUppercase {L}emma\nobreakspace \ref {labInd4AxPars} that $\alpha$, $\beta$ and $\gamma$ represent infinitesimal rotations about a right-handed set of orthogonal axes, so up to an orientation-preserving transformation we may assume that they are
\begin{equation}
\label{abcBasis}
\begin{pmatrix} 0 & i \\ i & 0 \end{pmatrix} \text{, } \begin{pmatrix} 0 & -1 \\ 1 & 0 \end{pmatrix} \text{ and } \begin{pmatrix} i & 0 \\ 0 & -i \end{pmatrix}.
\end{equation}
The disc $u_\mathrm{n}$ coincides with the disc $u$ from \protect \MakeUppercase {L}emma\nobreakspace \ref {ind4NonAxCount} up to reparametrisation, so we have $u = u_\mathrm{n} \circ \phi$ for some biholomorphism $\phi \mc D \rightarrow D$.  Then $\D u(1) = \phi'(1) \D u_\mathrm{n} (\phi(1))$, and $\phi'(1)$ is a positive real multiple of $\phi(1)$, so for $w$ equal to $\phi(1)$ the expression $iw\D u_\mathrm{n}(w)$ is positively proportional to the value of $i\D u(1)$ computed using \eqref{explDuval}.  Up to a positive real scale factor (which is irrelevant), we deduce that $i\phi(1)\D u_\mathrm{n}(\phi(1))$ is given by
\[
\begin{pmatrix} 0 & 9i+\sqrt{55} \\ 9i-\sqrt{55} & 0 \end{pmatrix}.
\]
Meanwhile $\xi_R$ and $\xi_I$ can be computed as the real and imaginary parts of the residue at $P$ from \protect \MakeUppercase {L}emma\nobreakspace \ref {ind4NonAxCount} (with respect to our usual real and imaginary splitting of $\mathfrak{sl}(2, \C)$).  The results, again ignoring positive real scalars, are
\[
\xi_R = \begin{pmatrix} -i & 0 \\ 0 & i \end{pmatrix} \text{ and } \xi_I = \begin{pmatrix} 0 & -\sqrt{1785+825 \sqrt{15}i} \\ \sqrt{1785-825 \sqrt{15}i} & 0 \end{pmatrix}.
\]
In terms of the basis \eqref{abcBasis}, the new basis $\xi_R$, $\xi_I$, $i\phi(1)\D u_\mathrm{n}(\phi(1))$ is given by
\[
\begin{pmatrix} 0 \\ 0 \\ -1 \end{pmatrix} \text{, } \begin{pmatrix} \im \sqrt{1785-825 \sqrt{15}i} \\ \re \sqrt{1785-825 \sqrt{15}i} \\ 0 \end{pmatrix}  \text{ and } \begin{pmatrix} 9 \\ -\sqrt{55} \\ 0 \end{pmatrix}.
\]
The matrix with these three columns has positive determinant, so the change of basis is orientation-preserving, completing the proof that both families of discs count with the same sign.
\end{proof}

Combining the two preceding lemmas proves \protect \MakeUppercase {P}roposition\nobreakspace \ref {labInd4Ipq}.

Next we verify the absence of bubbled configurations:

\begin{alem}\label{INoBub}  There are no index $4$ bubbled configurations through $p$ and $q$.
\end{alem}
\begin{proof}  To show this we calculate the explicit values in $\C\P^1$ of the vertices of $p$ and $q$, and use these to compute the sets \eqref{eqpqSet} and \eqref{eqCSet} from Section\nobreakspace \ref {sscBubConf}.  They are
\[
\left\{ \frac{1}{3} - \frac{2}{3\sqrt{5}}, \frac{1}{6}+\frac{1}{6\sqrt{5}}, \frac{1}{2} \pm \frac{\sqrt{5}}{6}, \frac{1}{2} \pm \frac{1}{6\sqrt{5}}, \frac{5}{6} - \frac{1}{6\sqrt{5}}, \frac{2}{3} + \frac{2}{3\sqrt{5}} \right\}
\]
and
\[
\left\{ 0, \frac{1}{2} \pm \frac{1}{2\sqrt{5}}, 1 \right\}
\]
respectively.  Since $1$ and $\sqrt{5}$ are linearly independent over $\mathbb{Q}$, these are disjoint.
\end{proof}

\section{Explicit representatives of the configurations}\label{secExplReps}

Here we collect together explicit expressions in standard coordinates for the triangle, tetrahedron, octahedron and icosahedron in each of three positions, depending on what feature is pointing vertically upwards: a vertex, the mid-point of an edge, or the centre of a face.  We denote these configurations by $C_v$, $C_e$ and $C_f$ respectively.  To remove any ambiguity regarding rotations about a vertical axis, for the edge (respectively face) case we take one end of the top edge (respectively one vertex of the top face) to lie on the positive real axis.  In the case of an upward-pointing vertex, we take one of next-northernmost vertices to lie on the positive real axis.  With these conventions, we have:

\begin{gather*}
\tri_v = [1: 0: -3: 0]
\\ \tri_e = [0: -3: 0: 1]
\\ \tri_f = [1: 0: 0: 1]
\end{gather*}

\begin{gather*}
T_v = [1: 0 : 0: 2\sqrt{2}: 0]
\\ T_e = [1: 0: 2\sqrt{3}: 0: -1]
\\ T_f = [0: 2\sqrt{2}: 0: 0: 1]
\end{gather*}

\begin{gather*}
O_v = [0: 1: 0: 0: 0: -1: 0]
\\ O_e = [1: 0: -5: 0: -5: 0: 1]
\\ O_f = [1: 0: 0: -5\sqrt{2}: 0: 0: -1]
\end{gather*}

\begin{gather*}
I_v = [0: 1: 0: 0: 0: 0: -11: 0: 0: 0: 0: -1: 0]
\\ I_e = [\sqrt{5}: 0: -22: 0: -33\sqrt{5}: 0: 44: 0: -33\sqrt{5}: 0: -22: 0: \sqrt{5}]
\\ I_f = [1: 0: 0: -11\sqrt{5}: 0: 0: -33: 0: 0: 11\sqrt{5}: 0: 0: 1].
\end{gather*}

These can be computed as follows, recalling that $r_C$ is the number of faces meeting at a vertex of $C$, equal to $2$, $3$, $4$ and $5$ for $C$ equal to $\tri$, $T$, $O$ and $I$ respectively.  Working from north to south on the sphere, the vertices of the configuration $C_v$ are at $\infty$, then at the vertices of a horizontal regular $r_C$-gon containing a point $a$ in $(0, \infty)$.  These contribute factors $x$ and
\[
\prod_{j=1}^n (e^{2j\pi i/r_C}ax+y) = (-1)^{r_C+1}a^{r_C}x^{r_C} + y^{r_C}
\]
respectively.  For the triangle and tetrahedron there are no other vertices, whilst the octahedron also contains $0$ and the icosahedron contains $0$ and the horizontal regular $r_C$-gon through $-1/a$.

Applying the cosine angle formula \eqref{eqAngle} to the points $\infty$, $a$ and then $a$, $e^{2\pi i/r_C}a$ we see that
\begin{equation}
\label{Angles}
\frac{a}{\sqrt{1+a^2}} = \frac{|1+e^{2\pi i/r_C}a^2|}{1+a^2}.
\end{equation}
It is easy to compute from this that
\[
a = \frac{1}{\sqrt{1-\cos(2\pi / r_C)}},
\]
and hence that $a$ is $1/\sqrt{3}$, $1/\sqrt{2}$, $1$ and $(1+\sqrt{5})/2$ in the four cases.  Plugging into $\tri_v = [x(-a^2x^2+y^2)]$, $T_v = [x(a^3x^3+y^3)]$, $O_v = [x(-a^4x^4+y^4)y]$ and $I_v = [x(a^5x^5+y^5)(-x^5/a^5+y^5)y]$ gives the claimed expressions for each $C_v$.

To compute $C_e$ and $C_f$ note that each of these configurations can be obtained by rotating $C_v$, and these rotations can be performed easily on a computer.  Explicitly, rotating $C_v$ through angle $\theta_e$ (right-handed) about the axis from $i$ to $-i$ gives $C_e$, where $\theta_e$ is the angle between the horizontal and one of the edges emanating from the north pole.  This is half the angle between $\infty$ and $a$, so its cosine is exactly given by the left-hand side of \eqref{Angles}.

In order to get $C_f$ from $C_v$, we first rotate through angle $\pi/r_C$ about the vertical axis so that the vertex at $a$ becomes one end of an edge which is parallel to the $y$-axis.  We then rotate through angle $\theta_f$ about the axis from $i$ to $-i$, where $\theta_f$ is the angle between the horizontal and a face meeting the north pole, and finally through angle $\pi$ about the vertical axis in order to make one of the vertices of the top face lie on the positive real axis.  Simple trigonometry then gives
\[
\sin \theta_f = \frac{2}{\sqrt{3}}\sin \theta_e = \frac{2}{\sqrt{3(1+a^2)}}.
\]

Using a computer to carry out these rotations and simplify, we obtain the claimed expressions for $C_e$ and $C_f$.

\bibliography{biblio}
\bibliographystyle{utcapsor2}

\end{document}